\documentclass[11pt,a4paper,leqno, headinclude, final]{amsart}

\usepackage[left=4cm,right=4cm,top=3cm,bottom=3cm]{geometry}
\usepackage{ifdraft}
\usepackage{pdfcolfoot}
\ifoptionfinal{
  \newcommand{\COMMENT}[1]{}
}{
  \newcommand{\COMMENT}[1]{{\color{gray}\(\otimes\)}\footnote{{\sffamily\color{gray} #1}}}
}
\usepackage[dvipsnames,usenames,fixpdftex]{xcolor}

\usepackage[T1]{fontenc}
\usepackage[utf8]{inputenc}
\usepackage{lmodern}

\usepackage[tracking=true]{microtype}
\DeclareMicrotypeSet*[tracking]{my}%
  { font = */*/*/sc/* }%
\SetTracking{ encoding = *, shape = sc }{ 45 }

\title[Vector bundles over cohomogeneity one manifolds]{Vector bundles of non-negative curvature over cohomogeneity one manifolds}
\def\titl{Vector bundles of non-negative curvature over cohomogeneity one manifolds}
\def\auth{Manuel Amann \& David Gonz\'alez-\'Alvaro \\ \& Marcus Zibrowius}
\author{\auth}
\date{\today}
\subjclass[2010]{
  53C20, % Global Riemannian geometry, including pinching
  57S15, % Compact Lie groups of differentiable transformations  
  19M05  % Miscellaneous applications of K-theory
  (Primary),
  55N91, % Equivariant homology and cohomology
  57T15, % Homology and cohomology of homogeneous spaces of Lie groups
  55P62, % Rational homotopy theory
  19L47  % Equivariant (topological) K-theory
  (Secondary)}
\keywords{\noindent Non-negative sectional curvature, vector bundle,  homogeneous space, cohomogeneity one manifold, equivariant K-theory, equivariant cohomology, rational homotopy theory}

\usepackage[british]{babel}

\usepackage[hidelinks,pdftex]{hyperref}
\hypersetup{
  pdftitle=\titl,
  pdfauthor=\auth,
  pdftoolbar=false,
  plainpages=false,
  pdfdisplaydoctitle=true
}

\usepackage{mathtools, amssymb, amscd, amsthm}
\usepackage{stmaryrd}
\usepackage{latexsym}
\usepackage[all]{xy}
\usepackage{pb-diagram}
\usepackage{rotating}
\usepackage{multicol}
\usepackage{lscape}
\usepackage{wasysym}
\usepackage{longtable}
\usepackage{enumerate,paralist}
\usepackage[nameinlink,capitalize]{cleveref}
\xyoption{all}

\theoremstyle{plain}
\newtheorem{theo}{Theorem}[section]\crefname{theo}{Theorem}{Theorems}

\newtheorem{main}{Theorem}
\newtheorem*{main*}{Theorem}

\newtheorem{prop}[theo]{Proposition}\crefname{prop}{Proposition}{Propositions}
\crefname{defn}{Definition}{Definitions}
\newtheorem{lemma}[theo]{Lemma}\crefname{lem}{Lemma}{Lemmas}
\crefname{note}{Note}{Notes}
\newtheorem{cor}[theo]{Corollary}\crefname{cor}{Corollary}{Corollaries}
\newtheorem*{cor*}{Corollary}
\newtheorem{ques}[theo]{Question}
\newtheorem*{theo*}{Theorem}
\newtheorem*{ques*}{Question}
\newtheorem*{prop*}{Property (S)}

\theoremstyle{remark}
\newtheorem{rem}[theo]{Remark}\crefname{rem}{Remark}{Remarks}
\newtheorem{introrem}{Remark}\crefname{introrem}{Remark}{Remarks}
\newtheorem{warning}[theo]{Warning}\crefname{warning}{Warning}{Warnings}
\newtheorem{example}[theo]{Example}\crefname{eg}{Example}{Examples}
\newtheorem*{example*}{Example}
\crefname{egs}{Examples}{Examples}
\newtheorem*{examples*}{Examples}

\newcommand*{\define}[1]{\textbf{#1}}  % emphasize concepts that are defined

\renewcommand{\H}{\operatorname{H}}                                % cohomology
\newcommand{\K}{\operatorname{K}}                                  % K-theory (generic notation)
\newcommand{\KO}{\operatorname{KO}}                                % Unitary K-theory
\newcommand{\R}{\operatorname{R}}                                  % Representation ring (generic notation)
\newcommand{\I}{\operatorname{I}}                                  % augmentation ideal
\newcommand{\chern}{\operatorname{ch}}                             % Chern character
\newcommand{\B}{{\operatorname{B}}}                                % classifying space
\newcommand{\E}{{\operatorname{E}}}                                % contractible space with free action
\newcommand{\rank}{\operatorname{rk}}                              % rank

\newcommand{\zz}{{\mathbb{Z}}}                                     % integers
\newcommand{\qq}{{\mathbb{Q}}}                                     % rational numbers
\newcommand{\rr}{{\mathbb{R}}}                                     % real numbers
\newcommand{\cc}{{\mathbb{C}}}                                     % complex numbers

\newcommand{\point}{\mathrm{pt}}                                   % point
\newcommand{\s}{{\mathbb{S}}}                                      % sphere

\newcommand{\dif} {{\operatorname{d}}}                             % differential operator d
\newcommand{\In} {{\,\subseteq\,}}                                 % subset
\newcommand{\im} {{\operatorname{im\,}}}                           % image
\newcommand{\ch}{{\operatorname{ch}}}                              % Chern-character
\newcommand{\Hom}{{\operatorname{Hom}}}                            % homomorphisms
\newcommand{\APL}{{\operatorname{A_{PL}}}}                         % polynomial differential forms

\newcommand{\xto}[1]{\xrightarrow{#1}}                             % abbreviation for \xrightarrow
\newcommand{\hto}[1]{\overset{#1}{\hookrightarrow}}                % abbreviation for labelled \hookrightarrow
\newcommand{\biq}[2]{#1\;\!\!\!\sslash \;\!\!\!#2}                 % creates a biquotient

\newcommand{\odd}{\mathrm{odd}}                                    % odd
\newcommand{\even}{\mathrm{even}}                                  % even
\newcommand{\borel}{\mathrm{Borel}}

\newcommand{\rk}{\operatorname{rk}}

\numberwithin{equation}{section}
\let\realequation\equation
\def\equation{\setcounter{equation}{\arabic{theo}}%
   \refstepcounter{theo}%
   \realequation}
\let\realalign\align
\def\align{\setcounter{equation}{\arabic{theo}}%
   \refstepcounter{theo}%
   \realalign}

\begin{document}
\setcounter{section}{0}
\maketitle \thispagestyle{empty}

%%%%%%%%%%%%%%%%%%%%%%%%%%%%%%%%%% Abstract%%%%%%%%%%%%%%%%%%%%%%%%%%%%%%%%%%%%%%%%

\begin{abstract}
We provide several results on the existence of metrics of non-negative sectional curvature on vector bundles over certain cohomogeneity one manifolds and homogeneous spaces up to suitable stabilization.

Beside explicit constructions of the metrics, this is achieved by identifying equivariant structures upon these vector bundles via a comparison of their equivariant and non-equivariant K-theory. For this, in particular, we transcribe equivariant K-theory to equivariant rational cohomology and investigate surjectivity properties of induced maps in the Borel fibration via rational homotopy theory.
\end{abstract}

%%%%%%%%%%%%%%%%%%%%%%%%%%%%%%%%%% Introduction %%%%%%%%%%%%%%%%%%%%%%%%%%%%%%%%%%%

\COMMENT{\textbf{Conventions}

  \noindent
  We use ``G-vector bundle'' instead of $G$-equivariant vector bundle etc.

  \noindent
  We use Oxford spelling, i.e.:
  \begin{compactitem}[-]
  \item
    generalize, formalize, realize, characterize, characterization, etc.      (as in US)

  \item
    analyse, neighbourhood, fibre    (as in generic British)
  \end{compactitem}
  When in doubt, check \url{https://www.oed.com/}.
}

\section*{Introduction}

The Soul Theorem of Cheeger and Gromoll \cite{CheGro72} determines the structure of an open Riemannian manifold of non-negative (sectional) curvature. For such a manifold $X$ there exists a totally convex and totally geodesic closed submanifold $M\subset X$ such that $X$ is diffeomorphic to the total space of the normal bundle of $M$ in $X$.

Every closed manifold $M$ of non-negative curvature can be realized as a soul: just take the Riemannian product $M\times\rr^k$ for any $k\geq 1$, where $\rr^k$ is endowed with the Euclidean flat metric. Since $M\times\rr^k$ is a trivial vector bundle over $M$, it is natural to ask what happens with non-trivial ones.

\begin{ques*}[Converse to the Soul Theorem]
Let $M$ be a closed manifold of non-negative curvature, and let $E$ be a vector bundle over $M$. Does (the total space of) $E$ admit a metric of non-negative curvature?
\end{ques*}

There exist examples of vector bundles over base spaces with infinite fundamental group where the total space admits no metric of non-negative curvature \cite{OzWa94,BelKap01,BelKap03}. However, no obstructions are known in the case of finite fundamental group. Moreover, all real vector bundles over $\mathbb{S}^n$, with $2\leq n\leq 5$, admit non-negative curvature \cite{GroveZiller:Milnor}. For higher dimensional spheres, there is the following result of Rigas \cite{Rig78}: for every real vector bundle $E\to\mathbb{S}^n$, $n$ arbitrary, there is some $k$ such that $E\times\rr^k$ admits a metric of non-negative curvature.  This is one of the results that motivate the present article. As we shall discuss below, Rigas' statement (or even stronger versions of it) have already been shown to hold when one replaces the base space $\mathbb{S}^n$ by certain classes of homogeneous spaces \cite{Gon17,GonZib1}, certain classes of biquotients \cite{GonZib2}, or by several $4$-dimensional cohomogeneity one spaces \cite{GroveZiller:Lifting}. The main purpose of this article is to prove \Cref{THM: big theorem} below, which provides extensions and variations of Rigas' result in the class of cohomogeneity one spaces of arbitrary dimension.

Recall that a manifold $M$ with an action of a compact Lie group $G$ is said to be of \define{cohomogeneity one} if the orbit space $M/G$ is of dimension one (and hence is a circle or an interval). In the most interesting case where $M/G$ is an interval we use the standard notation \((G,H,K_-,K_+)\) for the associated group diagram, i.e. $H$ denotes the isotropy group at a regular orbit and $K_\pm$ the isotropy groups of the orbits projecting to the endpoints of $M/G$ (see \Cref{sec:prelim_cohomo1} for details). Recall that the rank of a compact connected Lie group $G$, denoted by $\rank G$, is defined as the dimension of any of its maximal tori. Note that throughout the entire article ``non-negative curvature’’ will always refer to non-negative sectional curvature, and all the metrics are assumed to be complete.

\begin{main}\label{THM: big theorem}
  Let \(M\) be a closed connected manifold with an action by a compact connected Lie group $G$ with orbit space an interval, and let \((G,H,K_-,K_+)\) be the associated group diagram. Assume that the groups $H, K_\pm$ are connected and that the singular orbits have codimension $2$ in $M$ (i.e.\ $K_\pm/H \cong S^1$).
  \begin{enumerate}
  \item \label{THM: converse soul for cohomo 1}
    Suppose \(\rank G = \rank K_\pm\), and suppose moreover \(\pi_1(G)\) is torsion-free and the groups \(K_\pm\) satisfy Steinberg's assumptions, as detailed in \eqref{eq:Steinbergs-assumptions} below.  (These are satisfied for instance when \(K_+\) and \(K_-\) are tori or when their semisimple parts are simply-connected.)
    Then, for every complex vector bundle \(E\) over \(M\) there is an integer \(k\) such that the product manifold \(E\times\rr^k\) carries a metric of non-negative curvature.

  \item \label{THM: rational converse soul for cohomo 1}
    Suppose \(\rank G - \rank K_\pm\leq 1\). Then, for every real vector bundle \(E\) over \(M\) there are integers \(q>0\) and \(k\) such that the product manifold \(qE\times\rr^k\) carries a metric of non-negative curvature. (Here \(qE\) denotes the Whitney sum of \(q\) copies of \(E\).)
\end{enumerate}
\end{main}

The case where the orbit space is a circle is discussed in \Cref{SEC:orbit_space_circle}. Before putting \Cref{THM: big theorem} into perspective, let us discuss to which classes of manifolds it applies.

\begin{introrem}\label{EX:cohomo_1_codim_2_orbits}
First of all, the class of cohomogeneity one manifolds with singular orbits of codimension $2$ turns out to be fairly rich. It includes (without specifying the corresponding action): $\mathbb{S}^n$ with $n=2,3,4,5,7$, the $4$-manifolds $\mathbb{S}^2\times\mathbb{S}^2$, $\cc P^2$, $\cc P^2\sharp \overline{\cc P}^2$ (where $\overline{\cc P}^2$ denotes $\cc P^2$ with opposite orientation), every homotopy $\rr P^5$, every $SO(3)$- and every $SO(4)$-principal bundle over $\mathbb{S}^4$ \cite{GroveZiller:Milnor}, and an infinite family of $10$-dimensional manifolds which admit free $SU(2)$-actions producing very interesting classes of non-negatively curved $7$-manifolds \cite{GoKeSh17}. Moreover, this class is closed under taking products with homogeneous manifolds, endowed with the obvious product action.

Part~(1) of \Cref{THM: big theorem} only applies to few of these manifolds. Indeed, as explained in more detail in \Cref{REM:maximal_rank_iff_positive_characteristic} below, our assumptions imply that the manifolds in question have positive Euler characteristic. However, the set of manifolds to which Part~(1) applies is far from empty. For example, as we will explain in \Cref{SS:DeVito_example}, in every even dimension $\geq 4$ there exists a cohomogeneity one space, namely $(\cc P^2\sharp \overline{\cc P}^2 )\times\left(\mathbb{S}^2\right)^n$ with $n\geq 0$, to which Part~(1) of \Cref{THM: big theorem} applies (see \Cref{PROP:THMA_applies_to_M_2n4}), and this space is not even homotopy equivalent to a homogeneous space (see \Cref{PROP:M_2n4_isnot_homogeneous}). This example was communicated to us by Jason DeVito.

Part~(2) of \Cref{THM: big theorem}, on the other hand, also applies to certain manifolds with vanishing Euler characteristic. If a space $M$ as in Part~(2) additionally satisfies $\oplus_{i>0}\H^{4i}(M,\qq)\neq 0$, then the theorem ensures the existence of infinitely many real vector bundles over $M$ which admit non-negatively curved metrics and belong to pairwise distinct stable classes, see \Cref{REM:stable_classes}.
\end{introrem}

Let us now provide some historical and mathematical context. Cohomogeneity one manifolds are exactly one dimension more complicated than homogeneous manifolds, in the sense that they carry a group action whose orbit space is of dimension one rather than dimension zero. \Cref{THM: big theorem} is, on the one hand, a generalization of existence results known in the homogeneous case to cohomogeneity one manifolds, bringing together different tools from the literature. On the other hand, it provides new charaterizations via techniques which have not yet been applied in this context, and which also provide new insight in the homogeneous situation (see \Cref{THM:surjectivity of homogeneous bundles} below).  We hope that this main theorem and the subsequent theorems below will help complete the entire picture for actions of cohomogeneity at most one.

\medskip

Recall first that every closed homogeneous space $G/H$ admits a $G$-invariant metric of non-negative curvature, and that the same holds true for the total space of any real $G$-vector bundle over $G/H$ (see \Cref{SUBSEC:G-vector-bundles} for the definition of $G$-vector bundle). In contrast to the homogeneous case, cohomogeneity one manifolds do not fit so well into the world of non-negative curvature. In particular, there exist many examples without invariant metrics of non-negative curvature \cite{GrVeWiZi06}. However, the special class that we consider here is known to fit in very nicely: Grove and Ziller showed in \cite{GroveZiller:Milnor} that any cohomogeneity one $G$-manifold with non-principal orbits of codimension $\leq 2$ does admit an invariant metric of non-negative curvature, and it follows that any real $G$-vector bundle over such a $G$-space admits a $G$-invariant metric of non-negative curvature (see \Cref{hop:non-negative-metrics} in \Cref{SEC:G_vector_bundles_nonnegative_curvature_cohomo1}). In view of these results it is natural to ask:

\begin{ques*}
Given a manifold with a $G$-action, ``how many'' vector bundles are (isomorphic to the underlying vector bundles of) $G$-vector bundles?
\end{ques*}
 %by a Lie group $G$

There are some positive general results; for instance the tangent bundle is a $G$-vector bundle \cite[p.~303]{Bre72}. Moreover, there are some $G$-manifolds for which all (real) vector bundles are $G$-vector bundles: the transitive $SU(2)$-action on $\mathbb{S}^2 = SU(2)/S^1$, any action on $\mathbb{S}^3$ since all vector bundles are trivial, or the cohomogeneity one $SU(2)$-action on $\mathbb{S}^4$ (as shown by Grove and Ziller in \cite{GroveZiller:Milnor}).

However, vector bundles do not admit $G$-vector bundle structures in general. %in general a vector bundle over a manifold on which a Lie group $G$ acts does not admit a $G$-vector bundle structure.
It is therefore remarkable that, under certain (strong) symmetry assumptions on $M$, every complex or real vector bundle carries, up to stabilization, a $G$-vector bundle structure. That is, the following property holds:

\begin{prop*} For every complex or real vector bundle $E$ over $M$ there is some integer $k$ such that the Whitney sum $E\oplus\cc^k$ or $E\oplus \rr^k$ carries a \(G\)-vector bundle structure.
\end{prop*}

When all $G$-vector bundles over a given $M$ carry non-negatively curved metrics, Property (S) has the following immediate implication: for every complex or real vector bundle $E$ over $M$ there is some integer $k$ such that $E\times\rr^{2k}$ or $E\times\rr^k$ carries a metric of non-negative curvature.

\smallskip

A key feature of Property (S) is that it can be characterized in terms of the (equivariant) K-theory of the manifold. More precisely, Property (S) holds if and only if the map $\K_G^0(M)\to \K^0(M)$ or $\KO_G^0(M)\to \KO^0(M)$ induced by forgetting the $G$-vector bundle structure is surjective (see \Cref{surjectivity-of-forgetful-map-in-Kgroups}). Here $\K^0(M)$ and $\K_G^0(M)$ denote the usual and the $G$-equivariant K-theory ring of complex bundles, respectively, and the notation $\KO^0$ indicates K-rings of real bundles. These rings extend to generalized cohomology theories $\K^*(M)$ and $\K^*_G(M)$ and there is an induced forgetful map $\K_G^*(M)\to \K^*(M)$. Surjectivity of this map in all degrees evidently implies surjectivity of the forgetful map $\K_G^0(M)\to \K^0(M)$ in degree zero. (See \Cref{SUBSEC:genuine-equivariant-K-theory} for definitions and details.)  Let us review existing results concerning the surjectivity of these maps.

The surjectivity of $\K_G^*(M)\to \K^*(M)$ has been studied both in the homogeneous and in the cohomogeneity one case. Suppose that $G$ is connected with torsion-free fundamental group.
 Pittie showed that $\K_G^*(G/H)\to \K^*(G/H)$ is surjective if $\rank G=\rank H$ \cite{Pit72}. In recent work, Carlson \cite{Carlson:coho1} studied the K-theory of cohomogeneity one spaces \(M\) and showed that $\K_G^*(M)\to \K^*(M)$ is surjective if the groups in the associated diagram  \((G,H,K_-,K_+)\) are connected, \(\rank G = \max\{\rank K_+,\rank K_-\}\) and \(K_+\) and \(K_-\) satisfy Steinberg's assumptions \eqref{eq:Steinbergs-assumptions}. Building upon this work, in \Cref{carlson-surjectivity} we prove that Property~(S) holds for complex bundles over those manifolds which satisfy the prerequisites of Carlson's results. This, taken together with the results above concerning the existence of non-negatively curved metrics, implies Part~(1) of \Cref{THM: big theorem} (see \Cref{SUBSEC:stabilization-over-cohomogeneity-one-spaces} for details).

 %     \item the $G$-action is of cohomogeneity one with orbit space an interval, with associated group diagram \((G,H,K_-,K_+)\), and the groups \(K_\pm\) satisfy Steinberg's assumptions \eqref{eq:Steinbergs-assumptions}.
 % \end{enumerate}
 %   \item $\rank G =\rank H$, where $H\subset G$ denotes the principal isotropy group,
  %  \item \(\rank G = \max\{\rank K_+,\rank K_-\}\)
 %  Then every complex vector bundle over $M$ carries a $G$-vector bundle structure up to stabilization, i.e.\ for every complex vector bundle \(E\) over \(M\) there is an integer \(k\) such that the Whitney sum \(E\oplus\cc^k\) carries a $G$-vector bundle structure.

\smallskip

Motivated by Rigas' result mentioned above, it was observed in \cite{GonZib1} that Pittie's hypothesis on the rank can be relaxed while keeping the surjectivity of $\K_G^0(G/H)\to \K^0(G/H)$, which is enough for our purposes. More precisely, it was shown that $\K_G^0(G/H)\to \K^0(G/H)$ is surjective if $\rank G -\rank H\leq 1$ \cite[Theorem 3.6]{GonZib1}. In the present article we prove that this implication is an equivalence by studying the induced map $\K_G^0(G/H)\otimes\qq\to \K^0(G/H)\otimes\qq$ (see the discussion prior to \Cref{THM:double_stabilization_homogeneous}), which leads us to the following characterization of this class of homogeneous spaces.

\begin{main}\label{THM:surjectivity of homogeneous bundles}
  Let $G/H$ be a closed homogeneous space,  where $G$ is connected with torsion-free fundamental group, and $H\subset G$ is a closed connected subgroup. Then every complex vector bundle over $G/H$ carries a $G$-vector bundle structure up to stabilization if and only if $\rank G -\rank H\leq 1$.
\end{main}

The study of \emph{real} vector bundles is quite different and significantly more complicated. In particular, the condition $\rank G -\rank H\leq 1$ is neither neccesary nor sufficient for the surjectivity of $\KO_G^0(G/H)\to \KO^0(G/H)$. There exist examples with $\rank G -\rank H=0$ for which it is not surjective  (e.g.\ $SU(2)^4/T^4$, which is diffeomorphic to $\mathbb{S}^2\times\mathbb{S}^2 \times\mathbb{S}^2\times\mathbb{S}^2$), as well as examples with \(\rank G - \rank H\) equal to $1$, $2$ or $3$ where it is surjective (e.g.\ any $G/H$ diffeomorphic to $\mathbb{S}^7$, $\mathbb{S}^7\times\mathbb{S}^7$ or $\mathbb{S}^7\times\mathbb{S}^7\times\mathbb{S}^7$, since $\KO^0(G/H)$ is trivial). We refer to \cite{GonZib1} for these and other results. In the cohomogeneity one case, we are not aware of any existing work on the map $\KO_G^0(M)\to \KO^0(M)$. In this article we are able to provide results for real bundles both for homogeneous and cohomogeneity one spaces, however we obtain a weaker conclusion; see Theorems \ref{THM:double_stabilization_homogeneous} and \ref{THM:rational_stabilization_cohomo1} below.

\smallskip

Going back to the complex case, and as done in the homogeneous setting, we would like to relax Carlson's hypothesis \(\rank G = \max\{\rank K_+,\rank K_-\}\) while keeping the surjectivity of $\K_G^0(M)\to \K^0(M)$, see \Cref{REM:carlson_not_necessary}. Unfortunately, for cohomogeneity one spaces this is a less straight-forward task. In order to obtain further results we therefore pass to the rational setting and study the induced map $\K^0_G(M)\otimes\qq\to \K^0(M)\otimes\qq$. The surjectivity of this map also has a characterization in terms of vector bundles, although in this case it is weaker than Property~(S) and one needs a double stabilization (see \Cref{rational-surjectivity-of-forgetful-map-in-Kgroups}). This characterization has one advantage, though: it applies not only to complex bundles but also to real ones. This is because the realification map $\K^0(M)\otimes\qq\to \KO^0(M)\otimes\qq$ is surjective for any closed manifold $M$ (see \Cref{prop:rational_realification_ir_surjective}). In this work we characterize the surjectivity of $\K^0_G(M)\otimes\qq\to \K^0(M)\otimes\qq$ for some subclasses of homogeneous and cohomogeneity one spaces.

In the case of homogeneous spaces $G/H$ with $G,H$ connected we show that $\K^0_G(G/H)\otimes\qq\to \K^0(G/H)\otimes\qq$ is surjective if and only if $\rank G -\rank H\leq 1$, see \Cref{THM:half_equiv_formal_homogeneous}. This has two immediate consequences: one implication of Theorem~\ref{THM:surjectivity of homogeneous bundles} (namely, the ``only if'' direction), and the following result for the stabilization of \emph{real} vector bundles.

\begin{main}\label{THM:double_stabilization_homogeneous}
Let $G/H$ be a closed homogeneous space, with $G,H$ compact and connected and satisfying $\rank G- \rank H\leq 1$. Then, for every real vector bundle \(E\) over \(G/H\) there are integers \(q>0\) and \(k\) such that the Whitney sum \(qE\oplus \rr^k\) carries a $G$-vector bundle structure and hence the product manifold \(qE\times\rr^k\) admits a metric of non-negative curvature.
\end{main}

As explained in \Cref{EX:cohomo_1_codim_2_orbits,REM:stable_classes}, for any space $G/H$ as in \Cref{THM:double_stabilization_homogeneous} satisfying additionally $\oplus_{i>0}\H^{4i}(G/H;\qq)\neq 0$, there exist infinitely many real vector bundles which admit non-negatively curved metrics and belong to pairwise distinct stable classes.

\smallskip

In the case of cohomogeneity one spaces \(M\) with all groups in the associated diagram \((G,H,K_-,K_+)\) connected and $\dim K_\pm/H$ odd (which implies $\rank K_- = \rank K_+$), we show that the map $\K^0_G(M)\otimes\qq\to \K^0(M)\otimes\qq$ is surjective if and only if $\rank G - \rank K_\pm\leq 1$, see \Cref{THM:surjectivity_of_forgetful_map_even_cohomology}. As a consequence we obtain the following statement, which in combination with the results on the existence of non-negatively curved metrics implies Part~(2) of \Cref{THM: big theorem}.

\begin{main}\label{THM:rational_stabilization_cohomo1}
 Let \(M\) be a closed connected manifold with an action by a compact connected Lie group $G$ with orbit space an interval, and let \((G,H,K_-,K_+)\) be the associated group diagram. Assume that the isotropy groups $H, K_\pm$ are connected and that $\dim K_\pm/H$ is odd. Suppose moreover that $\rank G - \rank K_\pm\leq 1$. Then for every real or complex vector bundle \(E\) over \(M\), there exist integers \(q>0\) and \(k\geq 0\) such that the Whitney sum \(qE\oplus \rr^k\) or \(qE\oplus \cc^k\) carries a $G$-vector bundle structure.
\end{main}

We finish by providing some context to the study of the map $\K^0_G(M)\otimes\qq\to \K^0(M)\otimes\qq$. The more general map $\K^*_G(M)\otimes\qq\to \K^*(M)\otimes\qq$ has been investigated and its surjectivity is known to be equivalent to the surjectivity of the map in equivariant cohomology $j^*\colon\H_G^*(M;\qq )\to\H^*(M;\qq )$. Here $j$ is the pull-back of the fibre inclusion of the Borel fibration of a $G$-space $M$:
\[
  M\hto{j} M_G:= M\times_G \E G \xto{p} \B G.
\]
When $j^*$ is surjective, the $G$-action on $M$ is called ``equivariantly formal''. This property reverberates heavily in equivariant cohomology and can be found in different contexts ranging from Hamiltonian torus actions to $G$-actions on simply-connected K\"ahler manifolds.

In this work we show, via the Chern character, that the surjectivity of $\K^0_G(M)\otimes\qq\to \K^0(M)\otimes\qq$ is equivalent to the surjectivity of the Borel forgetful map in even degrees $j^*\colon\H_G^\even(M;\qq )\to \H^\even(M;\qq )$, see \Cref{surjectivity_of_Heven_iff_K0}.
We then study the surjectivity of the pull-backs in even-degree rational cohomology under the maps $j$ and $p$ for homogeneous spaces (and more generally for biquotients) and for cohomogeneity one spaces.
The common feature of these classes of spaces is that they are pure in the sense of rational homotopy theory, a property that will allow us to characterize the surjectivity of the maps we are interested in.

 %   \item We show that, for a cohomogeneity one space \(M\) with all groups in the associated diagram \((G,H,K_-,K_+)\) connected and $\dim K_\pm/H$ odd (which implies $\rank K_- = \rank K_+$), the map $\H_G^\even(M;\qq )\to \H^\even(M;\qq )$ is surjective if and only if $\rank G - \rank K_\pm\leq 1$, see \Cref{THM:surjectivity_of_forgetful_map_even_cohomology}.

 % X\hto{j} X_G:= X\times_G \E G \xto{p} \B G.

\subsection*{Structure of the article}
In \Cref{SEC:cohomo1}, we review the basic theory of cohomogeneity one manifolds and the existence of non-negatively curved metrics.
In \Cref{SEC:K-theory}, we assemble all K-theoretic background and results.  The first half of this section leads up to a proof of Part~(\ref{THM: converse soul for cohomo 1}) of \Cref{THM: big theorem} in \Cref{SUBSEC:stabilization-over-cohomogeneity-one-spaces}.  In the second half of \Cref{SEC:K-theory}, we relate the forgetful map in rational genuine equivariant K-theory to the forgetful map in rational (Borel) equivariant cohomology. This is the first step in our proof of \Cref{THM:rational_stabilization_cohomo1}.
Finally, in \Cref{sec:cohomo1_rational_cohomo} we study the rational cohomology of homogeneous spaces, biquotients and cohomogeneity one manifolds.  This leads to the proofs of Theorems~\ref{THM:surjectivity of homogeneous bundles}, \ref{THM:double_stabilization_homogeneous} and \ref{THM:rational_stabilization_cohomo1} and, finally, of Part~(\ref{THM: rational converse soul for cohomo 1}) of \Cref{THM: big theorem}.

\smallskip

\subsection*{Acknowledgements.} The first named author was supported both by a Heisenberg grant and his research grant AM 342/4-1 of the German Research Foundation; he is moreover associated to the DFG Priority Programme 2026 (``Geometry at Infinity''). The second author received support from SNF grant 200021E-172469, the DFG Priority Programme SPP 2026 (``Geometry at Infinity''), and MINECO grant MTM2017-85934-C3-2-P. The third author is a member of the DFG Research Training Group 2240: Algebro-Geometric Methods in Algebra, Arithmetic and Topology, and parts of the research for this article where conducted within this framework.

We thank Jason DeVito for bringing the example in \Cref{EX:cohomo_1_codim_2_orbits} to our attention, and Anand Dessai, Luis Guijarro and Jeff Carlson for many helpful comments on a preliminary version of this manuscript.

We are grateful to the referees for their careful study of the article and for several helpful suggestions.

%%%%%%%%%%%%%%%%%%%%%%%%%%%%%%%%%%%%%%%%%%%%%%%%%%%%%
%%%%%%%%%%%%%%%%%%%%%%%%%%%%%%%%%%%%%%%%%%%%%%%%%%%%%
%%%%%%%%%%%%%%%%%%%%%%%%%%%%%%%%%%%%%%%%%%%%%%%%%%%%%

\section{Manifolds of cohomogeneity at most one and non-negative curvature}\label{SEC:cohomo1}

This section is organized as follows: in \Cref{sec:prelim_cohomo1} we recall some basics on cohomogeneity one manifolds, in \Cref{SEC:G_vector_bundles_nonnegative_curvature_cohomo1} we show the existence of non-negatively curved metrics on $G$-vector bundles, in \Cref{SS:DeVito_example} we provide examples of cohomogeneity one manifolds which are not homotopy equivalent to any homogeneous space and in \Cref{SEC:orbit_space_circle} we discuss the case where the orbit space is a circle.

\subsection{Preliminaries}\label{sec:prelim_cohomo1}

A manifold $M$ on which a Lie group $G$ acts is said to be of cohomogeneity one if the orbit space $M/G$ is of dimension one. Cohomogeneity one manifolds were investigated by Mostert \cite{Mo57}. When $M$ is compact the orbit space can only be a circle or a closed interval. For the case where the orbit space is a circle see \Cref{SEC:orbit_space_circle}. Here we discuss the case where $M/G$ is a closed interval, say $[-1,1]$. There are two non-principal orbits which project to the endpoints of the interval. Denote by $K_\pm$ the isotropy group at the respective point of the orbit projecting to $\pm 1$, and by $H$ the principal isotropy group. Thus the non-principal orbits equal $G/K_\pm$ and the principal orbit equals $G/H$.

The invariant neighbourhoods of $G/K_\pm$ projecting to $[-1,0]$ and $[0,1]$ respectively can be described as $G$-disk bundles $G\times_{K_\pm} \mathbb{D}^{\ell_\pm +1}$. Here $\mathbb{D}^{\ell_\pm +1}$ denotes the normal unit disk to $G/K_\pm$ at any of its points. The group $K_\pm$ acts transitively on the boundary of $\mathbb{D}^{\ell_\pm +1}$ with isotropy group $H$, hence the action is linear and we have $\partial \mathbb{D}^{\ell_\pm +1} = \mathbb{S}^{\ell_\pm} = K_\pm /H$. (Note that the codimension of the non-principal orbit $G/K_\pm$ in $M$ equals $\ell_\pm +1=\dim (K_\pm/H) +1$.) Thus the manifold $M$ can be decomposed as the gluing of the disk bundles $G\times_{K_\pm} \mathbb{D}^{\ell_\pm +1}$ along their common boundary $G/H$:
\begin{equation}\label{EQ: cohomo 1 decomposition}
    M= G\times_{K_-} \mathbb{D}^{\ell_- +1} \cup_{G/H} G\times_{K_+} \mathbb{D}^{\ell_+ +1}
\end{equation}
Conversely, let $G$ be a compact Lie group and let $H<K_\pm<G$ be closed subgroups with $K_\pm /H$ a sphere $\mathbb{S}^{\ell_\pm}$. Then one can construct a manifold $M$ as in \eqref{EQ: cohomo 1 decomposition}, on which $G$ acts in a natural way.

The tuple \((G,H,K_-,K_+)\)  of groups as above is said to be the \define{group diagram} of the corresponding cohomogeneity one manifold. Note that the tuple and the associated manifold depend not just on the abstract groups \(H\) and \(K_{\pm}\), but also on their embeddings in \(G\).

There is a \(1:1\)-correspondence between \(G\)-equivariant diffeomorphism classes of cohomogeneity one manifolds on which $G$ acts and certain equivalence classes of \emph{group diagrams} \cite[Proposition~6.37]{AB:Lie}.  We refer to \cite[Section~6.3]{AB:Lie}, \cite{GaGaZa18} or \cite[Section~1]{GroveZiller:Milnor} for details.

%%%%%%%%%%%%%%%%%%%%%%%%%%%%%%%%%%%%%%%%%%%%%

\subsection{Non-negative curvature}\label{SEC:G_vector_bundles_nonnegative_curvature_cohomo1}

Here we review the existence of non-negatively curved metrics on $G$-vector bundles over homogeneous spaces and certain cohomogeneity one manifolds (see \Cref{SUBSEC:G-vector-bundles} for the definition of $G$-vector bundle).

\begin{theo}\label{hop:non-negative-metrics}
  Let $M$ be a closed connected manifold with an action by a compact Lie group $G$. Suppose that one of the following holds:
  \begin{enumerate}
      \item the $G$-action on $M$ is transitive, i.e.\ $M$ is homogeneous,
      \item the $G$-action is of cohomogeneity one with orbit space an interval, and the non-principal orbits have codimension $\leq 2$ in $M$.
  \end{enumerate}
  Then the total space of every real $G$-vector bundle over $M$ carries a $G$-invariant metric of non-negative curvature.
\end{theo}

\begin{rem}
\Cref{hop:non-negative-metrics} is stated for real $G$-vector bundles but of course equally applies to complex $G$-vector bundles; one can simply forget the complex structure.
\end{rem}

\begin{proof}[Proof of \Cref{hop:non-negative-metrics}]

Part~(1) is well known, we recall here the details. A closed homogeneous $G$-space can be written as $G/H$ with $G$ compact. Any $G$-vector bundle over $G/H$ is isomorphic to the quotient $G\times_H V:=(G\times V)/H$, where $V$ is a representation of $H$ and $H$ acts diagonally on $G\times V$, see \cite[p.~130]{Segal:Equivariant}. Take the product of a bi-invariant metric on $G$ and an $H$-invariant flat metric on the vector space underlying to $V$. This a non-negatively curved metric on $G\times V$ which is invariant under the action of $H$. Thus it induces a metric on the quotient $G\times_H V$ which is $G$-invariant by construction and of non-negative curvature by the Gray--O'Neill formula for Riemannian submersions \cite{Gr67,On66}.

\smallskip

Let now $M$ be a cohmogeneity one manifold as in the statement. Let $E$ be a $G$-vector bundle, and denote its rank by $m$. Endow $E\to M$ with an \emph{orthogonal} $G$-vector bundle structure (see \Cref{LEM:exisntece_of_orthogonal_Gbundle_structure}), thus the structure group of $E$ can be reduced from $GL(m)$ to $O(m)$. The associated principal $O(m)$-bundle $P\to M$ inherits a compatible $G$-action; more precisely, $P$ has a right free $O(m)$-action and a left $G$-action and these actions commute  \cite[p. 257]{Lashof:equivariant}.

Thus $P$ has an action of $O(m)\times G$, and the orbit space $P/\left(O(m)\times G\right)$ is by construction equivalent to $M/G = [-1,1]$, i.e.\ it is of cohomogeneity one. Since $O(m)$ acts freely on $P$ it follows that the diagram of the cohomogeneity one action equals \((O(m)\times G,H,K_-,K_+)\). In particular, the non-principal orbits in $P$ have the same codimension as those in $M$, thus $P$ carries an $O(m)\times G$-invariant metric $g_P$ of non-negative curvature by the existence result of Grove and Ziller \cite[Theorem~E and Remark 2.8]{GroveZiller:Milnor}.

Let $V$ denote the standard representation of $O(m)$ and endow it with a flat $O(m)$-invariant metric $g_V$. By construction, the associated bundle $P\times_{O(m)} V$ is a $G$-vector bundle (the action given by $g[p,v]=[gp,v]$) which is isomorphic to the original bundle $E$. We finish by constructing adequate metrics as in the homogeneous case above. The product metric $g_P + g_V$ on $P\times V$ is $O(m)$-invariant and of non-negative curvature, thus it descends to a metric on $E=P\times_{O(m)} V$ which is $G$-invariant by construction and of non-negative curvature by the Gray--O'Neill formula \cite{Gr67,On66}.

We refer to \cite[Proposition 7.1]{Hoel2010} for a detailed description of the general construction above in the case where $P$ equals the frame bundle of $M$, which is a principal $SO(n)$-bundle (where $n=\dim M$) and whose associated vector bundle $P\times_{SO(n)}V$ is the tangent bundle $TM$.
\end{proof}

%%%%%%%%%%%%%%%%%%%%%%%%%%%%%%%%%%%%%%%%%%%%%%%%%%%%

\subsection{Non-homogeneous cohomogeneity one manifolds}\label{SS:DeVito_example}

Here we justify the claims in the second paragraph of \Cref{EX:cohomo_1_codim_2_orbits}.

For $n\geq 0$, we define the product manifold $M^{2n+4}:= (\cc P^2\sharp \overline{\cc P}^2 )\times\left(\mathbb{S}^2\right)^{n}$. %As Jason DeVito explained to us, these manifolds carry cohomogeneity one actions satisfying the needed assumptions in Part~(1) of \Cref{THM: big theorem}, and they are not homotopy equivalent to any homogeneous space. In this section we elaborate on these facts.
The factor $\cc P^2\sharp \overline{\cc P}^2 $ admits a number of cohomogeneity one actions with associated group diagrams $(SU(2),\zz_k,S^1,S^1)$, for any $k$ odd (see \cite[(2.3)]{GroveZiller:Lifting} for details). Each factor $\mathbb{S}^2$ carries the standard transitive action by $SU(2)$ with principal isotropy group $S^1$. These actions induce a product action on the whole $M^{2n+4}$ by $SU(2)^{n+1}$ which is of cohomogeneity one with associated group diagram
$$\left( SU(2)^{n+1},\zz_k\times (S^1)^{n}, S^1\times (S^1)^{n}, S^1\times (S^1)^{n}\right)$$
All these actions have codimension two singular orbits  and the isotropy groups associated to the singular orbits have maximal rank. However, in order to apply Part~(1) of \Cref{THM: big theorem} we need all isotropy groups to be connected, thus we consider the action corresponding to $k=1$. We summarize the discussion in the following proposition, where $T^n:=(S^1)^{n}$ denotes the $n$-dimensional torus.

\begin{prop}\label{PROP:THMA_applies_to_M_2n4}
   For any $n\geq 0$ the manifold $M^{2n+4}$ admits a cohomogeneity one action with associated diagram $\left( SU(2)^{n+1},T^{n}, T^{n+1}, T^{n+1}\right)$. In particular, Part~(1) of \Cref{THM: big theorem} applies to $M^{2n+4}$.
\end{prop}

Next we compare the homotopy type of $M^{2n+4}$ with that of closed homogeneous spaces. The goal is to prove the following

\begin{prop}\label{PROP:M_2n4_isnot_homogeneous}
  The manifold $M^{2n+4}$ is not homotopy equivalent to any homogeneous space for any $n\geq 0$.
\end{prop}

\Cref{PROP:M_2n4_isnot_homogeneous} follows by combining \Cref{PROP:the_only_homogeneous_is_S2,PROP:M_is_not_homotopy_equivalent_S2} below.

\begin{lemma}\label{PROP:the_only_homogeneous_is_S2}
  Up to diffeomorphism, the only simply-connected closed homogeneous space with the rational homotopy groups of $M^{2n+4}$ is $(\mathbb{S}^2)^{n+2}$.
\end{lemma}

\begin{proof}
In order to prove \Cref{PROP:the_only_homogeneous_is_S2} one draws on the classification tools developed by Totaro for biquotients in his article \cite{Totaro:Cheeger}---for the definition of the latter see \Cref{secpre}; note, however, that Totaro's definition is slightly more general.

The first observation we make is that, the rational homotopy groups of $\cc P^2\sharp \overline{\cc P}^2$ are given by  $\pi_i (\cc P^2\sharp \overline{\cc P}^2)\otimes \qq= \qq^2$ for $i\in \{2,3\}$ and are trivial otherwise. That is, this space has the same rational homotopy groups as $\s^2\times \s^2$. (Indeed, for example, from its cohomological structure we can directly derive a minimal Sullivan model of $\cc P^2\sharp \overline{\cc P}^2$ as being given by $(\Lambda \langle x,y,n,m\rangle ,x\mapsto 0,y\mapsto 0,n\mapsto x^2+y^2, m\mapsto xy)$ with $\deg x=\deg y=2$ and $\deg n=\deg m=3$; see \Cref{secpre} for a discussion on rational homotopy theory). Hence the rational homotopy groups of $M^{2n+4}$  equal
  \[
   \pi_i (M^{2n+4}) \otimes \qq =\pi_i ((\s^2)^{n+2}) \otimes \qq =
  \begin{cases}
    \qq^{n+2}, & \text{for } i=2,3 \\
    0, & \text{otherwise}
  \end{cases}
  \]
 Let $G/H$ be a simply-connected closed homogeneous space with the rational homotopy groups of $M^{2n+4}$. In particular, $G/H$ is a biquotient. Since $G/H$ is simply-connected and since we only care about its diffeomorphism type, we can assume by \cite[Lemma 3.1]{Totaro:Cheeger} that $G$ is simply-connected and that $H$ is connected and does not act transitively on any simple factor of $G$ (note that, unlike in the case of a general biquotient where $H$ is a subgroup of $G\times G$, in our case $H$ is a closed subgroup of $G$).
Using Totaro's classification it was shown in Case 1 of the proof of \cite[Theorem 3.1]{Dev17} that if such a homogeneous space $G/H$ has the rational homotopy groups of $(\s^2)^{n+2}$, then $(G,H)=(SU(2)^{n+2},T^{n+2})$. The group $H$ hence is a maximal torus of $G$. Since all maximal tori are conjugated, this conjugation induces a diffeomorphism of $G/H$ and $(\s^2)^{n+2}=(SU(2)/S^1)^{n+2}$, which proves Lemma \ref{PROP:the_only_homogeneous_is_S2}.
\end{proof}

%%not to be confused with the following definition of degree, which is only %used in this section).

%%\begin{alignat*}{7}
%%s &= \#\{ \text{simple factors}\} &&= \text{multiplicity of \(2\) in %\(\deg G\)} &&= \dim_\qq \pi_3^\qq (G)\\
%%l &= \#\{ \text{circle factors}\} &&= \text{multiplicity of \(1\) in %\(\deg G\)} &&= \dim_\qq \pi_1^\qq (G)
%%\end{alignat*}
%% If \(f\) denotes the homomorphism \(H \to (G\times G)/Z(G)\), then
%% \[
%%   Stab(1) = f^{-1}(\Delta(G)/Z(G)) \supset f^{-1}(1) = \ker(f).
%% \]
%% So \(Stab(1) = 1\) implies that \(\ker f\) is trivial.%%
%%This is only possible when \(\pi_{2d-1}^{\qq}(B)\neq 0\).  %
%%A degree \(d\) of a simple factor \(G_i\) is \textbf{killed} by \(H\) if %the morphism \(\pi_{2d-1}^\qq(H)\to\pi_{2d-1}^\qq(G_i)\) is nonzero. %

\begin{lemma}\label{PROP:M_is_not_homotopy_equivalent_S2}
  The manifold $M^{2n+4}$ is not homotopy equivalent to $(\mathbb{S}^2)^{n+2}$.
\end{lemma}

\begin{proof}
To prove this result we look at squares of elements in $\H^2(-;\zz)$. Any element in $\H^2((\mathbb{S}^2)^{n+2};\zz)$ is of the form $\sum_i^{n+2} a_ix_i$, where $a_i\in \zz$ are arbitrary and $x_i\in \H^2(\mathbb{S}^2;\zz)$ coming from the $i$-th factor in $(\mathbb{S}^2)^{n+2}$. The square of this arbitrary element can be computed as follows, where we shall use the fact that the squares $x_i^2$ vanish:
$$
\left(\sum_i^{n+2} a_ix_i\right)^2  = \sum_i^{n+2} \left(a_ix_i\right)^2  + 2 \sum_{i<j}^{n+2} a_ia_jx_ix_j =  2 \sum_{i<j}^{n+2} a_ia_jx_ix_j,
$$
In particular, every square of an element in $\H^2((\mathbb{S}^2)^{n+2};\zz)$ is divisible by two.

On the other hand, recall that $\H^2(\cc P^2\sharp \overline{\cc P}^2;\zz)\cong\zz^2$ is generated by two elements $y,z$ with $y^2=1$ and $z^2=-1$ in $\H^4(\cc P^2\sharp \overline{\cc P}^2;\zz)\cong\zz$ (see e.g.\ \cite[Example 10.9, p.361]{Bre97}). Denote by $\bar y\in\H^2(M^{2n+4};\zz)$ the element represented by $y$.
The square $\bar y ^2$ clearly equals the element represented by $y^2$.
Since $y^2=1$, it follows that $\bar y ^2$ is not divisible by two. This completes the proof.
\end{proof}
\begin{rem}
One may give the following alternative proof for \Cref{PROP:M_is_not_homotopy_equivalent_S2}.
We recall that $\cc P^2\sharp \overline{\cc P}^2$ is not a spin manifold, as its second Stiefel--Whitney class is non-zero. Using the multiplicativity formula for Stiefel--Whitney classes we derive that $M^{2n+4}$ does have non-trivial second Stiefel--Whitney class.

Yet, since a product of spheres is stably parallelizable, its characteristic classes vanish. Since the latter classes are homotopy invariants, $M^{2n+4}$ cannot be homotopy equivalent to a product of spheres.
\end{rem}

\subsection{Orbit space being a circle}\label{SEC:orbit_space_circle}

Let $M$ be a $G$-manifold with $M$ and $G$ compact and with orbit space $M/G$ a circle, and denote by $H$ the principal isotropy group. In this case \cite{GaGaZa18}, all orbits are principal and $M$ is equivariantly diffeomorphic to the total space of a $G/H$-bundle over a circle with structure group $N_G(H)/H$, where $N_G(H)$ denotes the normalizer of $H$ in $G$. Equivalently, $M$ has the structure of the mapping torus
$$
\frac{[0,1]\times (G/H)}{(1,x)\sim (0,R_g x)}\qquad \text{ for certain } g\in N_G(H)
$$
where $R_g$ denotes right translation by $g$ (cf.~\cite[Corollary 4.3, p.~41]{Bre72}). In particular the fundamental group of $M$ is infinite.

Every such $M$ admits a $G$-invariant metric of non-negative curvature. Indeed, since $R_g$ is a $G$-equivariant diffeomorphism, the product metric of the interval with a non-negatively curved $G$-invariant metric on $G/H$ (e.g.\ a normal homogeneous metric) induces the desired metric on the mapping torus $M$. The same exact arguments as in the proof of \Cref{hop:non-negative-metrics} show that every $G$-vector bundle over $M$ admits a $G$-invariant metric of non-negative curvature.

In this case, by the Splitting Theorem in \cite{CheGro72} $M$ has a covering space diffeomorphic to $M_0\times T^m$, where $M_0$ is a simply connected space that admits a metric of non-negative curvature and $m$ is the rank of $\pi_1(M)\otimes\qq$. The existence of non-negatively curved metrics on vector bundles over spaces of the form $C\times T$ (with $C$ simply connected and non-negatively curved and $T$ a torus) has been extensively studied in the series of papers \cite{BelKap01,BelKap03}. However we include a couple of comments (\Cref{REM:integral_Ktheory_circle,REM:rational_Ktheory_circle}) in the corresponding sections to complement the discussion.

%%%%%%%%%%%%%%%%%%%%%%%%%%%%%%%%%%%%%%%%%%%%%%%%%%%
%%%%%%%%%%%%%%%%%%%%%%%%%%%%%%%%%%%%%%%%%%%%%%%%%%%%
%%%%%%%%%%%%%%%%%%%%%%%%%%%%%%%%%%%%%%%%%%%%%%%%%%%%

\section{K-theory}
\label{SEC:K-theory}
In this section we assemble all results in K-theory that we will need.  After a brief introduction to the theory in
\Cref{SUBSEC:G-vector-bundles,SUBSEC:genuine-equivariant-K-theory,SUBSEC:surjectivity-and-stabilization-of-bundles}, we can already prove Part~(\ref{THM: converse soul for cohomo 1}) of \Cref{THM: big theorem} in \Cref{SUBSEC:stabilization-over-cohomogeneity-one-spaces}.  The remaining \Cref{SUBSEC:Borel-equivariant-cohomology-theories,SUBSEC:surjectivity-in-genuine-versus-Borel-equivariant-K-theory,SUBSEC:surjectivity-in-K-theory-versus-cohomology} are devoted to a refinement of the following result of Fok:
\begin{theo*}[Fok]
  Let \(X\) be a smooth compact manifold with a smooth action by a compact Lie group \(G\).  The forgetful map in rational genuine equivariant K-theory \(\K^*_G(X;\qq)\to \K^*(X;\qq)\) is surjective if and only if the forgetful map in (Borel) equivariant rational cohomology \(\H^*_G(X;\qq)\to\H^*(X;\qq)\) is surjective.
\end{theo*}
This result first appeared as part of Theorem~1.3 in \cite{Fok:EquivariantFormality}, and later as Theorem~5.6 of \cite{CarlsonFok:isotropy}.\footnote{
  The order in which these two papers were published appears to be opposite to the order in which they were written.
}
The refinement we need for our application (\Cref{THM:rational_stabilization_cohomo1}) is a separate statement for odd- and even-degree cohomology, see \Cref{surjectivity_of_Heven_iff_K0} below.  It is not clear to us whether this refined version can be extracted from the original proof in \cite{Fok:EquivariantFormality}.  We instead supply a direct argument based on Atiyah--Segal completion and the Chern character.  As we were made aware by Jeff Carlson, this is the same line of proof as in the final version of \cite{CarlsonFok:isotropy}.  We nonetheless include a detailed exposition, as it serves several purposes: it keeps our presentation more self-contained, it makes the proof more accessible to non-expert readers, and it clarifies some ambiguities\slash inaccuracies in the existing literature (cf.\ in particular \Cref{warning}).

\subsection{\(G\)-vector bundles}
\label{SUBSEC:G-vector-bundles}
Let \(G\) be a topological group and \(X\) a \(G\)-space, i.e.\ a topological space on which \(G\) acts continuously.  We will soon add restrictions on \(G\) and \(X\), but the basic definitions work in full generality.  Recall that a (real or complex) \define{$G$-vector bundle} over $X$ is another $G$-space $E$ together with a $G$-equivariant map $p\colon E\to X$ such that
\begin{compactenum}[(i)]
\item $p\colon E\to X$ is a (real or complex) vector bundle, and
\item for any $g\in G$ and $x\in X$ the group action $g\colon E_x\to E_{gx}$ is an isomorphism of (real or complex) vector spaces.
\end{compactenum}
Here \(E_x\) denotes the fibre of \(p\) over \(x\).  Note that any complex \(G\)-vector bundle can also be viewed as a real \(G\)-vector bundle.

In some situations, it is convenient to have an additional orthogonal structure available.  An \define{orthogonal vector bundle} over a space $X$ is a real vector bundle $p\colon E\to X$ together with a fibre metric, i.e.\ together with a continuous choice of inner product on each fibre $E_x$.  Likewise, an \define{orthogonal $G$-vector bundle} over a \(G\)-space $X$ is a real \(G\)-vector bundle such that
\begin{compactenum}[(i)]
\item $p\colon E\to X$ is an orthogonal vector bundle, and
\item for any $g\in G$ and $x\in X$ the group action $g\colon E_x\to E_{gx}$ is a linear isometry with respect to the inner products on each fibre.
\end{compactenum}
We have already used the following observation concerning the existence of orthogonal structures in the proof of \Cref{hop:non-negative-metrics} above:

\begin{lemma}\label{LEM:exisntece_of_orthogonal_Gbundle_structure}
  Let $X$ be a compact (or more generally paracompact) Hausdorff \(G\)-space for a compact group \(G\).  Then any real \(G\)-vector bundle over \(X\) admits a fibre metric that makes it an orthogonal $G$-vector bundle.
\end{lemma}

\begin{proof}
  We can view a fibre metric on a vector bundle \(p\colon E\to X\) as a global section of the dual  of the second symmetric power bundle, \(s\colon X \to \operatorname{Hom}(\operatorname{Sym}^2(E),\rr)\),
  which is positive definite at each point.  The assumptions on \(X\) guarantee the existence of such an inner product \(s\) \cite[Part~I, Chapter~3, Theorems~5.5 and 9.5]{Husemoller}. Now the action of \(G\) on \(E\) induces an action on the bundle \(\operatorname{Hom}(\operatorname{Sym}^2(E),\rr)\).  As \(G\) is a compact, we can average \(s\) over \(G\) to obtain an equivariant section \(\bar s\)  (cf.\ \cite[above Proposition~1.1]{Segal:Equivariant}).  This equivariant section \(\bar s\) is still positive definite.  Indeed, for \(x\in X\) and \(v\in E_x\), the compactness of \(G\) ensures that the map \(G\to \rr\) given by \(g\mapsto s_{gx}(gv,gv)\) takes a non-negative minimum on \(G\), and this implies that \(\bar s_x(v,v) = 0\) if and only if \(v=0\) in \(E_x\).
\end{proof}

\subsection{Genuine equivariant K-theory}
\label{SUBSEC:genuine-equivariant-K-theory}

In the following discussion of K-theory, it is convenient to restrict attention to compact Lie groups \(G\) and to concentrate on \(G\)-CW complexes \cite{Matumoto:Whitehead}, or on spaces that are \(G\)-homotopy equivalent to \(G\)-CW complexes.
In our applications, we will mainly be interested in smooth compact manifolds on which \(G\) acts smoothly.  These fit into the framework by the following result of Illman \cite[Corollary~7.2]{Illman}:
\begin{theo}[Illman]\label{thm:illman}
  Any smooth manifold \(X\) with a smooth action of a compact Lie group \(G\) can be given the structure of a \(G\)-CW complex.
\end{theo}
Any compact such manifold consequently has the structure of a \emph{finite} \(G\)-CW complex.  It will also be important for us that, when we forget the \(G\)-action on our \(G\)-CW complexes, we remain within the category of spaces homotopy equivalent to CW complexes.  More generally, one has the following result:
\begin{theo}[Waner, Illman]\label{thm:G-CW-restriction}
  Consider a closed subgroup \(H\) of a compact Lie group \(G\).  Any \(G\)-CW complex is \(H\)-homotopy equivalent to an \(H\)-CW complex.  Any finite \(G\)-CW complex is \(H\)-homotopy equivalent to a finite \(H\)-CW complex.
\end{theo}
The first part of this theorem goes back to \cite[Proposition~3.8]{Waner:Milnor}. For the second part, see \cite[Theorem~A and Corollary~B]{Illman:restricting}.
\medskip

\subsubsection{Non-equivariant K-theory}
So let us turn to K-theory.
We will use the notation \(\K^*(-)\) to denote complex topological K-theory, viewed as a cohomology theory on CW complexes in the sense of \cite[Chapter~18]{May:Concise}, or more generally as a cohomology theory on spaces homotopy equivalent to CW complexes.  This theory is two-periodic in the sense that \(\K^{i+2}(-)\) is naturally isomorphic to \(\K^i(-)\).  For a finite CW complex, the K-group in degree zero, \(\K^0(X)\), can be identified with the Grothendieck group of complex vector bundles \(\K(X)\), defined as follows: the set of isomorphism classes of complex vector bundles over \(X\) form a monoid with respect to the direct sum of vector bundles (also known as Whitney sum), and \(\K(X)\) is the Grothendieck group of this monoid. In particular, any element of \(\K(X)\) can be written as a formal difference of two vector bundles over \(X\).  The tensor product of vector bundles induces a multiplication on the Grothendieck group \(\K(X)\),  and this can be extended to a ring structure on \(\K^0(X)\) and a \(\zz_2\)-graded ring structure on the direct sum \(\K^0(X)\oplus \K^1(X)\) for arbitrary \(X\).

Real K-theory, which we will denote as \(\KO^*(-)\), can be described analogously in terms of real vector bundles. This theory behaves very similary to complex K-theory, except that it is eight-periodic rather than two-periodic.  We will mostly use complex K-theory in this paper; real K-theory will show up only occasionally.

\subsubsection{Genuine equivariant K-theory}
We now consider genuine equivariant K-theory \(\K_G^*(-)\), as introduced in \cite{Segal:Equivariant}. We emphasize the attribute \emph{genuine} as a distinction from the Borel equivariant K-theory to be considered later in \Cref{SUBSEC:Borel-equivariant-cohomology-theories}.  For a finite \(G\)-CW complex \(X\), the group \(\K_G^0(X)\) in degree zero can again be identified with a Grothendieck group of vector bundles, namley with the Grothendieck group \(\K_G(X)\) of isomorphism classes of complex \(G\)-vector bundles over \(X\).  For example, as \(G\)-vector bundles over a point are simply \(G\)-representations, \(\K_G^0(\point)\) is isomorphic to the complex representation ring \(\R(G)\).  In general, \(\K_G^0(X)\) has the structure of an \(\R(G)\)-algebra, since the projection to a point \(p\colon X\to \point\) induces a ring homomorphism \(p^*\colon \R(G)\to \K_G^0(X)\).  We also have an obvious ring homomorphism \(u\colon \K_G(X)\to \K(X)\) that forgets the \(G\)-action on the bundles.  This ring homomorphism extends to a natural transformation \(u\colon \K_G^*(-)\to\K^*(-)\), the \define{genuine forgetful map}, defined on all \(G\)-CW complexes.  When \(G\) is trivial, \(u\) is the identity.

The equivariant K-theory of a \(G\)-CW complex \(X\) with a \emph{free} \(G\)-action can be canonically identified with the K-theory of the orbit space \(X/G\), such that the composition
\begin{equation}\label{eq:K_G-of-free}
  \K^*(X/G)\cong \K_G^*(X) \xrightarrow{u} \K^*(X)
\end{equation}
is the pullback along the projection \(X\to X/G\). This is easily verified geometrically in degree zero for finite \(G\)-CW complexes, see \cite[Proposition~2.1]{Segal:Equivariant} or \cite[Chapter~XIV, (1.1)]{May96}.  For the general statement, see \cite[Lemma~12]{MayMcClure:reduction}.

\subsubsection{Rational versions}
Rationalized versions of the complex K-theories discussed above will be denoted \(\K^*(-;\qq)\) and \(\K^*_G(-;\qq)\), respectively.  One way to obtain these theories is to note that the functor \(\K^*_G(-)\otimes\qq\) defines a cohomology theory on finite \(G\)-CW complexes, and to extend it to a cohomology theory on all \(G\)-CW complexes using an equivariant generalization of Adam's version of Brown representability \cite{Adams:Brown,Neeman:Brown}, see \cite[Chapter~XIII, Theorem~3.4]{May96}.
It is obvious from this construction that, for a finite \(G\)-CW complex, we have a canonical identification
\begin{equation}\label{eq:finite-rationalization}
  \K^*_G(X;\qq)\cong \K^*_G(X)\otimes\qq.
\end{equation}
(For infinite complexes, the two sides can differ. This is nothing special about K-theory:  \(\H^*(X;\qq)\) also differs from \(\H^*(X;\zz)\otimes\qq\) for infinite CW complexes.)

\medskip
The geometric description of \(\K^0(X)\) and \(\K^0_G(X)\) for finite \(G\)-CW complexes, which is essential for our applications, holds in slightly greater generality:
\begin{lemma}\label{lem:K0=K}
  For any compact Hausdorff space \(X\) that is homotopy equivalent to a finite CW complex, \(\K^0(X)\) can be identified with the Grothendieck group \(\K(X)\).  More generally, for any compact Lie group \(G\),
 and any compact Hausdorff \(G\)-space \(X\) that is \(G\)-homotopy equivalent to a finite \(G\)-CW complex, \(\K^0_G(X)\) can be identified with the Grothendieck group \(\K_G(X)\).
\end{lemma}

\begin{proof}
  This follows from the homotopy invariance of the Grothendieck group \(\K_G(X)\); see \cite[Proposition~2.3]{Segal:Equivariant}.
\end{proof}
\begin{warning}\label{warning}
  \Cref{lem:K0=K} fails for compact Hausdorff spaces \(X\) not homotopy equivalent to finite CW complexes.  More precisely, the notation \(\K^0(X)\) is ambiguous for such spaces,  a subtlety that needs to be taken into account when referring to the literature:

  Traditionally, K-theory is developed as a \emph{cohomology theory on compact Hausdorff spaces} \cite[Definition~2.9.1]{Park:K-theory}.  This is the approach taken in the early papers of Atiyah and Segal such as \cite{AtiyahSegal} and \cite{Segal:Equivariant}, and also in textbooks such as \cite{Atiyah:K-theory} or \cite{Karoubi:K-theory}.  Let us temporarily denote the resulting theory by \(\K^*_\text{geom}\).  In this theory, \(\K^0_{\text{geom}}(X)\) is the Grothendieck group \(\K(X)\) for any compact Hausdorff space \(X\), by definition.

  Our view of K-theory as a cohomology theory on CW complexes, on the other hand, is prevalent in homotopy theory and in later papers such as \cite{AHJM} (see third line therein).  This theory can easily by extended to a theory on all spaces in a completely formal way, using CW approximations,
  cf.\ \cite[Chapter~18, \S\,1 and Chapter~24, \S\,1, final definition and preceding paragraph]{May:Concise}.
  However, this extended theory \(\K^*\) differs from \(\K^*_\text{geom}\) for general compact Hausdorff spaces.

  The problem is that \(\K_{\text{geom}}^*(-)\) is not invariant under \emph{weak} equivalences.
  For example, consider the Warsaw circle \(W\): a compact Hausdorff space with a single point as CW approximation, but whose Čech cohomology agrees with the cohomology of a circle.  Clearly, \(\K^1(W)\) vanishes, but \(\K^1_{\text{geom}}(W) \neq 0\), as follows easily from \Cref{rem:Karoubis-Chern-character} below.
\end{warning}

\subsection{Surjectivity and stabilization of bundles}
\label{SUBSEC:surjectivity-and-stabilization-of-bundles}
Here is the well-known but key result that allows us to pass from K-theoretical considerations to the explicit statements about vector bundles in the main theorems stated in the introduction:

\begin{prop}\label{surjectivity-of-forgetful-map-in-Kgroups}
  Let \(X\) be a finite \(G\)-CW complex for a compact Lie group \(G\).
    Then the genuine forgetful map \(u\colon \K_G^0(X) \to \K^0(X)\) is surjective if and only if, up to stabilization, every complex vector bundle \(E\) over \(X\) is the underlying vector bundle of a \(G\)-vector bundle over~\(X\) (i.e.\ if and only if for any sufficiently large integer \(k\) the direct sum \(E\oplus \cc^k\) is isomorphic to the underlying non-equivariant vector bundle of a complex \(G\)-vector bundle over~\(X\)).
\end{prop}
\begin{proof}
  Note first that by assumption and by \Cref{thm:G-CW-restriction}, the underlying non-equivariant space \(X\) is a compact Hausdorff space homotopy equivalent to a finite CW complex.
  So by \Cref{lem:K0=K} the K-groups involved can be identified with the Grothendieck groups of vector bundles \(\K_G(X)\) and \(\K(X)\), respectively.

  Assume now that \(u\) is surjective.  Then any vector bundle \(E\) over \(X\) can be written as \(E = uF_1 - uF_2\) in \(\K(X)\) for certain \(G\)-vector bundles \(F_1\) and \(F_2\).
  By \cite[Proposition~2.4]{Segal:Equivariant}, there exists a \(G\)-vector bundle \(F_2^\perp\) over \(X\) such that \(u(F_2\oplus F_2^\perp) \cong \cc^m\) for some \(m\).  So, altogether, we have the equality
  \[
    E + \cc^m = (uF_1 - uF_2) + (uF_2 + uF_2^\perp) = u(F_1+F_2^\perp)
  \]
  in \(\K(X)\).  This implies that, for sufficiently large \(k'\), the bundle \(E\oplus\cc^{m+k'}\) is isomorpic to \(u(F_1\oplus F_2^\perp) \oplus \cc^{k'}\) and hence to \(uF\) for the \(G\)-vector bundle \(F:= F_1\oplus F_2^\perp\oplus\cc^{k'}\), where \(\cc^{k'}\) is equipped with the trivial \(G\)-action. Set $k=m+k'$ and the claim follows.

  Conversely, assume that every vector bundle over \(X\) is the underlying vector bundle of a \(G\)-vector bundle over~\(X\) up to stabilization. Take any class $E_1-E_2$ in \(\K(X)\). By assumption there is some $k_i$ such that the direct sum \(E_i\oplus \cc^{k_i}\) is isomorphic to the underlying non-equivariant vector bundle of a \(G\)-vector bundle $F_i$, for $i=1,2$. We may assume without loss of generality that \(k_1 = k_2\). Then
\[
  E_1 - E_2 = E_1\oplus \cc^{k_1} - E_2\oplus \cc^{k_2} = uF_1 - uF_2 = u(F_1-F_2)
\]
in \(\K(X)\).  Thus \(u\) is surjective.
\end{proof}

This integral result is all we need for the proof of \Cref{THM: big theorem}, Part~(\ref{THM: converse soul for cohomo 1}), as presented in \Cref{SUBSEC:stabilization-over-cohomogeneity-one-spaces} below.  However, we will also be interested in situations where we only know the surjectivity of the \emph{rational} forgetful map \(u_\qq\colon \K^0_G(X;\qq)\to \K^0(X;\qq)\).  This rational surjectivity also has implications for the stabilization of vector bundles.  They are weaker, of course, than the implications in the integral case, but they hold simultaneously for both real and complex bundles. This is due to the following observation.

\begin{prop}\label{prop:rational_realification_ir_surjective}
 For any finite CW complex \(X\), the rational realification map $\K^0(X;\qq)\to \KO^0(X;\qq)$ induced by forgetting the complex structure is surjective.
\end{prop}
\begin{proof}
The composition of the complexification \(\KO^0(X)\to \K^0(X)\) with the realification \(\K^0(X)\to \KO^0(X)\) is multiplication by \(2\) on \(\KO^0(X)\). This composition is still multiplication by \(2\) after passing to rational coefficients, hence an isomorphism.  So rational realification is surjective.
\end{proof}
Here, then, is the rational version of \Cref{surjectivity-of-forgetful-map-in-Kgroups}:
\begin{prop}\label{rational-surjectivity-of-forgetful-map-in-Kgroups}
  Let \(X\) be a finite \(G\)-CW complex for a compact Lie group \(G\).
  Suppose that the rational genuine forgetful map \(u_\qq\colon \K^0_G(X;\qq)\to \K^0(X;\qq)\) is surjective. Then for any real or complex vector bundle \(E\) over \(X\), there exist integers \(q>0\) and \(k\geq 0\) such that the Whitney sum \(qE\oplus \rr^k\) or \(qE\oplus \cc^k\) is isomorphic to the underlying non-equivariant bundle of a real or complex \(G\)-vector bundle, respectively.  (Here \(qE\) denotes the Whitney sum of \(q\) copies of \(E\).)
\end{prop}
\begin{proof}
  By the same argument as in the proof of \Cref{surjectivity-of-forgetful-map-in-Kgroups}, the K-groups involved can be identified with the Grothendieck groups of vector bundles \(\K_G(X)\) and \(\K(X)\), respectively.  Recall moreover from \eqref{eq:finite-rationalization} that, under our finiteness assumptions, we can identify \(u_\qq\) with the localization \(u\otimes\qq\colon \K_G(X)\otimes\qq\to \K(X)\otimes\qq\) of the integral forgetful map. Let \(E\) be a complex vector bundle. Since \(u\otimes\qq\) is surjective, there exist \(G\)-vector bundles \(F_1\) and \(F_2\) and integers \(p'\) and \(q'\neq 0\) such that \(E = \frac{p'}{q'} \cdot (uF_1 - uF_2)\) in \(\K(X)\otimes\qq\). Indentifying this latter group with the localization of the \(\zz\)-module \(\K(X)\) by the multiplicative set \(\zz\setminus\{0\}\), we deduce that
  \(
    q''q'\cdot E = q''p'\cdot (uF_1 - uF_2)
  \)
  in \(\K(X)\) for some integer \(q''\neq 0\).
  So
  \[
    qE = u(q''p'F_1 - q''p'F_2)
  \]
  in \(\K(X)\), where \(q:= q''q'\) and where \(q''p'F_1\) and \(q''p'F_2\) are \(G\)-vector bundles by construction. (Observe that a finite Whitney sum of $G$-vector bundles has a $G$-vector bundle structure). Arguing as in the proof of \Cref{surjectivity-of-forgetful-map-in-Kgroups}, the claim for complex bundles follows.

Now let $E$ be a real vector bundle. Since rational realification $\K(X)\otimes\qq\to \KO(X)\otimes\qq$ is surjective by \Cref{thm:G-CW-restriction} and \Cref{prop:rational_realification_ir_surjective}, there exist complex bundles $\bar E_1$ and $\bar E_2$ and integers \(p',q'\) such that \(E =  \frac{p'}{q'}\cdot (r\bar E_1 - r\bar E_2)\) in \(\KO(X)\otimes\qq\), where \(r\) denotes realification of complex bundles. This implies that
\[
  q''q'E = r(q''p'\bar E_1) - r(q''p'\bar E_2)
\]
in \(\KO(X)\) for some \(q''\neq 0\).  Using the existence of a complex bundle whose Whitney sum with $q''p'\bar E_2$ is a trivial bundle, we conclude that $q''q'E$ is stably equivalent to $r\bar E$ for some complex vector bundle $\bar E$. By the first part of the proof, there exists an integer $q'''$ such that $q'''\bar E$ is stably equivalent to some complex $G$-vector bundle $F$. Taking $q:=q'''q''q'$, we find that $qE$ is stably equivalent to $ruF=urF$, i.e.\ to the underlying bundle of the real $G$-vector bundle $rF$. This completes the proof.
\end{proof}

The next corollary is more of an aside, and needs some terminology not used elsewhere in this paper. For simplicity, and to ensure agreement between the definitions available in different sources, let us assume that \(X\) is a \emph{connected} finite CW complex.  Two real vector bundles $E_i\to X$ with $i=1,2$ belong to the same \define{stable class} if there exist integers $k_i$ such that $E_1\oplus\rr^{k_1}$ and $E_2\oplus\rr^{k_2}$ are isomorphic.
The set of stable classes is an abelian group which is isomorphic to the so-called \emph{reduced K-group} $\widetilde{\KO}(X)$, a subgroup of \(\KO(X)\) for which we have a canonical splitting  $\KO(X)\cong\zz\oplus\widetilde{\KO}(X)$ \cite[\S\,2.1]{Hatcher:VBKT}.
\begin{cor}\label{REM:stable_classes}
Suppose \(X\) as in \Cref{rational-surjectivity-of-forgetful-map-in-Kgroups} satisfies the additional hypothesis that $\oplus_{i>0}\H^{4i}(X;\qq)\neq 0$.  Then there exist infinitely many real vector bundles which carry a $G$-vector bundle structure and belong to pairwise distinct stable classes in \(\widetilde{\KO}(X)\).
\end{cor}
\begin{proof}
  The Atiyah--Hirzebruch spectral sequence yields a group isomorphism \(\KO(X)\otimes\qq \cong \oplus_{i}\H^{4i}(X;\qq)\) (see \cite[\S\,2.4]{AtiyahHirzebruch} for the complex variant \(\K(X)\otimes\qq\cong \oplus_{i}\H^{2i}(X;\qq)\); the real case is analogous \cite[\S\,2.4]{GonZib1}).  Under this isomorphism, \(\widetilde{\KO}(X)\otimes\qq\) gets identified with \(\oplus_{i>0}\H^{4i}(X;\qq)\).  Therefore, the condition \(\oplus_{i>0}\H^{4i}(X;\qq)\neq 0\) is equivalent to the existence of a stable class in \(\widetilde{\KO}(X)\) of infinite order.

  %Now consider a real vector bundle \(E\).   By definition, the bundle $E\oplus\rr^k$ is trivial for some $k$ if and only if the stable class of $E$ in \(\widetilde{\KO}(X)\) is trivial. For $q>1$, the bundle $qE\oplus\rr^k$ might be trivial for some $k$ if $E$ belongs to a stable class of finite order; however if it belongs to a stable class of infinite order then it is never trivial for any $q,k$.
  %
  % So, using our assumption, let us pick a real vector bundle \(E_1 := E\) whose stable class has infinite order.  By \Cref{rational-surjectivity-of-forgetful-map-in-Kgroups}, we can find integers \(q_1> 0\) and \(k_1 \geq 0\) such that \(F_1 := q_1E\oplus\rr^{k_1}\) carries a $G$-vector bundle structure and is non-trivial, by the discussion above.  Inductively, we can construct bundles \(E_i := (q_{i-1} +1) E_{i-1} \) and \(F_i = q_i E \oplus \rr^{k_i}\) such that all bundles \(F_i\) have the same two properties as \(F_1\).  The integers \(q_i\) will form a strictly increasing sequence, ensuring that the stable classes of the bundles \(F_i\) are all distinct.
  %
  % Simplification suggested by referee Y:

  So let us pick a real vector bundle \(E\) whose stable class has infinite order.
  By \Cref{rational-surjectivity-of-forgetful-map-in-Kgroups}, we can find integers \(q > 0\) and \(k \geq 0\) such that \(F := q E\oplus\rr^k\) carries a $G$-vector bundle structure.  Then the Whitney sum \(nF\) also carries a \(G\)-vector bundle structure, for each integer \(n\).
  As the stable class of \(F\) has infinite order, the stable classes of the bundles \(nF\) for different \(n\) are all distinct.
\end{proof}

\subsection{Stabilization over cohomogeneity one spaces}
\label{SUBSEC:stabilization-over-cohomogeneity-one-spaces}
We now apply the characterization of (integral) surjectivity in \Cref{surjectivity-of-forgetful-map-in-Kgroups} to cohomogeneity one spaces, building on a surjectivity result of Carlson.
Recall that the commutator subgroup of a compact connected Lie group \(G\) is semisimple, and that \(G\) is a semi-direct product of this commutator subgroup and a torus [Hilgert \& Neeb, Theorem~12.2.6].  We say that \(G\) satisfies \textbf{Steinberg's assumptions}
if
\begin{equation}\label{eq:Steinbergs-assumptions}
  \text{\parbox{0.8\linewidth}{\(G\) is a compact connected Lie group whose commutator subgroup is a product of simply-connected compact Lie groups and Lie groups \(SO(r)\) with odd~\(r\).}}
\end{equation}
See \cite[Theorem~1.2]{Steinberg:Pittie} for the significance of this condition and equivalent assumptions.

 %  Then

Carlson showed in \cite[Theorem~6.1]{Carlson:coho1} that the genuine forgetful map \(u\colon \K_G^d(M) \to \K^d(M)\) is surjective in each degree~\(d\) provided $M$ is a cohomogeneity one manifold such that all groups in the associated diagram \((G,H,K_-,K_+)\) are connected, \(\pi_1(G)\) is torsion-free, \(K_\pm\) satisfy Steinberg's assumptions \eqref{eq:Steinbergs-assumptions}, and \(\rank G = \max\{\rank K_+,\rank K_-\}\). Together with \Cref{surjectivity-of-forgetful-map-in-Kgroups} (and \Cref{thm:illman}), we get the following result.

\begin{theo}\label{carlson-surjectivity}
  Suppose \(M\) is a cohomogeneity one manifold with all groups in the associated diagram \((G,H,K_-,K_+)\) connected, \(\pi_1(G)\) torsion-free and such that \(K_\pm\) satisfy Steinberg's assumptions \eqref{eq:Steinbergs-assumptions}.  If \(\rank G = \max\{\rank K_+,\rank K_-\}\), then every complex vector bundle over $M$ carries a $G$-vector bundle structure up to stabilization, i.e.\ for every complex vector bundle \(E\) over \(M\) there is an integer \(k\) such that \(E\oplus\cc^k\) carries a $G$-vector bundle structure.
\end{theo}

\begin{rem}\label{REM:carlson_not_necessary}
For a cohomogeneity one $G$-manifold $M$, none of the sufficient conditions in Carlson's result is strictly necessary for the surjectivity of \(u\colon \K_G^0(M) \to \K^0(M)\). For example, $\mathbb{S}^5$ admits infinitely many non-equivalent cohomogeneity one actions with group diagram \((SO(2)SO(3),\zz_2,SO(2),O(2))\), see \cite[p. 343]{GroveZiller:Milnor}; for the surjectivity of $u$ note that $\K^0(\mathbb{S}^5)$ is trivial.
\end{rem}

\begin{rem}
Carlson's result only applies to cohomogeneity one spaces with non-principal orbits of codimension $\geq 2$, since the isotropy group of an exceptional orbit (i.e. an orbit of codimension $1$) must be disconnected.
\end{rem}
We are ready to give the proof of Part~(\ref{THM: converse soul for cohomo 1}) of \Cref{THM: big theorem}.

\begin{proof}[Proof of \Cref{THM: big theorem}, Part~(\ref{THM: converse soul for cohomo 1})]
  Let $E$ be an arbitrary complex vector bundle over a cohomogeneity one manifold $M$ as in Part~(\ref{THM: converse soul for cohomo 1}) of \Cref{THM: big theorem}. By \cref{carlson-surjectivity} there is $\cc^{k'}$ such that $E\oplus\cc^{k'}$ is isomorphic to (the underlying vector bundle of) a $G$-vector bundle $F$. The assumption \(K_\pm/H \cong S^1\) allows us to apply \cref{hop:non-negative-metrics}. Thus the total space of $F$ carries a non-negatively curved metric. Its pullback via the isomorphism $E\oplus\cc^{k'} \cong F$ yields the desired metric on the total space of the complex vector bundle $E\oplus\cc^{k'}$.  This total space can be identified with the total space of the real vector bundle $E\oplus\rr^{k}$ for $k=2k'$; in particular, it is diffeomorphic to $E\times\rr^{k}$.
\end{proof}

  %, with associated diagram \((G,H,K_-,K_+)\)
  %Carlson's \cref{carlson-surjectivity} together with \cref{surjectivity-of-forgetful-map-in-Kgroups} (and \Cref{thm:illman}) imply that there exists a trivial bundle $\cc^{k'}$ such that $E\oplus\cc^{k'}$ is isomorphic to (the underlying vector bundle of) a $G$-vector bundle $F$.

\begin{rem}\label{REM:maximal_rank_iff_positive_characteristic}
The assumption \(K_\pm/H \cong S^1\) implies \(\rank K_+ = \rank K_-=\rank H+1\). In this setting, the condition \(\rank G = \max\{\rank K_+,\rank K_-\}\) is equivalent to  \(\rank G = \rank K_\pm\). The latter implies positivity of the Euler characteristic, see~\cite[Corollary~1.6]{Frank:cohomogeneity}. Cohomogeneity one manifolds with $\chi(M)>0$ have been investigated in \cite{AlekPod:compact,Frank:cohomogeneity,DeVitoKennard:coho1}.
\end{rem}

\begin{rem}\label{REM:integral_Ktheory_circle}
  Carlson also computes both the non-equivariant and the equivariant K-ring of a $G$-manifold with orbit space a circle \cite[Proposition~1.7]{Carlson:coho1}.
  % For example, in the trivial case where $M=G/H \times S^1$ with $G$ acting transitively on the first factor and trivially on the second one, the rings are $\K^0(M)=\K^*(G/H)$ and $\K_G^0(M)=\K_G^*(G/H)$.  (This also follows from the existing Künneth formulas for K-theory.) Hence \emph{any} odd degree cohomology of $G/H$ yields even degree cohomology of the product.
  % So taking a product with just a circle already complicates significantly the corresponding rings of a homogeneous space.
\end{rem}

%%%%%%%%%%%%%%%%%%%%%%%%%%%%%%%%%%%%%%%%%%%%%%%%%%%
%%%%%%%%%%%%%%%%%%%%%%%%%%%%%%%%%%%%%%%%%%%%%%%%%%%%
%%%%%%%%%%%%%%%%%%%%%%%%%%%%%%%%%%%%%%%%%%%%%%%%%%%%

\subsection{Borel equivariant cohomology theories}
\label{SUBSEC:Borel-equivariant-cohomology-theories}

Our next aim will be to relate the forgetful map in genuine equivariant rational K-theory that appears in \Cref{rational-surjectivity-of-forgetful-map-in-Kgroups} to the forgetful map in rational, Borel equivariant cohomology.  As a stepping stone, we will use Borel equivariant K-theory.  So let us quickly recall the construction of these Borel equivariant theories.

The \define{homotopy orbit space} or \define{Borel construction} of a \(G\)-space \(X\) is given by
\begin{equation}\label{eq:homotopy-orbit-space}
  \begin{aligned}
    X_G &:= X\times_{G}\E G\\
    &:=(X\times \E G)/G
  \end{aligned}
\end{equation}
Here, \(\E G\to \B G\) is the universal principal \(G\)-bundle over the classifying space \(\B G\) of \(G\), and \(G\) acts diagonally on \(X\times \E G\). Associated with the universal principal \(G\)-bundle, we have a fibre bundle with fibre \(X\):
\begin{equation}\label{eq:Borel-fibration}
  X \overset{j}\hookrightarrow X_G \xrightarrow{p} \B G
\end{equation}
\cite[Chapter~IV, \S\,3.3]{Borel:Transformation}.
This fibre bundle is usually referred to as the \define{Borel fibration}.

\begin{rem}\label{rem:CW-structure-on-X_G}
  For a finite \(G\)-CW complex \(X\) under a compact Lie group \(G\), the Borel construction can be performed within the framework of spaces homotopy equivalent to CW complexes.  This can be seen from the following facts; we assume \(G\) to be a compact Lie group throughout:
    \begin{compactenum}[(i)]
  \item
    The universal \(G\)-space \(\E G\) can be constructed as a \(G\)-CW complex, see for example \cite[Theorem~1.9]{Lueck:Classifying}.
  \item
    The Cartesian product \(X\times Y\) of two \(G\)-CW complexes has the structure of a \((G\times G)\)-CW complex in a canonical way, at least if \(X\) is a finite \(G\)-CW complex \cite[(H)\,(i)]{Matumoto:Whitehead}.
    (The same statement should be true without any finiteness assumption porvided we work in a ``good'' category of topological spaces like compactly generated weak Hausdorff spaces.  But we have not found a reference for this statement, nor will we need this generality.)
  \item
    For \(X\) and \(Y\) as in the previous point, the \((G\times G)\)-CW complex \(X\times Y\) is \(G\)-homotopy equivalent to a \(G\)-CW complex with respect to the diagonal subgroup \(G\subset G\times G\).  This is a special case of \Cref{thm:G-CW-restriction}.
  \item
    The orbit space of a \(G\)-CW complex is an ordinary CW complex \cite[Proposition~1.6]{Matumoto:Whitehead}.
  \end{compactenum}
  So if we start with a finite \(G\)-CW complex \(X\), the product \(X\times \E G\) has the homotopy type of a \(G\)-CW complex by (i), (ii) and (iii), and by (iv) the Borel construction \(X_G\) is homotopy equivalent to an ordinary CW complex.
\end{rem}

\begin{rem}\label{rem:CW-structure-on-X_G-not-finite}
  It is not possible, on the other hand, to perform this construction within the category of spaces homotopy equivalent to \emph{finite} CW complexes.  Even when \(X\) is a single point and \(G\) is a finite or compact connected Lie group like \(G=\zz_2\) or \(G=S^1\), we have \(X_G = \B G\) homotopy equivalent to \(\rr P^\infty\) or \(\cc P^\infty\), respectively, which both have (co)homology in arbitrarily high degrees.
\end{rem}
For simplicity, we continue to assume that \(G\) is a compact Lie group in the following, and that \(X\) is \(G\)-homotopy equivalent to a finite \(G\)-CW complex.  The \textbf{Borel equivariant rational singular cohomology} of \(X\) is simply defined as \(\H^*_G(X;\qq) := \H^*(X_G;\qq)\).  For example \(\H^*_G(\point;\qq) = \H^*(\B G;\qq)\). Of course, the same definition works for arbitrary coefficients.  But we will deploy no coeffients other than \(\qq\) for singular cohomology in this paper, nor will we consider other equivariant versions of singular cohomology.  So we will refer to this theory simply as \textbf{equivariant cohomology}. The \textbf{Borel equivariant K-theory} of \(X\) is defined analogously as \(\K^*(X_G)\), or, with rational coefficients, as \(\K^*(X_G;\qq)\).  This theory needs to be carefully distinguish from the genuine equivariant K-theory \(\K_G^*(X)\) of \Cref{SUBSEC:genuine-equivariant-K-theory}, as we will see below.

The inclusion of the fibre \(j\) in the Borel fibration \eqref{eq:Borel-fibration} induces a natural transformation in any Borel equivariant theory, which we will refer to as the \textbf{Borel forgetful map}.  The three incarnations we are mainly interested in are:
\begin{align*}
  u^\borel :=&\K^*(j)\colon \K^*(X_G) \to \K^*(X)\\
  u_{\qq}^\borel := &\K^*(j;\qq)\colon \K^*(X_G;\qq) \to \K^*(X;\qq)\\
             &\H^*(j;\qq)\colon \H^*_G(X;\qq)\to \H^*(X;\qq)
\intertext{%
                                 It is difficult to find a universally consistent notation for these maps.  In this section, we want to relate these Borel forgetful maps to the genuine forgetful maps
                                 }
  u \colon &\K^*_G(X)\to \K^*(X)\\
  u_\qq \colon &\K^*_G(X;\qq)\to \K^*(X;\qq)
\end{align*}
introduced in \Cref{SUBSEC:genuine-equivariant-K-theory}, so we will mostly stick to the decorated versions of \(u\).  In \Cref{sec:cohomo1_rational_cohomo}, where we deal exclusively with singular cohomology but study the surjectivity of various induced maps, we will use the notation \(\H^*(j) := \H^*(j;\qq)\).  The restriction of $\H^*(j)$ to all even degrees or to a single degree $i$ will be denoted $\H^\even(j)$ or $\H^i(j)$, respectively.

\subsection{Surjectivity in genuine versus Borel equivariant K-theory}
\label{SUBSEC:surjectivity-in-genuine-versus-Borel-equivariant-K-theory}
There is a canonical, well-studied, natural and degree-preserving transformation \(\gamma_G\colon \K_G^*(X)\to \K^*(X_G)\): it sends the class of a \(G\)-vector bundle \(F\) over \(X\) to the class of the vector bundle \((F \times \E G)/G\) over \(X_G = (X\times \E G)/G\).  We will observe in the proof of \Cref{u-surjective-versus-u-borel-surjective} that this transformation is compatible with the two forgetful transformations \(u\) and \(u^\borel\) in the sense that the following square commutes:
\begin{equation}\label{diag:u-gamma-u-borel}
  \begin{aligned}
    \xymatrix{
      {\K^*_G(X)} \ar[r]^{u} \ar[d]^{\gamma_G} & {\K^*(X)} \ar[d]^{\gamma_1}_{=} \\
      {\K^*(X_G)} \ar[r]^{u^\borel} & {\K^*(X)}
    }
  \end{aligned}
\end{equation}
The subscript \(1\) of \(\gamma_1\) here indicates the trivial group; \(\gamma_1\) is the identity.  For non-trivial \(G\), the transformation \(\gamma_G\) is \emph{not} generally an isomorphism. The Atiyah--Segal Completion Theorem asserts that, for any compact Lie group \(G\) and any finite \(G\)-CW complex \(X\), the map \(\gamma_G\) can be identified with the completion of \(\K_G^*(X)\) at the rank zero ideal \(\I_G := \ker(\rank\colon \R(G)\to \zz)\), i.e.\ at the kernel of the rank homomorphism.\footnote{
  We refer to the homomorphism \(\R(G)\to \zz\) that sends a representation to its dimension as the ``rank homomorphism'' since we view a \(G\)-representation as a \(G\)-vector bundle over a point, and the dimension of the former coincides with the rank of the latter.
}  Explicitly, writing \(\I_G^n\) to denote the \(n^\text{th}\) power of the ideal \(\I_G\), \(\gamma_G\) induces an isomorphism as follows:
\begin{equation}\label{eq:ASC-lim}
  \lim_n\left(\K^*_G(X) \middle/ (\I_G^n\cdot \K^*_G(X)) \right) \xrightarrow{\;\cong\;} \K^*(X_G)
\end{equation}
In particular,  \(\K^*(\point_G) = \K^*(\B G)\) is isomorphic to the \(\I_G\)-completion of \(\K_G^*(\point) \cong \R(G)\), recall $\K_G^1(\point)=0$.
\begin{example}[\(X = \point, G=S^1\)]
  The complex representation ring \(\R(S^1)\) can be identified with \(\zz[\alpha,(\alpha+1)^{-1}]\), where \(\alpha = z-1\) for the standard one-dimensional representation \(z\), and \(\I_{S^1}\) is the ideal generated by \(\alpha\) under this identification.  The K-ring of \(\B S^1 \simeq \cc P^\infty\) can be identified with the formal power series ring \(\zz\llbracket \alpha \rrbracket\), where \(\alpha = x-1\) for the tautological line bundle \(x\) over \(\cc P^\infty\).  Under these identifications, \(\gamma_{S^1}\) is the completion \(\zz[\alpha,(\alpha+1)^{-1}] \to \zz\llbracket \alpha \rrbracket\).   Note in particular that \(\gamma_{S^1}\) is not surjective.
\end{example}

The completion theorem also holds with rational coefficients, as we will quickly verify:
\begin{prop}[Rational completion theorem]\label{prop:rational-ASC}
  For any compact Lie group \(G\), and any finite \(G\)-CW complex \(X\),
  we have a canonical isomorphism
  \begin{equation}\label{eq:ASC-lim-rational}
    \lim_n\left(\K^*_G(X;\qq)\middle/ (\I_G^n\cdot \K^*_G(X;\qq)) \right) \xrightarrow{\;\cong\;} \K^*(X_G;\qq)
  \end{equation}
\end{prop}
\begin{proof}
  Note that we cannot obtain \eqref{eq:ASC-lim-rational} simply by tensoring \eqref{eq:ASC-lim} with \(\qq\), because tensoring does not commute with the inverse limit, and because the (non-equivariant version of) the isomorphism \eqref{eq:finite-rationalization} does not hold for the generally infinite complex \(X_G\) (see \Cref{rem:CW-structure-on-X_G-not-finite}).  Instead, we need to start from a slightly stronger formulation of the completion theorem in terms of pro-groups.
  Recall from \Cref{rem:CW-structure-on-X_G} that \(X_G\) can be obtained as a CW complex by choosing a \(G\)-CW complex \((X\times \E G)'\) that is \(G\)-homotopy equivalent to \(X\times \E G\), and dividing out the \(G\)-action.  Let \(X_{G,\alpha}\) denote the filtration of \(X_G\) by its finite subcomplexes.  By \cite[Theorem~2.1]{AtiyahSegal} or, more precisely, by \cite[Corollary~2.1]{AHJM}, we have an isomorphism of pro-groups
  \begin{equation}\label{eq:ASC-pro}
    \left\{\K^*_G(X)\middle/ (\I_G^n\cdot \K^*_G(X)) \right\}_n \xrightarrow{\;\cong\;}  \{ \K^*(X_{G,\alpha})\}_\alpha
  \end{equation}
  Here, the pro-group on the left is simply indexed by the natural numbers, while the pro-group on the right is indexed by the directed system of finite subcomplexes of \(X_G\).  We have identified the groups \(\K^*_G((X\times \E G)'_\alpha)\) appearing in the sources with the groups \(\K^*(X_{G,\alpha})\), using \eqref{eq:K_G-of-free}.

  The isomorphism \eqref{eq:ASC-lim} is obtained by passing to the inverse limit. As the pro-group on the left satisfies the Mittag--Leffler condition, the derived limit functor \(\lim^1\) of the pro-group on the left vanishes.  It follows that \(\lim^1\) also vanishes for the pro-group on the right.  Anderson's generalization of Milnor's exact sequence \cite[Proposition~12]{ArakiYosimura}  therefore allows us to identify the inverse limit of the pro-group on the right with \(\K^*(X_G)\).

  We now tensor \eqref{eq:ASC-pro} with \(\qq\).
  By assumption, \(X\) is a finite \(G\)-CW complex, so we can apply \eqref{eq:finite-rationalization}.  Similary, on the right the non-equivariant version of \eqref{eq:finite-rationalization} applies to each of the finite CW complexes \(X_{G,\alpha}\).  Thus, altogether, we obtain an isomorphism of pro-groups
  \begin{equation*}%\label{eq:ASC-pro-rational}
    \left\{\K^*_G(X;\qq)\middle/ (\I_G^n\cdot \K^*_G(X;\qq)) \right\}_n \xrightarrow{\;\cong\;}  \{ \K^*(X_{G,\alpha};\qq)\}_\alpha
  \end{equation*}
   from which we can conclude as before.
\end{proof}

The completion theorem allows us to relate the genuine forgetful map and the Borel forgetful map as follows:
\begin{prop}\label{u-surjective-versus-u-borel-surjective}
  Fix a degree~\(d\).   For a finite \(G\)-CW complex \(X\), the genuine forgetful map \(u\colon \K^d_G(X)\to \K^d(X)\) is surjective if and only if the Borel forgetful map \(u^\borel\colon \K^d(X_G)\to \K^d(X)\) is surjective.
  The same is true with rational coefficients.
\end{prop}
\begin{proof}
  We first verify that square~\eqref{diag:u-gamma-u-borel} commutes.  Let \(p_X\colon X\times \E G\to X\) denote the projection onto the first factor. As the \(G\)-action on \(X\times \E G\) is free, by \eqref{eq:K_G-of-free} we have a canonical isomorphism \(\K^*(X_G) \cong \K^*_G(X\times \E G)\).  The map \(\gamma_G\) can be identified with the composition of \(p_X^*\) with the inverse of this isomorphism, as indicated by the left triangle in the following expanded version of square~\eqref{diag:u-gamma-u-borel}:
  \begin{equation*}
    \xymatrix{
      {\K_G^*(X)} \ar[dd]_{\gamma_G}\ar[dr]_{p_X^*} \ar[rrr]^{u} &&& {\K^*(X)} \ar[dl]^{p_X^*}_{\cong} \ar@{=}[dd]^{\gamma_1}\\
        & {\K_G^*(X\times \E G)} \ar[r]^{u} & {\K^*(X\times \E G)} \ar[dr]^{i^*}_{\cong} \\
      {\K^*(X_G)} \ar[ur]^{\cong}\ar[urr]_{\pi^*} \ar[rrr]_{j^* = u^\borel} &&& {\K^*(X)}
    }
  \end{equation*}
  In this expanded square, \(\pi\colon X\times \E G \to X_G\) is the canonical projection onto the orbit space, and \(i\colon X\to X\times \E G\) is the inclusion of \(X\) at a basepoint of \(\E G\).
  All subdiagrams of the square are easily seen to commute.  So square \eqref{diag:u-gamma-u-borel} commutes as claimed.

  All morphisms in \eqref{diag:u-gamma-u-borel} preserve degrees.  As the vertical map on the right is the identity, the surjectivity of \(u\) in a given degree clearly implies the surjectivity of \(u^\borel\) in the same degree.

  For the converse implication, recall from \Cref{SUBSEC:genuine-equivariant-K-theory} that the \(\R(G)\)-module structure on \(\K^*_G(X)\) is induced by the pullback map \(p^*\colon \K^*_G(\point)\to\K^*_G(X)\) along the projection to a point and by the canonical identification \(\K^*_G(\point)\cong \R(G)\), where \(\R(G)\) is concentrated in degree zero.
  If we equip \(\K^*(X_G)\) and \(\K^*(X)\) with the \(\R(G)\)-module structures induced by the maps \(\gamma_G\) and \(u\), then clearly \eqref{diag:u-gamma-u-borel} becomes a square of \(\R(G)\)-module homomorphisms.  As we said, \(\gamma_G\) is the \(\I_G\)-adic completion morphism.  (Here and above, we follow the common convention that `complete' means `complete and Hausdorff'.)

  The \(\R(G)\)-module \(\K^*(X)\) is also complete.  In fact, \(\I_G^n\cdot \K^*(X)=0\) for all \(n\).  Indeed, under the identifications  \(\K^*_G(\point)\cong \R(G)\) and \(\K^*(\point)\cong \zz\), the forgetful map \(u\colon \K^*_G(\point)\to \K^*(\point)\) gets identified with the rank homomorphism \(\rank\colon \R(G)\to \zz\), so \(up^*(\I_G) = p^*u(\I_G) = 0\), and hence \(\I_G^n\cdot\K^*(X) = (u p^*\I_G)^n\cdot \K^*(X) = 0\) as claimed.

  By the universal property of completion \cite[23H]{Matsumura1980},
  it follows that \(u^\borel\) is the \(\I_G\)-adic completion of \(u\). Explicitly, if we identify the completion with the inverse limit \(\lim_n \K^*_G(X)/(\I_G^n\cdot \K^*_G(X))\) as in \eqref{eq:ASC-lim},
  this means that the map \(u^\borel\) in degree \(d\) is simply given by
  \begin{align*}
    \lim_n \left(\K^d_G(X) \middle/ (\I_G^n\cdot \K^d_G(X)) \right) &\to \K^d(X)\\
    (x_1,x_2,x_3,\dots) &\mapsto u(x_n),
  \end{align*}
  for any fixed \(n\). This shows that the surjectivity of \(u^\borel\) in degree \(d\) implies the surjectivity of~\(u\) in degree \(d\).

  The rational statement can be obtained in the same way, using \eqref{eq:ASC-lim-rational} in place of  \eqref{eq:ASC-lim}.
\end{proof}

\subsection{Surjectivity in K-theory versus cohomology}
\label{SUBSEC:surjectivity-in-K-theory-versus-cohomology}
The Chern character can be constructed as a natural transformation of \(\zz_2\)-graded (or, equivalently, \(2\)-periodic \(\zz\)-graded) cohomology theories \(\K^*(-) \to \H^{**}(-;\qq)\) on the category of finite CW complexes.  Here \(\H^{**}(-;\qq)\) denotes the \emph{product} (not the direct sum) of all rational cohomology groups, regarded as a \(\zz_2\)-graded theory with the product of all even-degree cohomology groups in degree zero, and the product of all odd-degree cohomology groups in degree one.  The induced natural transformation on rational K-theory
\begin{equation}\label{eq:Chern-character}
  \ch\colon \K^*(X;\qq) \to \H^{**}(X;\qq)
\end{equation}
is known to be an isomorphism on any finite CW complex \(X\) \cite[Theorem in \S\,2.4]{AtiyahHirzebruch}.  In fact:
\begin{prop}\label{prop:Chern-character-iso}
  The rational Chern character \eqref{eq:Chern-character} can be defined and is a ring isomorphism for any CW complex \(X\).
\end{prop}

\begin{rem}[Integral failure]
For a finite CW-complex $X$ it follows that $\K^0(X;\qq)$ is isomorphic to $\H^\even(X;\qq)$ as abelian groups, and so is $\K^1(X;\qq)$ to $\H^\odd(X;\qq)$. This does not hold integrally: take for example the homogeneous (exceptional symmetric) space $M:=G_2/SO(4)$. The cohomology groups are $\H^\even(M;\zz)=\zz\oplus\zz\oplus\zz_2\oplus\zz$ in degrees $0,4,6$ and $8$ respectively and $\H^\odd(M;\zz)=\zz_2$ in degree $3$, see \cite[17.3]{BoHi58}. In contrast, we know from \cite[Theorem 3.6]{GonZib1} that $\K^0(M)$ is torsion free (and hence must equal $\zz^3$) and $\K^1(M)=0$.  The claim in \cite[\S\,4.5]{Park:K-theory} that \(\K^*(M)\) and \(\H^*(M;\zz)\) are isomorphic as groups for arbitrary compact manifolds \(M\) is incorrect.
\end{rem}
\begin{rem}[Čech variant]\label{rem:Karoubis-Chern-character}
  Recall our notation \(\K^*_{\text{geom}}\) from \Cref{warning}.
  Karoubi states in \cite[Theorem~3.25]{Karoubi:K-theory} that \(\K^*_{\text{geom}}(X)\otimes\qq \cong \smash{\check{\H}}^*(X;\qq)\) for any compact Hausdorff space \(X\), where \(\smash{\check{\H}}^*\) denotes Čech cohomology.  This does not contradict the results involving singular cohomology here, as Čech cohomology differs from singular cohomology in the same way that \(\K_{\text{geom}}^*\) differs from \(\K^*\).
\end{rem}
\begin{proof}[Proof of \Cref{prop:Chern-character-iso}]
  We include a proof for lack of reference.  The statement is well-known, and is a typical application of Brown representability, in the version proved by Adams \cite[Theorem 1.3 and Addenda 1.4 and 1.5]{Adams:Brown}.  The following discussion is based on the presentation in \cite[first page]{Neeman:Brown}.\footnote{
    Neeman describes cohomology theories as functors on the category of (finite) spectra; the equivalence with more down-to-earth descriptions in terms of (finite) CW complexes follows from the fact that the full subcategory of finite spectra in the category of all spectra is equivalent to the Spanier--Whitehead category on finite CW complexes.
  }
  By the first part of the representablity theorem, both cohomology theories \(\K^*(-;\qq)\) and \(\H^{**}(-;\qq)\) are representable by spectra.
  By the second part of the representability theorem, the Chern character is induced by a (non-unique) morphism between these spectra.  In particular, the Chern character extends to a natural transformation of the given cohomology theories viewed as cohomology theories on arbitrary CW complexes.  Moreover, by the observation after Remark~3.2 of \cite{Neeman:Brown}, the morphism of spectra is an isomorphism, hence so is the corresponding natural transformation.
  Alternatively, given any appropriate construction of the rational Chern character on arbitrary CW complexes, we could again use Anderson's generalization of Milnor's short exact sequence \cite[Proposition~12]{ArakiYosimura} to conclude that it is an isomorphism in general from the fact that it is an isomorphism on finite complexes.
\end{proof}

\begin{prop}\label{surjectivity_of_Heven_iff_K0}
For a finite \(G\)-CW complex \(X\), the rational genuine forgetful map \(\K^0_G(X;\qq)\to \K^0(X;\qq)\) is surjective if and only if the forgetful map in even-degree Borel equivariant rational singular cohomology \(\H^\even_G(X;\qq)\to \H^\even(X;\qq)\) is surjective.  Likewise, the forgetful map \(\K^1_G(X;\qq)\to \K^1(X;\qq)\) is surjective if and only if the forgetful map \(\H^\odd_G(X;\qq)\to \H^\odd(X;\qq)\) is surjective.
\end{prop}
\begin{proof}
  Recall from \Cref{u-surjective-versus-u-borel-surjective} that, in K-theory, the rational genuine forgetful map \(u_\qq\) is surjective in degree \(d\) if and only if the rational Borel forgetful map \(u^\borel_\qq\) is surjective in the same degree \(d\).  It suffices to observe that, in addition,
  \(u^\borel_\qq = \K^*(j;\qq)\) is surjective in degree zero if and only if \(\H^\even(j)\colon \H^\even(X_G;\qq)\to \H^\even(X;\qq)\) is surjective, and is surjective in degree one if and only if  \(\H^\odd(j)\colon \H^\even(X_G;\qq)\to \H^\odd(X;\qq)\) is surjective.  This is immediate from the following commutative diagram, \Cref{prop:Chern-character-iso} and the way in which the Chern character preserves even\slash odd degrees.
  \[\xymatrix@C=9em{
      {\K^*(X_G;\qq)} \ar[r]^{\K^*(j;\qq)}  \ar[d]^{\chern}& {\K^*(X;\qq)} \ar[d]^{\chern}\\
      {\H^{**}_G(X;\qq)} \ar[r]^{\H^{**}(j;\qq)} & {\H^{**}(X;\qq)}
    }\]
  Note that the direct product of cohomology groups \(\H^{**}(X;\qq)\) here is really just a direct sum, as \(X\) is homotopy equivalent to a finite CW complex by assumption and by \Cref{thm:G-CW-restriction}.
\end{proof}

%%%%%%%%%%%%%%%%%%%%%%%%%%%%%%%%%%%%%%%%%%%%%%%%%%%%%%%%%%%

\section{Rational homotopy and equivariant cohomology of biquotients and cohomogeneity one spaces}
\label{sec:cohomo1_rational_cohomo}

Having established the connection between the forgetful maps in equivariant rational K-theory and in equivariant rational Borel cohomology in Proposition \ref{surjectivity_of_Heven_iff_K0}, this section is now devoted to an in-depth investigation of the map on the cohomological side. Nonetheless, this section may be of independent interest, as it does provide several generalizations of the classical term ``equivariant formality''.

The main tool we use is rational homotopy theory. \Cref{secpre} contains the needed preliminaries and discusses pure spaces. In \Cref{SEC:surjectivity_cohomology} we study the surjectivity properties we are interested in. This is applied to homogeneous and cohomogeneity one spaces in \Cref{SEC:applications_biquotients_cohomogeneity1}, where we provide the proofs of Theorems \ref{THM:surjectivity of homogeneous bundles}, \ref{THM:double_stabilization_homogeneous}, \ref{THM:rational_stabilization_cohomo1} and Part~(\ref{THM: rational converse soul for cohomo 1}) of \Cref{THM: big theorem}.

\subsection{Preliminaries and main tools}\label{secpre}

This subsection is divided in three parts: first we review basic notions from rational homotopy theory, then we illustrate them in the case of biquotients and finally we study pure models.

\subsubsection{Basics, ellipticity, formality, pureness and spherical cohomology}

We recall some basic definitions from rational homotopy theory which we shall draw on in the subsequent sections. All the details can be found in \cite{FHT01,FOT08}.

In the following discussion $X$ will always denote a nilpotent path-connected CW complex. Additional assumptions on such a space $X$ (e.g.\ rational ellipticity) will be made explicitly. Recall that $X$ is said to be nilpotent if $\pi_1(X,\point)$ is a nilpotent group and acts nilpotently (via covering transformations) on each $\pi_n(\tilde X,\point)$, where $\tilde X$ denotes the universal covering space of $X$.

\medskip

A \define{Sullivan algebra} $(\Lambda V,\dif)$ is a free commutative graded algebra $\Lambda V:=\wedge V^\odd\otimes \qq[V^\even]$ on a graded \(\qq\)-vector space $V$ (concentrated in positive degrees), together with a differential. The differential $\dif\colon V\to \Lambda V$ is a map of degree $1$ extended to $\Lambda V$ as a derivation. It satisfies a nilpotence condition  \cite[p.~138]{FHT01}. We denote by $\Lambda^q V$ the linear span of elements of the form $v_1\wedge\dots\wedge v_q$ with $v_i\in V$ and $q\geq 2$, and we denote $\Lambda^{\geq q} V:=\oplus_{i\geq q}\Lambda^i V$, see \cite[p.~140]{FHT01}. The algebra is called \define{minimal} if $\im \dif \in \Lambda^{\geq 2} V$, i.e.~if the differential maps into the subspace generated by the so-called decomposable elements.
The algebra is said to be \define{contractible} if it is of the form $(\Lambda (U\oplus\delta U),\delta)$ with $\delta:U\to\delta U$ an isomorphism.
Any Sullivan algebra is a tensor product of a minimal one and a contractible one \cite[Theorem~14.9, p.~187]{FHT01}.

A (minimal) Sullivan algebra $(\Lambda V,\dif)$ is a \define{(minimal) Sullivan model} of the cochain algebra $(A,\dif)$ if there is a \define{quasi-isomorphism}, i.e.~a morphism $(\Lambda V,\dif)\xto{} (A,\dif)$ of cochain algebras inducing an isomorphism in cohomology. A (minimal) Sullivan algebra is a \define{(minimal) Sullivan model} of a space $X$ if it is a (minimal) Sullivan model for the cochain algebra $\APL(X)$ of polynomial differential forms on $X$, see \cite[Chapter 10, p.~115]{FHT01}.

A \define{connected cochain algebra}, i.e.~a commutative differential graded algebra $(A=A^{\geq 0},\dif)$ with $\H^0(A)=\qq$, has a minimal Sullivan model unique up to isomorphism \cite[Theorem~2.24, p.~64]{FOT08}. Hence any path-connected space $X$ possesses a minimal Sullivan model \((\Lambda V, \dif)\) unique up to isomorphism.
\begin{theo}[Main theorem of rational homotopy theory {\cite[Theorem~2.50, p.~75]{FOT08}}]\label{thmmainrht}
  Let $(\Lambda V,\dif)$ be a minimal model of a space $X$ with $\H^*(X;\qq)$ finite-dimensional in each degree. Then
  \[
  V^{\geq 2}\cong \Hom(\pi_{\geq 2}(X),\qq)
  \]
\end{theo}

A minimal Sullivan algebra $(\Lambda V, \dif)$ is called \define{elliptic} if both $V$ and $\H(\Lambda V, \dif)$ are finite-dimensional. Its \define{Euler characteristic} is defined as the alternating sum of the ranks of the ranks of the cohomology groups, and its \define{formal dimension} is the largest degree $d$ in which $\H^d(\Lambda V, \dif)$ is non-trivial.
A space $X$ with an elliptic minimal model is called \define{rationally elliptic}. If, additionally, the Euler characteristic $\chi(X)$ is positive, then the algebra respectively the space are called \define{positively elliptic}. In view of Theorem~\ref{thmmainrht} rational ellipticity of a space $X$ is equivalent to finite-dimensional rational cohomology and the property that from some degree on all homotopy groups are finite.

Beside the usual (cohomological) Euler characteristic used above, in the rationally elliptic situation we can also make the following definition: The \define{homotopy Euler characteristic} of an elliptic Sullivan algebra $(\Lambda V,\dif)$ is given by
$$
\chi_\pi(\Lambda V,\dif):= \dim V^\odd - \dim V^\even
$$
For a rationally elliptic space $X$ its homotopy Euler characteristic $\chi_\pi(X)$ is defined as the one of a Sullivan model of $X$.
The homotopy Euler characteristic does not depend on the choice of the Sullivan model.\COMMENT{MZ: I think this follows from the tensor product decomposition of a Sullivan model into a minimal one and a contratible one quoted above, but we don't need to write this.}
In particular, by \Cref{thmmainrht} it agrees with the alternating sum of the dimensions of rational homotopy groups (for say simply-connected $X$). Moreover, with the given sign convention, it is known that both the homotopy Euler characteristic and the usual Euler characteristic of a rationally elliptic space are non-negative and one is equal to zero if and only if the other one is not \cite[Proposition~32.10, p.~444]{FHT01}.

A minimal Sullivan algebra $(\Lambda V,\dif)$ is called \define{formal} if it comes with a quasi-isomorphism $(\Lambda V,\dif)\to \H(\Lambda V,\dif)$. A space $X$ is formal if so is its minimal Sullivan model.

A Sullivan algebra $(\Lambda V,\dif)$ admits a \define{two-stage decomposition} if it admits a homogeneous decomposition $V=V_0\oplus V_1$ such that
\begin{itemize}
\item[-] $\dif V_0=0$, and
\item[-] $\dif V_1\In \Lambda V_0$
\end{itemize}
A Sullivan algebra $(\Lambda V,\dif)$ is \define{pure} if it admits a two-stage decomposition with \(V_0=V^\even\) and \(V_1=V^\odd\).
(Unlike \cite{FHT01}, we do not insist that \(V\) be finite-dimensional for a pure Sullivan algebra.)
A space $X$ with a two-stage/pure model is called two-stage/pure. (In this case also a minimal model will be two-stage/pure.)
As any Sullivan model contains a minimal one as a tensor factor, any pure Sullivan algebra has a pure minimal model.

\begin{example}\label{EX:pure_elliptic}
Homogeneous spaces and biquotients are pure, rationally elliptic spaces, see \Cref{biq_pure}. In \Cref{props1s1} we verify the pureness of those cohomogeneity one manifolds with the spheres $K_\pm/H$ being odd dimensional. Recall that due to \cite{GH87} cohomogeneity one manifolds are rationally elliptic. More generally, pure spaces of finite-dimensional cohomology are rationally elliptic, see \Cref{rem:pure-ellipticity}.
\end{example}

On a pure elliptic Sullivan algebra $(\Lambda V,\dif)$ there exists, additionally to the upper grading in the cohomology ring $\H(\Lambda V,\dif)$,  the \define{lower grading}
\begin{equation}\label{EQ:lower_grading}
\H(\Lambda V,\dif)=\bigoplus_{i\geq 0} \H_i(\Lambda V,\dif)
\end{equation}
where $\H_i(\Lambda V,\dif)$ is the subspace representable by cocycles in $\Lambda (V^\even)\otimes \Lambda^i (V^\odd)$, see \cite[p.~435]{FHT01}.
These subspaces are invariant under quasi-isomorphism.  That is, given a quasi-isomorphism of pure Sullivan algebras \((V,\dif)\to (W,\dif)\), the subspace \(\H_i(\Lambda V,\dif)\subset \H(\Lambda V,\dif)\) agrees with subspace \(\H_i(\Lambda W,\dif)\subset \H(\Lambda W,\dif)\) under the identification \(\H(\Lambda V,\dif)\cong \H(\Lambda W,\dif)\).  (This again follows from the decomposition of any Sullivan algebra as a tensor product of a minimal one and a contractible one.)  By definition, we have the identities
\begin{equation}\label{EQ:even_odd_upper_lower}
    \H_\even(\Lambda V,\dif)=\H^\even(\Lambda V,\dif),\qquad \H_\odd(\Lambda V,\dif)=\H^\odd(\Lambda V,\dif)
\end{equation}

\begin{rem}\label{rem:pure-ellipticity}
  As mentioned in \Cref{EX:pure_elliptic}, a pure space of finite-dimensional cohomology is automatically rationally elliptic. For this, let \((\Lambda V,\dif)\) be a pure minimal model with finite \(\H(\Lambda V,\dif)\); we need to show that \(\dim V<\infty\). Note that finite-dimensionality of \(\H_0(\Lambda V,\dif)\) and minimality imply that \(\dim V^\even<\infty\), hence that \(\Lambda V^\even\) is noetherian and hence that the ideal in \(\Lambda V^\even\) generated by \(\dif(\Lambda V^\odd)\) is equal to the ideal generated by \(\dif(\Lambda V^\odd_f)\) for some finite-dimensional subspace \(V^\odd_f\subset V^\odd\). Now observe that \(\H_1(\Lambda V,\dif)\) contains a sub-vector space isomorphic to a complement of \(V^\odd_f\) in \(V^\odd\). So finite-dimensionality of \(\H_1(\Lambda V,\dif)\) implies finite-dimensionality of \(V^\odd\).
\end{rem}

Let \((\Lambda V,\dif)\) be a minimal Sullivan algebra. We will call \define{spherical cohomology} the set of all cohomology elements for which there exists an (isomorphic) minimal Sullivan model \((\Lambda V',\dif)\) in which they are represented by an element of \(V'\). As the linear part of an isomorphism of minimal Sullivan models is an isomorphism on the underlying vector spaces, (non-trivial) spherical cohomology cannot be represented by a decomposable element (i.e.~by an element in \(\Lambda ^{\geq 2} V\)). If \((\Lambda V,\dif)\) is a minimal model which is additionally pure, then the subalgebra generated by all spherical cohomology classes of even-degree coincides with \(\H_0(\Lambda V,\dif)\).

In general, minimal Sullivan models are only unique up to isomorphism and spherical cohomology may be larger than the subspace of $V$ of closed forms. For example, $(\Lambda \langle a,x,y\rangle, \dif)$ with $\deg a=2$, $\deg x=\deg y=3$ and $\dif x=\dif y=a^2$ has spherical cohomology generated by $a$ and $x-y$. (The model clearly is isomorphic to $(\Lambda \langle a,x,y\rangle, \dif')$ with $\dif' a=\dif' x=0$, $\dif' y=a^2$.) An isomorphism on models induces one on cohomologies, by which we identify corresponding elements.

\subsubsection{Lie groups and biquotients}\label{SEC:preliminaries_lie_biquotients}

Recall that an H-space is a connected topological space with a ``multiplication'' having a neutral element up to homotopy.
Let us quickly review minimal models of H-spaces. %, or say, in our case, quotients of compact connected Lie groups.
This relies heavily on the following:

\begin{theo*}[Hopf]
  For an H-space \(X\) such that $\H_*(X;\qq)$ is finite-dimensional in each degree, $\H^*(X;\qq)$ is a free commutative graded algebra.
\end{theo*}

Since H-spaces are \define{simple spaces}, i.e.~their fundamental group is abelian acting trivially on higher homotopy groups, it follows that the free algebra $(\Lambda V_X,0):=(\H^*(X;\qq),0)$  constitutes a minimal Sullivan model of $X$. If the formal dimension of $X$ (i.e. of $(\Lambda V_X,0)$) is finite it also follows that $\H^*(X;\qq)$ as an algebra is generated in odd degrees, i.e.~$V_X=V_X^\odd$, as otherwise cohomology could not be finite-dimensional.

A compact connected Lie group $G$ is an H-space, and its minimal Sullivan model $(\Lambda V_G,0)$ is known to satisfy $\dim V_G=\rk G$. Note that this directly yields a model for the classifying space $\B G$. Indeed, since $V_G$ is dual to rational homotopy groups, from the long exact sequence of rational homotopy groups of the classifying fibration $G\hto{} \E G\to \B G$ together with the contractibility of $\E G$, we derive that $V_G^{+1}=V_{\B G}$. That is, a model of $\B G$ can be obtain from the model of $(\Lambda V_G,0)$ of $G$ by a degree shift of $+1$ on $V_G$. Since then $V_{\B G}=V_G^{+1}$ is concentrated in even degrees, also the differential upon it is necessarily $0$ and a minimal Sullivan model for $\B G$ is given by $(\Lambda V_G^{+1},0)$. Observe that both $G$ and $\B G$ are formal.

Recall the notion of biquotient. Let $G$ be a compact Lie group and let $H\In G\times G$ be a closed subgroup. Then $H$ acts on $G$ on the left by $(h_1,h_2)\cdot g=h_1gh_2^{-1}$. The orbit space is called the \define{biquotient} $\biq{G}{H}$ of $G$ by $H$. If the action of $H$ on $G$ is free then $\biq{G}{H}$ possesses the structure of a smooth manifold and there is an associated principal $H$-bundle $H\hto{}G \to \biq{G}{H}$. This is the only case we shall consider, and moreover we assume that $G$ and $H$ (and consequently $\biq{G}{H}$) are connected. It is well known that biquotients are rationally elliptic spaces, and in particular nilpotent.  We also know that
\begin{equation}\label{EQ:parity_dim_biquotient}
\dim \biq{G}{H}=\dim G-\dim H\equiv\rk G-\rk H \mod 2
\end{equation}
since $\H^*(G)$ and $\H^*(H)$ are freely generated by $\rk G$ and $\rk H$ many generators in odd degrees. %The spaces are formal, the cohomology algebras are minimal models and their formal dimension (in this case just the sum of the ranks of the generators) equals the dimension of the groups.
Recall from \cite[Theorem~2.75, p.~85]{FOT08} that:
\begin{equation}\label{EQ:positive_elliptic_biquotients}
    \H^\odd(\biq{G}{H})= 0\Leftrightarrow \rk G=\rk H
\end{equation}
Clearly, the category of biquotients contains the one of homogeneous spaces: for a subgroup \(H\subset G\), the homogeneous space \(G/H\) can be identified with the biquotient \(\biq{G}{(\{e\}\times H)}\), where \(\{e\}\) denotes the trivial subgroup.

\smallskip

Let us now review a model of $\biq{G}{H}$. In \cite[p.~3]{Kap}, \cite{Sin93}, \cite[Section~3.4.2, p.~137]{FOT08} the reader may find the construction of the following (usually highly non-minimal) Sullivan model of $\biq{G}{H}$. From now on we will use the following notation: if the elements $x_1,\dots, x_n$ form a basis of a vector space $V$ then we will write $\langle x_1,\dots, x_n\rangle$ or $\langle x_i\rangle_{1\leq i\leq n}$ for $V$.

\begin{theo}\label{Sullivan-for-biquotient}
A Sullivan model of the %simply-connected
biquotient $\biq{G}{H}$ of compact connected Lie groups $G$, $H$, is given by
\begin{align*}
(\Lambda (V_{\B H}\oplus V_G),\dif)
\end{align*}
with $(\Lambda V_{\B H},0)$ a model for $\B H$ and $(\Lambda V_G,0)$ with $V_G=\langle q_1, \dots, q_k\rangle$ a model for $G$ with $k=\rk G$. The differential $\dif$ is defined by $\dif|_{V_{\B H}}=0$ and by
\begin{align*}
\dif(q_i)=\H^*(\B j)(x_i\otimes 1- 1\otimes y_i)
\end{align*}
where $\H^*(\B j)\colon \H^*(\B G\times \B G)\to{}  \H^*(\B H)$ is induced by the inclusion map $j\colon H\hto{} G\times G$ and $x_i$ respectively $y_i$ are formal copies of the $q_i$.
\end{theo}

Observe that in the model $(\Lambda V,\dif):=(\Lambda (V_{\B H}\oplus V_G),\dif)$ we have $V^\even=V_{\B H}$, $V^\odd=V_G$, thus:
\begin{equation}\label{EQ:homotopy_Euler_biquotient}
\chi_\pi(\biq{G}{H})=\dim V_{G}-\dim V_{\B H}=\rk G-\rk H
\end{equation}
Moreover, the model is clearly two-stage with $V_0=V_{\B H}$ (concentrated in even degrees) and $V_1=V_G$ (necessarily concentrated in odd degrees). Hence we may draw the obvious and well-known conclusion \cite[Proposition 7.1]{BelKap03}:
\begin{cor}\label{biq_pure}
Such a biquotient is a pure space, i.e.~the model above is pure.
\end{cor}

\subsubsection{Pure spaces and even-degree cohomology}
Let us now provide the main tools we will use in \Cref{SEC:surjectivity_cohomology,SEC:applications_biquotients_cohomogeneity1}. Our main focus lies on pure spaces, as these will comprise the examples we are interested in (see \Cref{EX:pure_elliptic}).

Recall the bigrading of the cohomology of a pure Sullivan algebra $(\Lambda V,\dif)$ described in \eqref{EQ:lower_grading}. For later applications we are interested in the size of $\H_0(\Lambda V,\dif)$, i.e. those cohomology elements represented by $\Lambda V^\even$. It is clear from \eqref{EQ:even_odd_upper_lower} that
\begin{align*}
\H_0(\Lambda V,\dif)\In \H^\even(\Lambda V,\dif)\In \H(\Lambda V,\dif)
\end{align*}
These inclusions do not have be strict by any means, and it is known that $\H_0(\Lambda V,\dif)= \H(\Lambda V,\dif)$ if and only if $\chi_\pi(\Lambda V,\dif)=0$, see \cite[Proposition 32.2(ii), p.~436]{FHT01}. The goal of the next results is to characterize when the first inclusion is strict. Let us first recall:

\begin{prop}\label{prop00}
Let $(\Lambda V,\dif)$ be a pure elliptic minimal Sullivan algebra. % of formal dimension $d$.
Then up to isomorphism it is of the form
\begin{align}\label{eqn08}
(\Lambda V,\dif)\cong (\Lambda V',\dif) \otimes (\Lambda\langle x_i\rangle_{1\leq i\leq k}, 0)
\end{align}
for some maximal $k\geq 0$ (possibly zero) and with $(\Lambda V',\dif)$ a pure minimal elliptic Sullivan algebra. (In particular, $\deg x_i$ is odd for all $1\leq i\leq k$.) Moreover, $(\Lambda V,\dif)$ is formal if and only if $(\Lambda V',\dif)$ has positive Euler characteristic, or, equivalently, vanishing homotopy Euler characteristic $\chi_\pi(\Lambda V',\dif)=0$.
\end{prop}
In the remainder of the section we will use without further reference the splitting \eqref{eqn08} and the notation from \Cref{prop00}. We will also use that the homotopy Euler characteristic is additive under taking products, so that:
$$
\chi_\pi(\Lambda V,\dif)=\chi_\pi(\Lambda V',\dif)+\chi_\pi(\Lambda\langle x_i\rangle_{1\leq i\leq k}, 0)=\chi_\pi(\Lambda V',\dif)+k
$$

\Cref{prop00} is proved in \cite[Lemma~2.5]{Ama13}.  The decomposition is obtained by splitting off odd-degree spherical cohomology classes. Since $(\Lambda V,\dif)$ is pure, the subalgebra of \(\H(\Lambda V,\dif)\) that these elements generate is free.  For the second part of the proposition, recall that in general a product is formal if and only if so is every factor. The construction of the model \eqref{eqn08} in the proof of \cite[Lemma 2.5]{Ama13} reveals that $(\Lambda V',\dif)$ is formal if and only if the relations formed by the differentials of the odd-degree generators form a regular sequence. Since cohomology is finite-dimensional by assumption, this is the case if and only if $\dim (V')^\even=\dim (V')^\odd$, i.e. if and only if $\chi_\pi(\Lambda V',\dif)=0$ as stated in \Cref{prop00}.

Moreover, as mentioned above, the fact that $\chi_\pi(\Lambda V',\dif)=0$ implies in particular that $\H_0(\Lambda V',\dif)=\H^\even(\Lambda V',\dif)$.
This equality remains true for $(\Lambda V,\dif)$ if and only if $k\leq 1$, this can be seen using the K\"unneth formula for \eqref{eqn08}.
Thus, for a model \((\Lambda V,\dif)\) as in \Cref{prop00} which additionally is formal, we find that \(\H_0(\Lambda V,d)=\H^\even(\Lambda V,d)\) if and only if \(k\leq 1\), or, equivalently, if and only if \(\chi_\pi(\Lambda V,\dif)\leq 1\). Noting that \(\H_0\) and \(\chi_\pi\) are invariant under quasi-isomorphism, we may of course drop the minimality assumption:

\begin{cor}\label{corformm}
  For a formal pure elliptic Sullivan algebra \((\Lambda V,\dif)\),
  \[
  \H_0(\Lambda V,\dif)=\H^\even(\Lambda V,\dif) \iff \chi_\pi(\Lambda V,\dif)\leq 1
  \]
\end{cor}

It is our goal to generalize this to the non-formal case, which will be done in \Cref{prop01} below. First we need the following result, which follows directly from \cite[Proposition 32.2 (i)-(ii), p.~436]{FHT01}.
\begin{lemma}\label{lemmatop}
Let $(\Lambda V,\dif)$ be a pure elliptic minimal Sullivan algebra.
We set $r:=\chi_\pi(\Lambda V',\dif)$ and we denote by $d'$ the formal dimension of $(\Lambda V',\dif)$. %(i.e.~$\chi_\pi(\Lambda V,\dif)=k+r$)///the homotopy Euler characteristic of $(\Lambda V',\dif)$
Then
\begin{align*}
\H_r^{d'}(\Lambda V',\dif)=\H^{d'}(\Lambda V',\dif)
\end{align*}
and is one-dimensional. %In particular, if $r=1$ and $k\geq 1$, then a generator of $\H_r^{d'}(\Lambda V',\dif)$ multiplied with $[x_1]$ yields a non-trivial even-degree cohomology class of $\H(\Lambda V,\dif)$ not generated by $\Lambda V^\even$.
\end{lemma}

Now we are ready to state the main result in this section:
\begin{prop}\label{prop01}
  For a pure elliptic minimal Sullivan algebra $(\Lambda V,\dif)$ as in \Cref{prop00},
  $$
  \H_0(\Lambda V,\dif)=\H^\even(\Lambda V,\dif) \iff  k+r \leq 1 \iff \chi_\pi(\Lambda V,\dif)\leq 1
  $$
  More precisely, if $\chi_\pi(\Lambda V',\dif)\leq 1$, then $\H^\even(\Lambda V',\dif)=\H_0(\Lambda V',\dif)$;
  if $\chi_\pi(\Lambda V',\dif) > 1$, there is an even-degree class $0\neq x\in \H_s(\Lambda V',\dif)$ with $s>0$, i.e.~a class not representable by $\Lambda V^\even$.\COMMENT{MA/DG: Both $\Lambda (V')^\even$ and $\Lambda V^\even$ are correct. Usually, we are interested in $\Lambda V^\even$.}
\end{prop}

In particular, using again the invariance of \(\H_0\) and \(\chi_\pi\) under quasi-isomorphism, we obtain the following generalization of \Cref{corformm}:

\begin{cor}\label{cor01}
  For a pure elliptic Sullivan algebra $(\Lambda V,\dif)$,
  \[
    \H_0(\Lambda V,\dif)=\H^\even(\Lambda V,\dif) \iff \chi_\pi(\Lambda V,\dif)\leq 1
  \]
\end{cor}

\begin{proof}[Proof of \Cref{prop01}]
Recall that we use the notation $r:=\chi_\pi(\Lambda V',\dif)$. If $r=0$, then $(\Lambda V,\dif)$ is formal and then the result follows from \Cref{corformm}.

If $r$ is even and $\geq 2$ , then \Cref{lemmatop} tells us that $\H_r(\Lambda V',\dif)$ is non-trivial, and since $\H_r(\Lambda V',\dif)\In \H_r(\Lambda V,\dif)\In \H^\even(\Lambda V,\dif)$ it follows that $\H_0(\Lambda V,\dif)\neq \H^\even(\Lambda V,\dif)$.

For $r$ odd we need to distinguish two cases. If $k\geq 1$, we can multiply a non-trivial class from $\H_r(\Lambda V',\dif)$ with $[x_1]$ in order to obtain a non-trivial cohomology class of even lower degree (i.e.~a class whose lower degree is even), thus $\H_0(\Lambda V,\dif)\neq \H^\even(\Lambda V,\dif)$.

If $k=0$ (and $r$ odd) we have that $(\Lambda V,\dif)=(\Lambda V',\dif)$. We again distinguish two subcases. If $r=1$, then $\H(\Lambda V',\dif)=\H_0(\Lambda V',\dif)\oplus \H_1(\Lambda V',\dif)$, and $\H_1(\Lambda V',\dif)\In \H^\odd(\Lambda V',\dif)$ whence $\H^\even(\Lambda V',\dif)=\H_0(\Lambda V',\dif)$.

It remains to investigate the case $r+k=r\geq 3$. We shall use that $(\Lambda V',\dif)$ satisfies Poincar\'e duality and that lower degree is additive. As above we have that $\H_r(\Lambda V',\dif)\neq 0$, and by Lemma \ref{lemmatop} that $\H_r^{d'}(\Lambda V',\dif)=\H^{d'}(\Lambda V',\dif)$ (with $d'$ the formal dimension of $(\Lambda V',\dif)$). Taken together, these observations imply that, unless
\begin{align}\label{eqneven}
\H_{0<\ast<r}(\Lambda V',\dif)=0, %\H_0(\Lambda V',\dif)=\H_{0<\ast<r}(\Lambda (V')^\even,0),
\end{align}
there exists a non-trivial element of positive lower degree in $\H^\even(\Lambda V',\dif)$. (Indeed, an element of odd lower degree requires a Poincar\'e dual of even lower degree.)
Thus we need to show that \eqref{eqneven} cannot hold under the assumption that $k+r=r\geq 3$ (as $r$ is odd). For this we can partially quote the proof of \cite[Proposition 32.3, p.~437]{FHT01}. Indeed, we choose a homogeneous basis $(n_i)_{1\leq i\leq l}$ of $(V')^\odd$ and work by induction. By definition, it holds that $l> r$. We observe that, since $r\geq 3$, we have $\dim (V')^\odd-\dim (V')^\even\geq 3$. Consequently, as depth is restricted from above by Krull dimension and $\dim (V')^\even\leq \dim (V')^\odd-3$, neither $(\dif n_i)_{1\leq i\leq l}$, nor, more importantly, $(\dif n_i)_{1\leq i\leq l-1}$ can form a regular sequence in $\Lambda (V')^\even$. In other words, we obtain that
\begin{align*}
\H_{>0}(\Lambda (V')^\even \otimes \Lambda\langle n_1, n_2, \ldots, n_{l-1}\rangle,\dif)\neq 0
\end{align*}
Let $y$ represent such a non-trivial class of minimal (ordinary) degree. As observed (in even slightly larger generality) in the proof of \cite[Proposition 32.3, p.~437]{FHT01} we deduce that
\begin{align*}
0\neq [y] \in \H_{>0}(\Lambda (V')^\even \otimes \Lambda\langle n_1, n_2, \ldots, n_{l-1}, n_l\rangle,\dif)=\H_{>0}(\Lambda V',\dif),
\end{align*}
i.e.~$[y]$ considered as a cohomology class of the larger algebra is not trivial as well. As the arguments were omitted in \cite{FHT01} let us quickly sketch this: We consider the rational spherical fibration
\begin{align*}
(\Lambda (V')^\even \otimes \Lambda \langle n_1, n_2, \ldots, n_{l-1}\rangle,\dif)&\hto{}(\Lambda (V')^\even \otimes \Lambda\langle n_1, n_2, \ldots, n_{l-1}, n_l\rangle,\dif)
\\&\to (\Lambda \langle n_l\rangle,0)
\end{align*}
and filter the consequent Serre spectral sequence from it. On the $E_2$-term we see that, since the original algebra is two-stage, the differential on $[n_l]$ (which in the original algebra then has lower degree $0$) cannot hit an element of positive lower degree. Since $[y]$ was of minimal (ordinary) degree amongst all elements of positive lower degree, all potential preimages under any differential $d_i$ on any page $E_i$ are exclusively represented by elements of lower degree at most $1$---i.e.~they correspond either to multiples of $y_l$ by $\Lambda (V')^\even$ or to the latter itself. It follows that $[y]$ survives to $E_\infty$. (As observed in \cite{FHT01} this argument generalizes to classes $[y]$ which are only minimal within their respective lower degree.)

It remains to see that $[y]$ can be taken not to be of lower degree $r$, i.e.~that it is of strictly smaller lower degree, without loss of generality. Here, we may basically quote the proof of \cite[Proposition 32.3, p.~437]{FHT01} again in order to refine the arguments from above some more.

Indeed, we may even successively consider the sequences
\begin{align*}
\dif n_1, \ldots, \dif n_i
\end{align*}
for $1\leq i\leq l$. As observed above, since the sequence is not regular for $i=l-1$, there is indeed some $2\leq i_0\leq l-1$ such that $\dif n_1, \dif n_2, \ldots, \dif n_{i_0-1}$ is regular, but $\dif n_1, \dif n_2, \ldots, \dif n_{i_0}$ is not. As elaborated in the cited proof, we now obtain that
\begin{align*}
& \H(\Lambda (V')^\even \otimes \Lambda \langle n_1, n_2, \ldots, n_{i_0}\rangle,\dif)
\\\cong
& \H(\Lambda (V')^\even /(\dif n_1, \ldots, \dif n_{i_0-1})\otimes \Lambda \langle n_{i_0}\rangle,\bar\dif)
\\ \cong &
\Lambda (V')^\even /(\dif n_1, \ldots, \dif n_{i_0}) \oplus (K\otimes n_{i_0})
\end{align*}
where $K\In (V')^\even /(\dif n_1, \ldots, \dif n_{i_0-1})$ is the annulator of $\dif n_{i_0}$ (the latter considered as a class in the quotient $\Lambda (V')^\even /(\dif n_1, \ldots, \dif n_{i_0-1})$) and where $\bar\dif n_{i_0}$ is the class of $\dif n_{i_0}$ in the quotient. It is observed that
\begin{align*}
\H_1(\Lambda (V')^\even \otimes \Lambda \langle n_1, n_2, \ldots, n_{i_0}\rangle,\dif)&\cong K\otimes n_{i_0}
\intertext{and}
\H_{\geq 2}(\Lambda (V')^\even \otimes \Lambda \langle n_1, n_2, \ldots, n_{i_0}\rangle,\dif) &=0
\end{align*}
Since the sequence $\dif n_1, \ldots, \dif n_{i_0}$ is not regular, it follows that there is a non-trivial class $[y]$ of lower degree $1$. We actually now may replace the class $[y]$ considered above by this class $[y]$. In other words, we have seen that this class $[y]$ passes non-trivially to
\begin{align*}
0\neq [y] \in \H_{1}(\Lambda (V')^\even \otimes \Lambda\langle n_1, n_2, \ldots, n_{l-1}, n_l\rangle,\dif)=\H_{1}(\Lambda V',\dif),
\end{align*}

As a conclusion, we have found a non-trivial cohomology class of positive lower degree  smaller than $r$ (actually, it may be taken to have lower degree equal to one) whence either this class or its Poincar\'e dual (using $r$ odd) is of even lower whence even ordinary degree. Thus this respective class is not represented by $\Lambda V^\even$.
\end{proof}

%%%%%%%%%%%%%%%%%%%%%%%%%%%%%%%%%%%%%%%

\subsection{Surjectivity properties in equivariant cohomology}\label{SEC:surjectivity_cohomology}
For a $G$-space $X$ recall from \Cref{SUBSEC:genuine-equivariant-K-theory} the Borel fibration
\begin{align}\label{eq:borel-fibration}
X\hto{j} X_G:=X\times_G \E G\xto{p} \B G
\end{align}
where $X_G$ is the Borel construction defining equivariant cohomology $\H_G^*(X):=\H^*(X_G)$. (Coefficients will always be rational.) We refer to \Cref{SUBSEC:Borel-equivariant-cohomology-theories} for the definition of the Borel forgetful maps \(\H^*(j)\) and \(\H^\even(j)\). \Cref{surjectivity_of_Heven_iff_K0} begs the following question:

\begin{ques}\label{Q:surjectivity_of_j} When is $\H^\even(j)\colon \H^\even_G(X)\to \H^\even(X)$ surjective?
\end{ques}

In order to put into perspective \Cref{Q:surjectivity_of_j}, recall that $X$ is called \define{equivariantly formal} if $\H^*(j)\colon \H^*(X_G)\to \H^*(X)$ is surjective (i.e.~if the Borel fibration~\eqref{eq:borel-fibration} is totally cohomologous to zero). Equivariant formality has undergone vast research and is a rather important concept in equivariant cohomology. Obviosuly, equivariant formality implies \Cref{Q:surjectivity_of_j}. Due to this proximity of our question to big research fields it seems too much to hope for simple and short answers. This is why for the rest of this section we shall have to restrict the class of spaces $X$ which we consider. %We begin by assuming that $X$ is a nilpotent rationally elliptic space.

We will also be interested in the maps \(\H^*(p)\) and \(\H^\even(p)\), defined in the obvious similar way: as pull-backs in rational cohomology along the map $p$. More precisely we will study:

\begin{ques}\label{Q:surjectivity_of_p} When is $\H^\even(p)\colon \H^\even(\B G)\to \H^\even_G(X)$ surjective?
\end{ques}

The motivation for \Cref{Q:surjectivity_of_p} comes from the study of homogeneous spaces, as we explain in \Cref{remhom} below. This will be refined to biquotients in \Cref{SEC:applications_biquotients_cohomogeneity1}.

\begin{rem}\label{remhom}
Let $X$ be a homogeneous space $G/H$ of compact connected groups. Then both questions applied to different Borel fibrations yield the same problem, the application we are interested in. More precisely, consider the following two actions: left multiplication of $G$ on $G/H$ and left multiplication of $H$ on $G$. The corresponding Borel fibrations are
$$
G/H \hto{j} (G/H)_G \to \B G\qquad\quad \text{and} \qquad\quad
G\hto{} G/H\simeq G_H \xto{p} \B H
$$
 Note that the action of $H$ on $G$ is free and so the Borel construction $G_H$ is homotopy equivalent to $G/H$, thus $\H_H^*(G)\cong\H^*(G/H)$ (see \cite[Proposition 2.9, p.~37]{FHT01}). Moreover, $\H_G^*(G/H)$ functorially identifies with $\H^*(\B H)$. Under these identifications, $\H^\even(j)$ for the first fibration equals $\H^\even(p)$ for the second one. Indeed, a model for the Borel construction $(G/H)_G$ is given by
\begin{align*}
(\Lambda (V_{\B H} \oplus V_G \oplus V_{\B G}),\dif)
\end{align*}
using the respective minimal models. The differential $\dif$ is trivial on $V_{\B H}$ and $V_{\B G}$ and such that it is a perturbed version (reflecting the homogeneous structure of $G/H$) of the morphism induced by the identity  (up to degree shift $+1$) $V_G\to V_{\B G}$. That is, after performing an isomorphism and splitting off the contractible algebra $(\Lambda (V_G\oplus V_{\B G}),\dif)$, we obtain the minimal model of the Borel construction $(G/H)_G$ given by $(\Lambda V_{\B H},0)\cong (\H^*(\B H),0)$. The morphisms induced by $j$ and $p$ are just the inclusion of $(\Lambda V_{\B H},0)$ into the model $(\Lambda (V_{\B H}\oplus V_G),\dif)$ of $G/H$ constructed in \Cref{Sullivan-for-biquotient}.
\end{rem}

Let us investigate Questions \ref{Q:surjectivity_of_p} and \ref{Q:surjectivity_of_j}, in that order.

\smallskip
\noindent{\textbf{Convention.}} For the rest of the present section $X$ will always denote a nilpotent path-connected finite CW complex. Additional assumptions on such a space $X$ (e.g.\ rational ellipticity or pureness) will be made explicitly. As usual $G$ will denote a compact connected Lie group.

\subsubsection{The morphism $\H^\even(p)\colon \H^\even(\B G)\to \H^\even_G(X)$}

Here we restrict to almost free actions. Recall that, in contrast to a free action for which all isotropy groups are trivial, an \define{almost free action} is one for which isotropy groups are finite. For such an action we still have the rational homotopy equivalence
$$
X_G\simeq _\qq X/G, \quad \text{which implies}\quad\H^\even_G(X)\cong\H^\even(X/G).
$$
Due to standard localization results, it is known that $G$ acts almost freely on $X$ if and only if $X/G$ has finite-dimensional rational cohomology (cf.~\cite[Theorem~7.7, p.~276]{FOT08}).

In order to state the main result, which is \Cref{theomor1} below, we first need to discuss a model for $X/G$. It can be constructed from the Borel fibration out of minimal models of $X$ and $\B G$ as follows \cite[Chapter 15]{FHT01}. Let $(\Lambda V_X,\bar \dif)$ and $(\H^*(\B G),0)=(\Lambda V_{\B G},0)$ be minimal Sullivan models of $X$ and $\B G$ respectively. Then
\begin{equation}\label{EQN:model_fro_almost_free_quotients}
    (\Lambda V_X\otimes \H^*(\B G),\dif)
\end{equation}
is a Sullivan model of $X/G$, where $\dif$ is such that $\dif-\bar \dif: \Lambda V_X\to \Lambda^{\geq 1}(\B G)\otimes \Lambda V_X$ and such that $\dif$ vanishes on $\H^*(\B G)$. The model in \eqref{EQN:model_fro_almost_free_quotients} has the structure of a \define{relative model} and as such is minimal, i.e. it has the structure of a \define{minimal relative model}; however as a Sullivan model it is not minimal in general (see \cite[Section 14]{FHT01} for definitions and results). Observe that if $X$ is rationally elliptic, then this model \eqref{EQN:model_fro_almost_free_quotients} is generated over a finite-dimensional vector space as well.

\begin{theo}\label{theomor1}
Let $X$ be a rationally elliptic space with an almost free action of a compact connected Lie group $G$. If the above model \eqref{EQN:model_fro_almost_free_quotients} for $X/G$ is pure, then the following two assertions are equivalent:
\begin{enumerate}
\item $\H^\even(p)\colon\H^\even(\B G)\to \H^\even(X/G)$ is surjective.
\item $\pi_*(X)\otimes \qq=\pi_\odd(X)\otimes \qq$ and $\dim \pi_*(X)-\rank G\leq 1$.
\end{enumerate}
Without the purity assumption the implication $(2)\Rightarrow (1)$ still holds.
\end{theo}

\begin{rem}\label{REM:condition_euler_homotopy}
If $X/G$ above is nilpotent (and hence rationally elliptic) and satisfies that its rational homotopy is concentrated in odd degrees, then Assertion~(2) is equivalent to $\chi_\pi(X/G)\leq 1$. To show it, observe that in the model \eqref{EQN:model_fro_almost_free_quotients} the evenly graded part is generated by $\pi_*(\B G)\otimes \qq$ and the odd one by $\pi_*(X)\otimes \qq$. The difference of dimensions of homotopy groups (recall $\rank G=\dim \pi_* (G)$) equals the homotopy Euler characteristic even if the model of the Borel fibration is not minimal as a Sullivan model (as it usually will not be), since Euler characteristics are not affected by taking differentials.
\end{rem}

\begin{rem}\label{remhomapp}
  We can apply \Cref{theomor1} in the context of biquotients. For a given biquotient $\biq{G}{H}$
  of compact connected Lie groups, consider the free action of $H$ on $X:=G$. Recall from \Cref{biq_pure} that biquotiens are pure and rationally elliptic, thus \Cref{theomor1} together with \Cref{REM:condition_euler_homotopy} imply that the morphism $\H^\even(p)\colon \H^\even(\B H)\to\H^\even(\biq{G}{H})$ is surjective if and only if $\rank G-\rank H \leq 1$. This result is stated explictely in \Cref{THM:surjectivity-of-Kalpha-biquotients} below, and the proof of it shall be reviewed in \Cref{SEC:applications_biquotients_cohomogeneity1}.
\end{rem}

%%%%%%%%%%%%%%%%%%%%%%%%%%%%%%%%%%%%%%%

Let us discuss the proof of \Cref{theomor1}. For it we use the model in \eqref{EQN:model_fro_almost_free_quotients} for $X/G$. Under the isomorphism $\H^*(X/G)\cong\H(\Lambda V_X\otimes \H^*(\B G),\dif)$, the map
$$\H^\even(p)\colon \H^\even(\B G)\to \H(\Lambda V_X\otimes \H^*(\B G),\dif)$$
is the one induced by the inclusion morphism between models
$$
(\H^*(\B G),0)\hto{} (\Lambda V_X \otimes \H^*(\B G),\dif)
$$
onto the second factor. In particular, the image of this inclusion is a differential subalgebra of the product.
A first step will be to ``reduce to the pure case'', since in such case the image of $\H^\even(p)$ coincides with $\H_0(\Lambda V_X\otimes \H^*(\B G),\dif)$ and we can use the tools from \Cref{secpre}.

Recall that any Sullivan algebra $(\Lambda V,\dif)$ has an \define{associated pure Sullivan algebra} $(\Lambda V,\dif_\sigma)$, see \cite[p.~438]{FHT01} for the construction of $\dif_\sigma$. In the case of our model in \eqref{EQN:model_fro_almost_free_quotients}, since the differential $\dif$ vanishes on $\H^*(\B G)$, the associated pure model $(\Lambda V_X\otimes \H^*(\B G),\dif_\sigma)$ has the structure of a relative model over $\H^*(\B G)$ with fibre $(\Lambda V_X,\dif_\sigma)$ yielding the inclusion $\H^*(\B G)\hto{} \H^*(\Lambda V_X \otimes \H^*(\B G),\dif_\sigma)$. The key observation is:

\begin{lemma}\label{lemmared}
If the morphism $\H^\even(\B G)\hto{} \H^\even(\Lambda V_X \otimes \H^*(\B G),\dif_\sigma)$ is surjective, then so is $\H^\even(\B G)\hto{} \H^\even(\Lambda V_X \otimes \H^*(\B G),\dif)$.
\end{lemma}
\begin{proof}
The odd spectral sequence (see \cite[p.~438]{FHT01}) satisfies $(E_0,\dif_0)=(\Lambda V_X\otimes \H^*(\B G),\dif_\sigma)$ and converges to $\H(\Lambda V_X \otimes \H^*(\B G),\dif)$. Hence any cohomology class of $\H^\even(\Lambda V_X \otimes \H^*(\B G),\dif)$ is represented by one from the associated pure model. This implies the lemma.

\end{proof}

We may reformulate Lemma \ref{lemmared} or at least its proof additionally using Proposition \ref{prop01} as

\begin{cor}\label{COR:spherical_cohomology}
Let $X$ be a rationally elliptic space with $\chi_\pi(X)\leq 1$ such that all even-degree rational homotopy groups define spherical cohomology. Then $\H^\even(X)$ is generated by spherical cohomology classes of even degree.
\end{cor}

\begin{proof}
The arguments are basically the same as before: Let $(\Lambda V,\dif)$ be a minimal model of $X$. The odd spectral sequence converges from the associated pure model $(\Lambda V,\dif_\sigma)$ to $\H(\Lambda V,\dif)$. The associated pure model is elliptic if and only if so is $(\Lambda V,\dif)$. If $(\Lambda V,\dif)$ is minimal, so is $(\Lambda V, \dif_\sigma)$, and respective homotopy Euler characteristics agree. Thus, by Proposition \ref{prop01} we get $\H(\Lambda V,\dif_\sigma)=\H_0(\Lambda V,\dif_\sigma)$ and hence is generated by $\Lambda V^\even$. By the assumption on spherical cohomology, we derive that $V^\even$ is closed in $(\Lambda V,\dif)$, and $\Lambda V^\even$ surjects onto all elements of even total degree in each sheet of the odd spectral sequence, and hence, due to convergence, the map $(\Lambda V^\even,0) \to \H^\even(\Lambda V,\dif)$ induced by the identity is surjective.
\end{proof}
Note that, without the assumption on spherical cohomology in \Cref{COR:spherical_cohomology}, even degree homotopy groups need not define cohomology classes and they only represent cohomology classes of $\H^\even(\Lambda V,\dif)$ in the sense of the spectral sequence; i.e.~on the level of the model they would need to be perturbed.

\smallskip

We now combine the obtained insight to prove \Cref{theomor1}.

\begin{proof}[Proof of \Cref{theomor1}]

Let $X$ be as in the statement. This proof consists of two parts. In the first part we show that, in order to prove the implication $(2)\Rightarrow (1)$, we can assume that $X/G$ is pure. For this, consider the minimal relative model $(\Lambda V_X\otimes \H^*(\B G),\dif)$ of $X/G$ as constructed in \eqref{EQN:model_fro_almost_free_quotients}. We want to show that $\H^\even(\B G)\to \H^\even(X/G)$ is surjective, i.e.~that $\H^\even(\B G)\to \H^\even(\Lambda V_X\otimes \H^*(\B G),\dif)$ is surjective. In view of \Cref{lemmared} it suffices to show that $\H^\even(\B G)\to \H^\even(\Lambda V_X\otimes \H^*(\B G),\dif_\sigma)$ is surjective.

Now notice that $(\Lambda V_X\otimes \H^*(\B G),\dif_\sigma)$ is the same as forming the relative minimal model over $\H^*(\B G)$ with respect to $(\Lambda V_X,\dif_\sigma)$, i.e.~in $(\Lambda V_X\otimes \H^*(\B G),\dif)$ first projecting to $(\Lambda V_X,\dif_\sigma)$ and then projecting to $(\Lambda V_X\otimes \H^*(\B G),\dif_\sigma)$. Since the model was chosen relatively minimal, $(\Lambda V_X,\bar \dif)$ is minimal, and so is $(\Lambda V_X,\dif_\sigma)$. The assumptions
$$
    \pi_*(X)\otimes \qq=\pi_\odd(X)\otimes \qq \quad \text{and}\quad \dim \pi_*(X)-\rank G\leq 1
$$
translate to
$$
\qquad V_X=V_X^\odd \quad \text{and}\quad\dim V_X-\dim \pi_*(\B G)\leq 1
$$
respectively for both $(\Lambda V_X,\dif)$ and $(\Lambda V_X,\dif_\sigma)$. Note that the number of algebra generators of the free polynomial algebra $\H^*(\B G)$ equals $\dim \pi_*(\B G)$.
Moreover, by \cite[Proposition~32.4, p.~438]{FHT01} we preserve finite-dimensional cohomology, i.e.~$\operatorname{fdim}  \H(\Lambda V_X\otimes \H^*(\B G),\dif_\sigma)=\operatorname{fdim} \H(\Lambda V_X\otimes \H^*(\B G),\dif)<\infty$, where $\operatorname{fdim}$ denotes formal dimension. Hence we reduced to the case that the model in \eqref{EQN:model_fro_almost_free_quotients} of $X/G$ is pure whilst preserving the assumption on the homotopy groups of $X$.

\smallskip

In the second part of the proof we show the equivalence $(1)\Leftrightarrow (2)$ in the pure case. This will follow from Proposition \ref{prop01}. That is, we assume that $(\Lambda V\otimes \H^*(\B G),\dif)$ is pure elliptic and minimal (by possibly replacing it by a minimal model). Note that passing to a minimal model has no effect on the homotopy Euler characteristic as noted below \Cref{theomor1}.

If $X$ has even-degree rational homotopy, then, by construction (using that $\H^*(\B G)$ is concentrated in even degrees), this passes on to non-trivial even degree rational homotopy other than $\pi_*(\B G)\otimes \qq$ in $(\Lambda V_X\otimes \H^*(\B G),\dif)$. Consequently, the map from $\H^*(\B G)$ will not be surjective. This shows that $(1)$ implies the first assertion in $(2)$.

We are now in the situation that $(\Lambda V_X\otimes \H^*(\B G),\dif)$ is minimal pure and $V_X=V_X^\odd$.
%%$V^\even=\H^*(\B G)$ (or, more precisely, $V^\even$ is generated by %$\H^*(\B G)$), and
Observe that the surjectivity of $\H^\even(p)$ is equivalent to $\chi_\pi(\Lambda V_X\otimes \H^*(\B G),\dif)\leq 1$ by \Cref{prop01}, since the image of such map equals precisely $\H_0(\Lambda V_X\otimes \H^*(\B G),\dif)$. Parallely, the first assertion in $(2)$ (i.e. $V_X=V_X^\odd$) implies
$$
\chi_\pi(\Lambda V_X\otimes \H^*(\B G),\dif)= \dim \pi_*(X)-\rank G,
$$
using that  $\pi_*(\H^*(\B G))=\pi_*(\B G)=\pi_\even(\B G)$ has dimension $\rk G$. Altogether this proves the equivalence $(1)\Leftrightarrow (2)$.
\end{proof}

\begin{rem}
We remark further that the proof of \Cref{theomor1} shows a little more: If $X/G$ does possess some pure model, then its minimal model will be pure as well, and we can decompose it as a rational fibration over the image of $\H^*(\B G)$ in $\H^*(X/G)$; observe that this image is a differential subalgebra which we denote by $(H,0)\In \H(\Lambda V\otimes \H^*(\B G),\dif)$. The fibre then is again a model of $X$, and we can proceed with the line of arguments leading to the proof as given, with the role of $(\H^*(\B G),0)$ taken by $(H,0)$.
\end{rem}

%%%%%%%%%%%%%%%%%%%%%%%%%%%%%%%%%%%%%%%%%%%%
\subsubsection{The morphism $\H^\even(j)\colon \H^\even_G(X)\to \H^\even(X)$}\label{pagesubsubsec}

Recall that the surjectivity of this map is crucial for our applications in K-theory due to  \Cref{surjectivity_of_Heven_iff_K0}.

\begin{rem}\label{remhomapp2}
We have already implicitly shown that $\H^\even(j)$ is surjective for the $G$-action on a homogeneous space $X=G/H$ of compact connected groups if and only if $\chi_\pi(G/H)\leq 1$. This follows from \Cref{remhom}, \Cref{theomor1} and \Cref{remhomapp}. It shall be restated in \Cref{THM:half_equiv_formal_homogeneous} below.
\end{rem}

In \Cref{seccohom} it will be our goal to characterize the surjectivity of $\H^\even(j)$ for cohomogeneity one actions. In the present section we derive the following further general situations in which surjectivity holds.

\begin{prop}\label{PROP:sufficient_conditions_surjectivity_half_equivariant}
Let $X$ be a pure (and hence rationally elliptic) space with $\chi_\pi(X)\leq 1$ with an action of a compact connected Lie group $G$. Suppose that one of the following holds:
\begin{enumerate}
    \item $\deg \pi_\even(X)\otimes \qq< \deg \pi_\odd(X)\otimes \qq$ (i.e.~if all even degree rational homotopy groups have lower degrees than the lowest of odd degree),
    \item $X_G$ is pure (whilst $\H^*(X_G)$ not necessarily finite dimensional).
\end{enumerate}
Then $\H^\even(j)\colon \H^\even_G(X)\to \H^\even(X)$ is surjective.
\end{prop}
\begin{proof}
Note that either of the assumptions (1) or (2) implies that the Leray--Serre spectral sequence of the Borel fibration $X\to X_G\to \B G$ degenerates on all spherical cohomology classes of even-degree. Under Assumption (1), this follows, since $\H^*(\B G)$ is concentrated in even degrees. Under (2), it follows from the fact that the even-degree rational homotopy groups of $X_G$ are the ones of $X$ plus the ones of $\B G$.

Thus we have that every spherical cohomology class of even-degree lies in the image of $\H^\even(j)\colon \H^\even_G(X)\to \H^\even(X)$. Since $X$ is pure by assumption, the set of spherical cohomology classes of even-degree generates exactly $\H_0(X)$ (see \Cref{secpre}). Finally, \Cref{prop01} says that $\H_0(X)=\H^\even(X)$ if $\chi_\pi(X)\leq 1$, and this completes the proof.
\end{proof}
\begin{example}
Any sphere, any complex or quaternionic projective space $X$, and, much more generally, any complex or quaternionic generalized Grassmanian  of the form $U(n)/(U(n_1)\times \ldots \times U(n_k))$ respectively $Sp(n)/(Sp(n_1)\times \ldots \times Sp(n_k))$ with $n-1\leq \sum_{1\leq i\leq k} n_i\leq n$, is pure by \Cref{biq_pure} and satisfies both $\deg \pi_\even(X)\otimes \qq< \deg \pi_\odd(X)\otimes \qq$ (using the fact that the used standard blockwise inclusions $U(n_i)\to U(n)$ respectively $Sp(n_i)\to Sp(n)$ are highly connected) and $\chi_\pi(X)\leq 1$ by \eqref{EQ:homotopy_Euler_biquotient}.  Thus \Cref{PROP:sufficient_conditions_surjectivity_half_equivariant} applies to any $G$-action on such a space. In contrast, any odd dimensional sphere admits actions which are not equivariantly formal. The easiest example of this sort is the free standard $S^1$ action with quotient a complex projective space. %\cite[Remark 2.12]{GoRo2011}.
\end{example}

We finish by commenting on the assumption $\chi_\pi(X)\leq 1$. On the one hand, it is easy to state examples lacking surjectivity if $\chi_\pi(X)\geq 2$: Consider for example left multiplication of a Lie group $G$ on itself $X:=G$. In this case the Borel fibration equals the universal principal $G$-bundle $G\hto{} \E G \to \B G$, and we have that $\H_G^\even(G)=\H^*(\E G)=\qq$. If $\chi_\pi(G)=\rk G\geq 2$, then $\dim_\qq\H^\even(G)\geq 2$ and hence $\H_G^\even(G)$ cannot surject onto $\H^\even(G)$.
On the other hand, $\chi_\pi(X)\leq 1$ is not a condition which is necessary in general, consider for example the trivial $G$-action on any space $X$.

\subsection{Applications to homogeneous spaces, cohomogeneity one manifolds and biquotients}\label{SEC:applications_biquotients_cohomogeneity1}

We now provide applications of the established result
to the specified spaces. Note that as is known/as we shall illustrate all these spaces under consideration admit elliptic pure models and hence are very well accessible to our theorems. Coefficients in this section will always be rational.

\subsubsection{Homogeneous spaces and biquotients}\label{subsecbiq}

Here we collect and review the results of the previous sections for homogeneous spaces and more generally biquotients. We hope this section illustrates the discussion above and provides a service to the reader interested in such spaces.

Recall that related to a homogeneous pace $G/H$ there are two natural actions: the transitive action of $G$ on $G/H$ and the free action of $H$ on $G$. The corresponding Borel fibrations are
$$
G/H \hto{j} \B H \simeq (G/H)_G \to \B G\qquad\quad \text{and} \qquad\quad
G\hto{} G/H\simeq G_H \xto{p} \B H
$$
In \Cref{remhom} we have seen that the morphism $\H^\even(j)$ for the first fibration identifies with $\H^\even(p)$ for the second one. As we already discussed in \Cref{remhomapp2} we have:

\begin{theo}\label{THM:half_equiv_formal_homogeneous}
Let $G/H$ be a homogeneous space with $G,H$ connected. The morphism $\H^\even(j)\colon \H_G^\even(G/H)\to\H^\even(G/H)$ is surjective if and only if $\rank G-\rank H \leq 1$.
\end{theo}

This theorem is crucial in the proofs of Theorems \ref{THM:surjectivity of homogeneous bundles} and \ref{THM:double_stabilization_homogeneous}, which are given at the end of this section.

To put some context, we recall the well-known fact that $\H^*(j)\colon \H_G^*(G/H)\to\H^*(G/H)$ is surjective in all degrees (i.e. $G/H$ is equivariantly formal) if and only if $\rank G=\rank H$ (see e.g.\ below Proposition 2.9 in \cite{GM14:coho1}).

In the more general case of biquotients $\biq{G}{H}$ (we refer to \Cref{secpre} for a brief introduction), note that for we no longer have a natural action of $G$ on $\biq{G}{H}$ in general. Instead we consider the free action of $H$ on $G$, yielding the borel fibration $G\hto{} \biq{G}{H}\simeq G_H \xto{p} \B H$. As we discussed in \Cref{remhomapp}:

\begin{theo}\label{THM:surjectivity-of-Kalpha-biquotients}
Let $\biq{G}{H}$ be a biquotient of compact connected Lie groups. The morphism $\H^\even(p)\colon \H^\even(\B H)\to\H^\even(\biq{G}{H})$ is surjective if and only if $\rank G-\rank H \leq 1$.
\end{theo}

Although we have proven the more general result \Cref{theomor1} above, we review it here for biquotients to illustrate how the tools from \Cref{secpre} apply in this case.
First, we rewrite the cohomology of $\biq{G}{H}$ by refining the model in \Cref{prop00} in the special case of a biquotient as follows.
\begin{align*}
\H(\Lambda V,\dif)&\cong \H(\Lambda V_{\B H}\otimes \Lambda V_G,\dif) \\
&\cong \H((\Lambda V',\dif) \otimes (\Lambda\langle x_i\rangle_{1\leq i\leq k}, 0)) \\
&\cong \H(\Lambda V',\dif) \otimes \Lambda\langle [x_i]\rangle_{1\leq i\leq k}
\\&\cong (\biq{\H^*(\B H)}{\H^*(\B G)}\oplus N) \otimes \Lambda\langle [x_i]\rangle_{1\leq i\leq k}
\end{align*}
The first isomorphism follows from the standard model of a biquotient given in \Cref{Sullivan-for-biquotient}, which is pure. In the second line we pass to the cohomology of the quasi-isomorphic minimal model from \Cref{prop00}, which is simplified in the third line. In the last isomorphism we decompose further the factor $\H (\Lambda V',\dif)$ by applying the lower grading in cohomology described in \eqref{EQ:lower_grading}. More precisely, we define
\begin{equation*}
    \biq{\H^*(\B H)}{\H^*(\B G)}:=\H_0(\Lambda V',\dif), \quad \text{and} \quad N:=\bigoplus_{i> 0} \H_i(\Lambda V',\dif)
\end{equation*}
Observe that the notation for $\H_0(\Lambda V',\dif)$ is rather intuitive, since this space is exactly the quotient of $\H^*(\B H)\cong \Lambda V_{\B H}$ by the ideal generated by $\dif (\Lambda V_G)$. (Lower gradings $>0$ involve ``Massey products''.) For later reference we summarize the discussion above in the following:

\begin{prop}\label{prop02}
For $(\Lambda V,\dif)$ a model of a biquotient $\biq{G}{H}$ of compact connected Lie groups we have
\begin{align*}
\H(\Lambda V,\dif)\cong (\biq{\H^*(\B H)}{\H^*(\B G)}\oplus N) \otimes \Lambda\langle [x_i]\rangle_{1\leq i\leq k}
\end{align*}
Here $N$ has a positive lower grading and its non-trivial top degree plus $k$ equals $\chi_\pi(\biq{G}{H})=\rk G - \rk H$.
\end{prop}
%%once we have seen that the algebra generators not coming from $\H^*(\B H)$ form a %subalgebra intersecting %$\biq{\H^*(\B H)}{\H^*(\B G)}\otimes \Lambda\langle [x_i]\rangle$ %trivially. See proof in \cite[Theorem~10.12, %p.~60]{GHV3}.
%%For this we consider the lower grading on $(\Lambda V',\dif)$. Denote by $N$ the %homogeneous part of lower %degree $\geq 1$.
%%Use that positively lower graded elements generate an ideal.

\begin{rem}\label{REM:image_biquotient_map}
Under the isomorphism in \Cref{prop02} for the cohomology of $\biq{G}{H}$, the image of the map
$$
    \H^\even(p)\colon  \H^\even(\B H)\to (\biq{\H^*(\B H)}{\H^*(\B G)}\oplus N) \otimes \Lambda\langle [x_i]\rangle_{1\leq i\leq k}
$$
equals \(\H_0(\Lambda V,d) = \biq{\H^*(\B H)}{\H^*(\B G)}\).
\end{rem}

With these observations at hand we can already make the following observations to illustrate the surjectivity of $\H^\even(p)\colon \H^\even(\B H)\to\H^\even(\biq{G}{H})$ in different cases. The reader may skip this part and directly go to its short proof below. We set $c:=\rank(G)-\rank(H)$, and we will make use of the notation and the statements in \Cref{prop02} and \Cref{REM:image_biquotient_map} throughout.

\begin{itemize}
\item In the case $c=0$ it follows that $N=0$ and $k=0$, thus the quotient $\biq{\H^*(\B H)}{\H^*(\B G)}$ of cohomology algebras indeed computes the cohomology of $\biq{G}{H}$ (and, moreover, encodes the rational type of $\biq{G}{H}$ due to its formality). Hence $\H^\even(p)$ is just a quotient map whence it is surjective. Note that in this case $\biq{G}{H}$ is positively elliptic and $\H^*(\biq{G}{H})$ is concentrated in even degrees (see \eqref{EQ:homotopy_Euler_biquotient} and \eqref{EQ:positive_elliptic_biquotients}).

\item For $c=1$, we have that either $N\neq 0$ and $k=0$, or $N=0$, $k=1$. In the latter subcase $\biq{G}{H}$ is formal, and cohomology splits as
\[
\H^*(\biq{G}{H})\cong \H^\even(\biq{G}{H})\oplus \H^\even(\biq{G}{H}) \cdot [x_1]
\]
Moreover, $\H^\even(\biq{G}{H})$ is a quotient of $\H^*(\B H)$.

For $k=0$, $N\neq 0$ the biquotient is non-formal, and we derive that $N=\H_1(\Lambda V',\dif)$.
It follows that
\begin{align*}
\H^*(\biq{G}{H})\cong \H_0(\biq{G}{H})\oplus \H_1(\biq{G}{H})
\end{align*}
and hence from \eqref{EQ:even_odd_upper_lower} we get
$$
\H_0(\biq{G}{H})=\H^\even (\biq{G}{H})\quad \text{ and }\quad \H_1(\biq{G}{H})=\H^\odd (\biq{G}{H})
$$
In particular, since $\H_0(\biq{G}{H})=\biq{\H^*(\B H)}{\H^*(\B G)} $ is just a  quotient of $\H^*(\B H)$, the morphism $\H^\even(p)$ again is surjective.

Note that it is not at all necessary to split into these two subcases; the latter arguments also apply to the formal case with
\begin{align*}
\H^\odd (\biq{G}{H})\cong \H_1(\biq{G}{H}) \cong \H^\even(\biq{G}{H}) \cdot [x_1]
\end{align*}
Nonetheless, we hope that the distinction serves to illustrate the arguments.

\item
If $c\geq 2$ and $\biq{G}{H}$ is formal, then the decomposition in \Cref{prop00} implies that $(\Lambda V',\dif)$ is positively elliptic and hence $k=c\geq 2$. Thus there are two generators $[x_1], [x_2]$ in the free algebra whose product is non-trivial of (positive lower) even degree and hence cannot lie in the image of $\H^\even(p)$.

\item
Suppose more generally that $c$ is even and $\geq 2$. Then so is $\dim \biq{G}{H}$, see \eqref{EQ:parity_dim_biquotient}.
Since $c\geq 2$, rational cohomology cannot be concentrated in even degrees \eqref{EQ:positive_elliptic_biquotients}, hence there is $0\neq \alpha\in \H^\odd(\biq{G}{H})$. Using that $c$ is even, by Poincar\'e duality there is $0\neq \beta \in \H^\odd(\biq{G}{H})$ with $\alpha\cdot \beta\neq 0$. Moreover, since $\alpha^2=0$, $\beta \not \in I(\alpha)$  with $I(\alpha)$ the ideal generated by $\alpha$. It follows that the volume form $\alpha\cdot \beta$ cannot lie in $\H_0(\biq{G}{H})$ and hence is not in the image of $\H^\even(p)$.

\item
It remains to see what happens for $c\geq 2$.
In the terminology from \Cref{prop01}, we have that $k+r\geq 2$. Thus we have to consider three (not mutually disjoint) cases. In all three cases we have to identify a non-trivial even-degree cohomology class not represented by $\Lambda V^\even$.
\begin{description}
\item[$k\geq 2$] Here such a class is given by $[x_1]\cdot [x_2]$.
\item[$k=r=1$] A generator of $\H_r^{d'}(\Lambda V',\dif)$ multiplied with $[x_1]$.
\item[$r\geq 2$] In this case, such a class is constructed in \Cref{prop01}.
\end{description}
Since $\im \H^*(\B H) \In \H_0(\Lambda V',\dif)$, none of the constructed classes ever lies in the image of $\H^*(\B H)$ under $\H^\even(p)$.

\end{itemize}

\smallskip
The direct proof of Theorem~\ref{THM:surjectivity-of-Kalpha-biquotients} follows by applying Proposition \ref{prop01}.

\begin{proof}[Proof of Theorem~\ref{THM:surjectivity-of-Kalpha-biquotients}]
  Let $\biq{G}{H}$ a biquotient of compact connected Lie groups. Let $(\Lambda V,\dif)$ be the standard model for $\biq{G}{H}$ given in \Cref{Sullivan-for-biquotient}, which is pure. From \Cref{REM:image_biquotient_map} we know that the image of the map \(\H^\even(p)\) equals \(\H_0(\Lambda V,d) = \biq{\H^*(\B H)}{\H^*(\B G)}\). Thus \(\H^\even(p)\) is surjective if and only if \(\H_0(\Lambda V,d)=\H^\even(\Lambda V,d)\). Since $(\Lambda V,\dif)$ is pure, \Cref{cor01} tells us that this is the case if and only if $\chi_\pi(\biq{G}{H})\leq 1$. The identity $\chi_\pi(\biq{G}{H})=\rank G-\rank H$ in \eqref{EQ:homotopy_Euler_biquotient} completes the proof.
\end{proof}

\begin{proof}[Proof of Theorem~\ref{THM:surjectivity of homogeneous bundles}]
Let $G/H$ be a homogeneous space as in \Cref{THM:surjectivity of homogeneous bundles}. By \Cref{surjectivity-of-forgetful-map-in-Kgroups} we know that \Cref{THM:surjectivity of homogeneous bundles} is equivalent to the following statement: ``The morphism \(u\colon \K^0_G(G/H) \to \K^0(G/H)\) is surjective if and only if $\rank G -\rank H\leq 1$.''
In \cite[Theorem~3.6]{GonZib1} it is shown that \(u\colon \K^0_G(G/H) \to \K^0(G/H)\) is surjective if $\rank G -\rank H\leq 1$ (recall that $\K^0_G(G/H)\cong\R(H)$).  So it remains to show that \(u\colon \K^0_G(G/H) \to \K^0(G/H)\) is not surjective if $\rank G -\rank H\geq 2$. We argue by contradiction. Suppose it is surjective for some homogeneous space as in the assertion of \Cref{THM:surjectivity of homogeneous bundles} with $\rank G -\rank H\geq 2$. Then its rationalization \(u\otimes\qq\colon \K^0_G(G/H)\otimes\qq \to \K^0(G/H)\otimes\qq\) is surjective as well, and so is the map \(\H^\even_G(G/H;\qq)\to \H^\even(G/H;\qq)\) by \Cref{surjectivity_of_Heven_iff_K0}. Yet, the latter is never surjective if $\rank G -\rank H\geq 2$ by \Cref{THM:half_equiv_formal_homogeneous}.
\end{proof}

\begin{proof}[Proof of \Cref{THM:double_stabilization_homogeneous}]
It is an immediate consequence of \Cref{THM:half_equiv_formal_homogeneous}, \Cref{surjectivity_of_Heven_iff_K0}, \Cref{rational-surjectivity-of-forgetful-map-in-Kgroups} and finally Part (1) of \cref{hop:non-negative-metrics}.
\end{proof}

%%%%%%%%%%%%%%%%%%%%%%%%%%%%%%%%%%%%%%%%%%%%%%%%

\subsubsection{Cohomogeneity one spaces}\label{seccohom}
We refer the reader to \Cref{sec:prelim_cohomo1} for the structure of cohomogeneity one manifolds. In particular, we described the double mapping cylinder decomposition these manifolds admit in Equation \eqref{EQ: cohomo 1 decomposition}, and the fact that they are determined by a group diagram $(G,H,K_+, K_-)$ with $K_\pm/H$ spheres. We will use the notation $F_\pm:=K_\pm/H$. The goal of this section is to prove the following result.

\begin{theo}\label{THM:surjectivity_of_forgetful_map_even_cohomology}
Let $M$ be a cohomogeneity one manifold with all groups in the associated diagram \((G,H,K_-,K_+)\) connected and with \(F_\pm\) both odd-dimensional spheres. Then the morphism $\H^\even(j)\colon \H_G^\even(M)\to \H^\even(M)$ is surjective if and only if \(\rank G - \rank K_\pm\leq 1\).

(Note that our assumption on $F_\pm$ implies that $\rk K_\pm - \rk H = 1$ and hence $\rk K_+=\rk K_-$.)
\end{theo}
Before going into the proof, let us make a number of comments on the context, the statement and its consequences.

The surjectivity of the map $\H^*(j)\colon \H_G^*(M)\to \H^*(M)$ in all degrees has been studied. It is shown in \cite[Corollary 1.3]{GM14:coho1} that a cohomogeneity one manifold with associated diagram \((G,H,K_-,K_+)\) and $G$ connected is equivariantly formal if and only if $\max \{\rank K_\pm\}=\rank G$.

In \Cref{THM:surjectivity_of_forgetful_map_even_cohomology} we solely focus on the case when both $F_\pm$ are odd-dimensional spheres, as this is the only case necessary for our geometric considerations (i.e. for \Cref{hop:non-negative-metrics}, where the $F_\pm$ are required to be circles). A generalization to arbitrary cohomogeneity one manifolds seems easily feasible but tedious and not necessary for our purposes.

\Cref{THM:surjectivity_of_forgetful_map_even_cohomology} is a crucial ingredient in the proofs of Theorems \ref{THM: big theorem} and \ref{THM:rational_stabilization_cohomo1}, which are given at the end of this section.

    %\item A cohomogeneity one action of a compact connected Lie group on a compact connected manifold is equivariantly formal if and only if the rank of at least one isotropy group is maximal, i.e.~$\max \{\rank K_\pm\}=\rank G$.
 %   \item If both $F_+=K_+/H$ and $F_-=K_-/H$ are odd-dimensional spheres, then
 %   \begin{align*}
 %   \H_G^*(M)\cong \H^*(\B H)[e_+,e_-]/(e_+\cdot e_-)
 %       \end{align*}
 %   with $\H^*(\B K_\pm)\cong \H^*(\B H)[e_\pm]$ and with $\deg e_\pm=\dim F_\pm+1$.

The proof of \Cref{THM:surjectivity_of_forgetful_map_even_cohomology} consists of two steps. First we compute a model for the cohomogeneity one spaces we are interested in (\Cref{props1s1}). This will then allow us in a second step to identify the morphism from $G$-equivariant cohomology to ordinary cohomology explicitly when deriving its surjectivity.

So let us start with the model. In the following result we slightly extend famous work by Grove and Halperin \cite{GH87} and explicitly compute a Sullivan model for the cohomogeneity one manifold in our special case. This will come in very handy to compare its cohomology to its $G$-equivariant cohomology.

\begin{prop}\label{props1s1}
Let $M$ be a cohomogeneity one manifold with all groups in the associated diagram \((G,H,K_-,K_+)\) connected and with $F_\pm$ both odd-dimensional spheres.
Then a pure Sullivan model
of $M$ is given by
\begin{align*}
\big(\H^*(\B H)[e_+,e_-]\otimes \H^*(G)\otimes \Lambda \langle n \rangle,\dif\big)
\end{align*}
with $\deg e_\pm=\dim F_\pm+1$, $\dif|_{\H^*(\B H)[e_+,e_-]}=0$, $\dif n=e_+\cdot e_-$ and the differential on $\H^*(G)$ induced as a derivation up to degree shift by $\H^*(\B G)\to \H^*(\B K_\pm)$ (and $\H^*(\B K_\pm)\cong \H^*(\B H)[e_\pm]$).
\end{prop}

\begin{rem}\label{rem:G-cohomology-of-M}
Note that one can directly derive that, for a manifold $M$ as in \Cref{props1s1}, there is an isomorphism
$$
\H_G^*(M)\cong \H^*(\B H)[e_+,e_-]/(e_+\cdot e_-)
$$
with $\H^*(\B K_\pm)\cong \H^*(\B H)[e_\pm]$ and with $\deg e_\pm=\dim F_\pm+1$. This isomorphism is shown to hold in \cite[Theorem~1.1(b), p.~2]{CGHM18} for a wider class of cohomogeneity one spaces.

Indeed, we can obtain a model of the Borel construction $M_G$ as follows: tensor our model of \(M\) from \Cref{props1s1} with \(\H^*(\B G)\), and use that the linear part of the differential
is an isomorphism from the algebra generators of $\H^*(G)$ to the ones of $\H^*(\B G)$. Consequently, up to isomorphism we can split off the contractible algebra $\H^*(\B G)\tilde\otimes \H^*(G)\simeq \qq$, where $\tilde\otimes$ denotes twisted tensor product. This yields the pure model of $M_G$:
$$
\left( \H^*(\B H)[e_+,e_-]\otimes \Lambda \langle n\rangle ,\dif \right)\simeq \left( \H^*(\B H)[e_+,e_-]/(e_+\cdot e_-) ,\dif \right)
$$
with $\dif n=e_1\cdot e_2$, which is formal and hence computes \(\H_G^*(M)\) as stated above.
\end{rem}

\begin{rem}
Observe that the models in \Cref{props1s1} and \Cref{rem:G-cohomology-of-M} for $M$ and $M_G$ are pure.
As we shall see in the proof of \Cref{THM:surjectivity_of_forgetful_map_even_cohomology} below, $\chi_\pi(M)=\rank G - \rank K_{\pm}$. Hence at this point one might also just draw on \Cref{PROP:sufficient_conditions_surjectivity_half_equivariant} to conclude one direction of \Cref{THM:surjectivity_of_forgetful_map_even_cohomology}, namely the surjecitivity part. %We prefer, however, to spend some few more lines on spelling out the arguments in Theorem~\ref{THM:surjectivity_of_forgetful_map_even_cohomology}.
\end{rem}

\begin{proof}[Proof of \Cref{props1s1}]
  From \cite[p.~445]{GH87} we see that a model of a double mapping cylinder is given by its induced algebraic double mapping cylinder.
  In our case, this is the algebraic mapping cylinder of the double mapping cylinder of
  \begin{align*}
    G/K_+ \leftarrow G/H \rightarrow G/K_-
  \end{align*}
  Using the standard Sullivan models of homogeneous spaces (of compact \emph{connected} Lie groups) as in \Cref{Sullivan-for-biquotient} this dualizes to
  \begin{align*}
    (\H^*(\B K_+)\otimes \H^*(G),\dif) \xto{\phi_+} (\H^*(\B H)\otimes \H^*(G) ,\dif) \xleftarrow{\phi_-}(\H^*(\B K_-)\otimes \H^*(G) ,\dif)
  \end{align*}
  Note that the $\phi_\pm$ induce the identity on $\H^*(G)$ and equal the natural morphisms $\H^*(\B K_\pm)\to \H^*(\B H)$ on the first factor (i.e. those induced by the inclusions \(H\hookrightarrow K_\pm\)).
  This means that a model for $M$ is given by
  \begin{align*}
    % &(\Lambda V,\dif)\\=
      &\{(x,y) \in(\H^*(\B K_+)\oplus \H^*(\B K_-)) \mid \phi_+(x)=\phi_-(y)\} \tilde\otimes \H^*(G)
    \\\simeq&\big(\big(\H^*(\B H)[e_+,e_-]/(e_+\cdot e_-)\big)\otimes \H^*(G),\dif\big)
  \end{align*}
   where, by construction, the differential on the first tensor factor is trivial. The differential on $\H^*(G)$ is the one induced as a derivation on generators of $\H^*(G)$ up to a degree shift by $+1$ by $\H^*(\B G)\to \H^*(\B K_\pm)$, see \cite[Theorem~3.50, p.~137]{FOT08}.
  Here we take profit of the surjectivity of the morphisms $\H^*(\B K_\pm)\to \H^*(\B H)$, since $F_\pm$ are odd-dimensional spheres \cite[Proposition 3.1]{GM14:coho1}. Moreover, $e_\pm$ denote the volume forms of the fibre spheres $F_\pm$ up to degree shift, i.e.~$\H^*(\B K_\pm)=\H^*(\B H)[e_\pm]$.

  We now find a quasi-isomorphic Sullivan model $(\Lambda V,\dif)$ for this model---which is already rather close to a Sullivan algebra---by introducing an additional generator \(n\) with $\dif n=e_+\cdot e_-$. (Indeed, we may replace the base algebra $\H^*(\B H)[e_+,e_-]/e_+\cdot e_-$ by a minimal model and apply \cite[Corollary, p.~199]{FHT01}.) This yields the model from the assertion. It is pure by construction.
\end{proof}

Now we can give the proof of \Cref{THM:surjectivity_of_forgetful_map_even_cohomology}.

\begin{proof}[Proof of \Cref{THM:surjectivity_of_forgetful_map_even_cohomology}]
Using the model $(\Lambda V,\dif)$ of $M$ constructed in \Cref{props1s1} and the model of $M_G$ discussed in \Cref{rem:G-cohomology-of-M}, we can write the morphism $\H_G^\even(M)\to \H^\even(M)$ as
\begin{align*}
&\H^\even\left( \H^*(\B H)[e_+,e_-]\otimes \Lambda \langle n\rangle ,\dif \right)
\\ \to &
\H^\even\big(\H^*(\B H)[e_+,e_-]\otimes \Lambda \langle n \rangle \otimes \H^*(G),\dif\big)
\end{align*}
Both models are pure. Consequently, we may decompose $\H^*(M)=\H(\Lambda V,\dif)$ using the lower grading
\eqref{EQ:lower_grading}.
By construction, $\H_0(\Lambda V,\dif)$ is the cohomology generated by $\Lambda V^\even=\H^*(\B H)[e_+,e_-]$ in $\H(\Lambda V,\dif)=\H^*(M)$. There is a commutative diagram
\begin{align*}\footnotesize
\xymatrix{\H_0(\Lambda V,\dif) \ar@{^{(}->}[r]& \H^\even(\Lambda V,\dif)\ar@{=}[d]\\
\H^\even\left( \H^*(\B H)[e_+,e_-]\otimes \Lambda \langle n\rangle ,\dif \right)\ar@{->>}[u]\ar[r] & \H^\even\big(\H^*(\B H)[e_+,e_-]\otimes \Lambda \langle n \rangle \otimes \H^*(G),\dif\big) \ar@{=}[d]\\
\H_G^\even(M) \ar@{=}[u]\ar[r] &\H^\even(M)
}
\end{align*}
It follows that the surjectivity of \(\H_G^\even(M) \to \H^\even(M)\) is equivalent to the equality $\H_0(\Lambda V,\dif) = \H^\even(\Lambda V,\dif)$.

By \Cref{cor01} we know that $\H_0(\Lambda V,\dif) = \H^\even(\Lambda V,\dif)$ if and only if \( \chi_\pi(\Lambda V,\dif) \leq 1\). By computing the relevant homotopy Euler characteristic, we find that this is precisely the inequality that appears in the Theorem:
\begin{align*}
\chi_\pi(M)&=\chi_\pi(\Lambda V,\dif)=\dim V^\odd-\dim V^\even\\&=(\rk G+1)- (\rk K_\pm+1)=\rk G-\rk K_\pm
\end{align*}
Indeed, the $+1$ summand in $V^\odd$ results from the extra generator $n$ encoding the relation $e_+\cdot e_-=0$, the $+1$ summand in $V^\even$ is due to $\dim V^\even$ being equal to the number of algebra generators of $\H^*(\B H)[e_+,e_-]$ (which equals $\dim \H^*(\B K_\pm)+1$, where dimension is Krull dimension).
\end{proof}

We are finally ready to prove  \Cref{THM:rational_stabilization_cohomo1} and Part~(\ref{THM: rational converse soul for cohomo 1}) of \Cref{THM: big theorem}.

\begin{proof}[Proof of \Cref{THM:rational_stabilization_cohomo1}]
  For a cohomogeneity one manifold as in \Cref{THM:rational_stabilization_cohomo1} the induced morphism $\H_G^\even(M;\qq)\to \H^\even(M;\qq)$ is surjective by \Cref{THM:surjectivity_of_forgetful_map_even_cohomology}, implying the surjectivity of the map \(\K^0_G(M)\otimes\qq \to \K^0(M)\otimes\qq\) by \Cref{surjectivity_of_Heven_iff_K0}. The theorem follows by  \Cref{rational-surjectivity-of-forgetful-map-in-Kgroups}.
\end{proof}

\begin{proof}[Proof of Part~(\ref{THM: rational converse soul for cohomo 1}) of \Cref{THM: big theorem}]
  Let $M$ be a cohomogeneity one manifold as in Part~(\ref{THM: rational converse soul for cohomo 1}) of \Cref{THM: big theorem}, with associated diagram \((G,H,K_-,K_+)\), and let $E$ be an arbitrary real vector bundle over $M$. By \Cref{THM:rational_stabilization_cohomo1} there exist integers \(q,k\) such that the bundle \(qE\oplus \rr^k\) is isomorphic to (the underlying non-equivariant vector bundle of) a $G$-vector bundle $F$. The assumption \(K_\pm/H \cong S^1\) allows us to apply Part (2) of \cref{hop:non-negative-metrics}, thus $F$ carries a non-negatively curved metric. Its pullback via the isomorphism \(qE\oplus \rr^n= F\) yields the desired metric on \(qE\oplus \rr^k\), which is canonically identified with \(qE\times\rr^k\).
\end{proof}

\begin{rem}\label{REM:rational_Ktheory_circle}
In the case where $M$ is a $G$-manifold with orbit space a circle and $G$ connected, it is known that $M$ is equivariantly formal if and only if $\rank G = \rank H$, where $H$ denotes the principal isotropy group \cite[Corollary 1.3]{GM14:coho1}. In particular $\rank G = \rank H$ implies surjectivity of $\H_G^\even(M;\qq)\to \H^\even(M;\qq)$. Thus \Cref{surjectivity_of_Heven_iff_K0} and \Cref{rational-surjectivity-of-forgetful-map-in-Kgroups} can be applied to this case as well to obtain the same conclusion as in Part~(\ref{THM: rational converse soul for cohomo 1}) of \Cref{THM: big theorem}.
\end{rem}

%%%%%%%%%%%%%%%%%%%%%%%%%%%%%%%%%% Bibliography %%%%%%%%%%%%%%%%%%%%%%%%%%%%%%%%%%%

\bibliographystyle{amsalpha}

\begin{thebibliography}{GVWZ06}

\bibitem[AB15]{AB:Lie}
{Alexandrino, Marcos M.} and {Bettiol, Renato G.}, \emph{Lie groups and
  geometric aspects of isometric actions}, Springer, Cham, 2015.

\bibitem[{Ada}71]{Adams:Brown}
{Adams, J. F.}, \emph{A variant of e. h. brown's representability theorem},
  Topology \textbf{10} (1971), 185--198.

\bibitem[AH61]{AtiyahHirzebruch}
{Atiyah, M. F.} and {Hirzebruch, F.}, \emph{Vector bundles and homogeneous
  spaces}, 1961, pp.~7--38.

\bibitem[AHJM88]{AHJM}
{Adams, J. F.}, {Haeberly, J.-P.}, {Jackowski, S.}, and {May, J. P.}, \emph{A
  generalization of the {A}tiyah-{S}egal completion theorem}, Topology
  \textbf{27} (1988), no.~1, 1--6.

\bibitem[Ama13]{Ama13}
Manuel Amann, \emph{Non-formal homogeneous spaces}, Math. Z. \textbf{274}
  (2013), no.~3-4, 1299--1325. \MR{3078268}

\bibitem[AP97]{AlekPod:compact}
D.~V. Alekseevsky and F.~Podest\`a, \emph{Compact cohomogeneity one
  {R}iemannian manifolds of positive {E}uler characteristic and quaternionic
  {K}\"{a}hler manifolds}, Geometry, topology and physics ({C}ampinas, 1996),
  de Gruyter, Berlin, 1997, pp.~1--33. \MR{1605271}

\bibitem[AS69]{AtiyahSegal}
{Atiyah, M. F.} and {Segal, G. B.}, \emph{Equivariant {K}-theory and
  completion}, J. Differential Geometry \textbf{3} (1969), 1--18.

\bibitem[{Ati}67]{Atiyah:K-theory}
{Atiyah, M. F.}, \emph{$k$-theory}, Lecture notes by D. W. Anderson, W. A.
  Benjamin, Inc., New York-Amsterdam, 1967.

\bibitem[AY72]{ArakiYosimura}
{Araki, Sh\^{o}r\^{o}} and {Yosimura, Zen-ichi}, \emph{A spectral sequence
  associated with a cohomology theory of infinite $cw$-complexes}, Osaka Math.
  J. \textbf{9} (1972), 351--365.

\bibitem[BH58]{BoHi58}
A.~Borel and F.~Hirzebruch, \emph{Characteristic classes and homogeneous
  spaces. {I}}, Amer. J. Math. \textbf{80} (1958), 458--538. \MR{102800}

\bibitem[BK01]{BelKap01}
Igor Belegradek and Vitali Kapovitch, \emph{Topological obstructions to
  nonnegative curvature}, Math. Ann. \textbf{320} (2001), no.~1, 167--190.
  \MR{1835067}

\bibitem[BK03]{BelKap03}
\bysame, \emph{Obstructions to nonnegative curvature and rational homotopy
  theory}, J. Amer. Math. Soc. \textbf{16} (2003), no.~2, 259--284.
  \MR{1949160}

\bibitem[{Bor}60]{Borel:Transformation}
{Borel, Armand}, \emph{Seminar on transformation groups}, With contributions by
  G. Bredon, E. E. Floyd, D. Montgomery, R. Palais. Annals of Mathematics
  Studies, No. 46, Princeton University Press, Princeton, N.J., 1960.

\bibitem[Bre72]{Bre72}
Glen~E. Bredon, \emph{Introduction to compact transformation groups}, Academic
  Press, New York, 1972, Pure and Applied Mathematics, Vol. 46. \MR{MR0413144
  (54 \#1265)}

\bibitem[Bre97]{Bre97}
\bysame, \emph{Topology and geometry}, Graduate Texts in Mathematics, vol. 139,
  Springer-Verlag, New York, 1997, Corrected third printing of the 1993
  original. \MR{1700700 (2000b:55001)}

\bibitem[{Car}22]{Carlson:coho1}
{Carlson, Jeffrey D.}, \emph{The k-theory of cohomogeneity-one actions},
  arXiv:1805.00502v4 (2022).

\bibitem[CF18]{CarlsonFok:isotropy}
{Carlson, Jeffrey D.} and {Fok, Chi-Kwong}, \emph{Equivariant formality of
  isotropy actions}, J. Lond. Math. Soc. (2) \textbf{97} (2018), no.~3,
  470--494.

\bibitem[CG72]{CheGro72}
Jeff Cheeger and Detlef Gromoll, \emph{On the structure of complete manifolds
  of nonnegative curvature}, Ann. of Math. (2) \textbf{96} (1972), 413--443.
  \MR{0309010}

\bibitem[CGHM19]{CGHM18}
Jeffrey~D. Carlson, Oliver Goertsches, Chen He, and Augustin-Liviu Mare,
  \emph{The equivariant cohomology ring of a cohomogeneity-one action}, Geom.
  Dedicata \textbf{203} (2019), 205--223. \MR{4027592}

\bibitem[DeV17]{Dev17}
Jason DeVito, \emph{The classification of compact simply connected biquotients
  in dimension 6 and 7}, Math. Ann. \textbf{368} (2017), no.~3-4, 1493--1541.
  \MR{3673662}

\bibitem[DK20]{DeVitoKennard:coho1}
Jason DeVito and Lee Kennard, \emph{Cohomogeneity one manifolds with singly
  generated rational cohomology}, Doc. Math. \textbf{25} (2020), 1835--1863.
  \MR{4184453}

\bibitem[FHT01]{FHT01}
Yves F{\'e}lix, Stephen Halperin, and Jean-Claude Thomas, \emph{Rational
  homotopy theory}, Graduate Texts in Mathematics, vol. 205, Springer-Verlag,
  New York, 2001. \MR{MR1802847 (2002d:55014)}

\bibitem[{Fok}19]{Fok:EquivariantFormality}
{Fok, Chi-Kwong}, \emph{Equivariant formality in $k$-theory}, New York J. Math.
  \textbf{25} (2019), 315--327.

\bibitem[FOT08]{FOT08}
Yves F{\'e}lix, John Oprea, and Daniel Tanr{\'e}, \emph{Algebraic models in
  geometry}, Oxford Graduate Texts in Mathematics, vol.~17, Oxford University
  Press, Oxford, 2008. \MR{MR2403898 (2009a:55006)}

\bibitem[Fra13]{Frank:cohomogeneity}
Philipp Frank, \emph{Cohomogeneity one manifolds with positive {E}uler
  characteristic}, Transform. Groups \textbf{18} (2013), no.~3, 639--684.
  \MR{3084330}

\bibitem[GA17]{Gon17}
David Gonz\'{a}lez-\'{A}lvaro, \emph{Nonnegative curvature on stable bundles
  over compact rank one symmetric spaces}, Adv. Math. \textbf{307} (2017),
  53--71. \MR{3590513}

\bibitem[GAZ20]{GonZib2}
David Gonz\'{a}lez-\'{A}lvaro and Marcus Zibrowius, \emph{Open manifolds with
  non-homeomorphic positively curved souls}, Math. Proc. Cambridge Philos. Soc.
  \textbf{169} (2020), no.~2, 357--376. \MR{4138927}

\bibitem[GAZ21]{GonZib1}
\bysame, \emph{The stable converse soul question for positively curved
  homogeneous spaces}, J. Differential Geom. \textbf{119} (2021), no.~2,
  261--307. \MR{4318296}

\bibitem[GGZ18]{GaGaZa18}
Fernando Galaz-Garc\'{\i}a and Masoumeh Zarei, \emph{Cohomogeneity one
  topological manifolds revisited}, Math. Z. \textbf{288} (2018), no.~3-4,
  829--853. \MR{3778980}

\bibitem[GH87]{GH87}
Karsten Grove and Stephen Halperin, \emph{Dupin hypersurfaces, group actions
  and the double mapping cylinder}, J. Differential Geom. \textbf{26} (1987),
  no.~3, 429--459. \MR{MR910016 (89h:53113)}

\bibitem[GKS17]{GoKeSh17}
Sebastian Goette, Martin Kerin, and Krishnan Shankar, \emph{Highly connected
  7-manifolds and non-negative sectional curvature}, To appear in Ann. of Math.
  (2017).

\bibitem[GM14]{GM14:coho1}
Oliver Goertsches and Augustin-Liviu Mare, \emph{Equivariant cohomology of
  cohomogeneity one actions}, Topology Appl. \textbf{167} (2014), 36--52.
  \MR{3193423}

\bibitem[Gra67]{Gr67}
Alfred Gray, \emph{Pseudo-{R}iemannian almost product manifolds and
  submersions}, J. Math. Mech. \textbf{16} (1967), 715--737. \MR{0205184}

\bibitem[GVWZ06]{GrVeWiZi06}
Karsten Grove, Luigi Verdiani, Burkhard Wilking, and Wolfgang Ziller,
  \emph{Non-negative curvature obstructions in cohomogeneity one and the
  {K}ervaire spheres}, Ann. Sc. Norm. Super. Pisa Cl. Sci. (5) \textbf{5}
  (2006), no.~2, 159--170. \MR{2244696}

\bibitem[GZ00]{GroveZiller:Milnor}
{Grove, Karsten} and {Ziller, Wolfgang}, \emph{Curvature and symmetry of
  {M}ilnor spheres}, Ann. of Math. (2) \textbf{152} (2000), no.~1, 331--367.

\bibitem[GZ11]{GroveZiller:Lifting}
\bysame, \emph{Lifting group actions and nonnegative curvature}, Trans. Amer.
  Math. Soc. \textbf{363} (2011), no.~6, 2865--2890.

\bibitem[{Hat}17]{Hatcher:VBKT}
{Hatcher, Allen}, \emph{Vector bundles and {K}-theory}, 11/2017, Version~2.2,
  available from \protect \url
  {https://pi.math.cornell.edu/~hatcher/VBKT/VBpage.html}.

\bibitem[Hoe10]{Hoel2010}
Corey~A. Hoelscher, \emph{Classification of cohomogeneity one manifolds in low
  dimensions}, Pacific J. Math. \textbf{246} (2010), no.~1, 129--185.
  \MR{2645881}

\bibitem[{Hus}94]{Husemoller}
{Husemoller, Dale}, \emph{Fibre bundles}, 3 ed., Graduate Texts in Mathematics,
  vol.~20, Springer-Verlag, New York, 1994.

\bibitem[{Ill}83]{Illman}
{Illman, S\"{o}ren}, \emph{The equivariant triangulation theorem for actions of
  compact lie groups}, Math. Ann. \textbf{262} (1983), no.~4, 487--501.

\bibitem[{Ill}90]{Illman:restricting}
\bysame, \emph{Restricting the transformation group in equivariant cw
  complexes}, Osaka J. Math. \textbf{27} (1990), no.~1, 191--206.

\bibitem[Kap05]{Kap}
Vitali Kapovitch, \emph{A note on rational homotopy of biquotients}, preprint
  (ca. 2005).

\bibitem[{Kar}78]{Karoubi:K-theory}
{Karoubi, Max}, \emph{$k$-theory}, Springer-Verlag, Berlin-New York, 1978, An
  introduction; Grundlehren der Mathematischen Wissenschaften, Band 226.

\bibitem[Las82]{Lashof:equivariant}
R.~K. Lashof, \emph{Equivariant bundles}, Illinois J. Math. \textbf{26} (1982),
  no.~2, 257--271. \MR{650393}

\bibitem[{Lü}05]{Lueck:Classifying}
{Lück, Wolfgang}, \emph{Survey on classifying spaces for families of
  subgroups}, 2005, pp.~269--322.

\bibitem[{Mat}71]{Matumoto:Whitehead}
{Matumoto, Takao}, \emph{On $g$-${\protect \rm cw}$ complexes and a theorem of
  j. h. c. whitehead}, J. Fac. Sci. Univ. Tokyo Sect. IA Math. \textbf{18}
  (1971), 363--374.

\bibitem[{Mat}80]{Matsumura1980}
{Matsumura, Hideyuki}, \emph{Commutative algebra}, 2 ed., Mathematics Lecture
  Note Series, vol.~56, Benjamin/Cummings Publishing Co., Inc., Reading, Mass.,
  1980.

\bibitem[{May}96]{May96}
{May, J. P.}, \emph{Equivariant homotopy and cohomology theory}, CBMS Regional
  Conference Series in Mathematics, vol.~91, Published for the Conference Board
  of the Mathematical Sciences, Washington, DC; by the American Mathematical
  Society, Providence, RI, 1996, With contributions by M. Cole, G.
  Comeza\~{n}a, S. Costenoble, A. D. Elmendorf, J. P. C. Greenlees, L. G.
  Lewis, Jr., R. J. Piacenza, G. Triantafillou, and S. Waner.

\bibitem[{May}99]{May:Concise}
\bysame, \emph{A concise course in algebraic topology}, Chicago Lectures in
  Mathematics, University of Chicago Press, Chicago, IL, 1999.

\bibitem[MM82]{MayMcClure:reduction}
{May, J. P.} and {McClure, J. E.}, \emph{A reduction of the {S}egal
  conjecture}, 1982, pp.~209--222.

\bibitem[Mos57]{Mo57}
Paul~S. Mostert, \emph{On a compact {L}ie group acting on a manifold}, Ann. of
  Math. (2) \textbf{65} (1957), 447--455. \MR{85460}

\bibitem[{Nee}97]{Neeman:Brown}
{Neeman, Amnon}, \emph{On a theorem of {B}rown and {A}dams}, Topology
  \textbf{36} (1997), no.~3, 619--645.

\bibitem[O'N66]{On66}
Barrett O'Neill, \emph{The fundamental equations of a submersion}, Michigan
  Math. J. \textbf{13} (1966), 459--469. \MR{200865}

\bibitem[OW94]{OzWa94}
Murad \"{O}zaydin and Gerard Walschap, \emph{Vector bundles with no soul},
  Proc. Amer. Math. Soc. \textbf{120} (1994), no.~2, 565--567. \MR{1162091}

\bibitem[{Par}08]{Park:K-theory}
{Park, Efton}, \emph{Complex topological $k$-theory}, Cambridge Studies in
  Advanced Mathematics, vol. 111, Cambridge University Press, Cambridge, 2008.

\bibitem[Pit72]{Pit72}
Harsh~V. Pittie, \emph{Homogeneous vector bundles on homogeneous spaces},
  Topology \textbf{11} (1972), 199--203. \MR{0290402}

\bibitem[Rig78]{Rig78}
A.~Rigas, \emph{Geodesic spheres as generators of the homotopy groups of
  $\mathrm{O}$, $b\mathrm{O}$}, J. Differential Geom. \textbf{13} (1978),
  no.~4, 527--545 (1979). \MR{570216}

\bibitem[{Seg}68]{Segal:Equivariant}
{Segal, Graeme}, \emph{Equivariant {K}-theory}, Inst. Hautes \'{E}tudes Sci.
  Publ. Math. \textbf{34} (1968), 129--151.

\bibitem[Sin93]{Sin93}
W.~Singhof, \emph{On the topology of double coset manifolds}, Math. Ann.
  \textbf{297} (1993), no.~1, 133--146. \MR{1238411 (94k:57054)}

\bibitem[{Ste}75]{Steinberg:Pittie}
{Steinberg, Robert}, \emph{On a theorem of {Pittie}}, Topology \textbf{14}
  (1975), 173--177.

\bibitem[{Tot}02]{Totaro:Cheeger}
{Totaro, Burt}, \emph{Cheeger manifolds and the classification of biquotients},
  J. Differential Geom. \textbf{61} (2002), no.~3, 397--451.

\bibitem[{Wan}80]{Waner:Milnor}
{Waner, Stefan}, \emph{Equivariant homotopy theory and milnor's theorem},
  Trans. Amer. Math. Soc. \textbf{258} (1980), no.~2, 351--368.

\end{thebibliography}
\def\cprime{$'$}
\providecommand{\bysame}{\leavevmode\hbox to3em{\hrulefill}\thinspace}
\providecommand{\MR}{\relax\ifhmode\unskip\space\fi MR }
\providecommand{\MRhref}[2]{%
  \href{http://www.ams.org/mathscinet-getitem?mr=#1}{#2}
}
\providecommand{\href}[2]{#2}

\vfill

\noindent
\parbox[t]{0.3\textwidth}{\raggedright
\tiny \noindent \textsc
{Manuel Amann} \\
\textsc{Institut f\"ur Mathematik}\\
\textsc{Universit\"at Augsburg}\\
\textsc{Universit\"atsstra{\ss}e 14}\\
\textsc{86159 Augsburg}\\
\textsc{Germany}
}
\hfill\parbox[t]{0.3\textwidth}{
\tiny \noindent \textsc
{Marcus Zibrowius} \\
\textsc{Mathematisches Institut}\\
\textsc{Heinrich-Heine-Universit\"at}\\
\textsc{Universit\"atsstra{\ss}e 1}\\
\textsc{40225 D\"usseldorf}\\
\textsc{Germany}
}
\hfill\parbox[t]{0.3\textwidth}{\raggedright
\tiny \noindent \textsc{David Gonz\'alez-\'Alvaro}\\
\textsc{ETSI de Caminos, Canales y Puertos}\\
\textsc{Universidad Politécnica de Madrid}\\
\textsc{28040 Madrid}\\
\textsc{Spain}
}
\end{document}